\documentclass[12pt]{amsart}
\usepackage{amsmath}	
\usepackage{amssymb}
\usepackage{array}

\usepackage{accents}
\usepackage{graphicx}
\usepackage{tikz}
\usetikzlibrary{calc,intersections,arrows,automata}

\newtheorem{theorem}{Theorem}[section]
\newtheorem{lemma}[theorem]{Lemma}
\newtheorem{proposition}[theorem]{Proposition}
\newtheorem{corollary}[theorem]{Corollary}

\theoremstyle{remark}
\newtheorem*{remark}{Remark}
\newtheorem*{acknowlegments}{Acknowlegments}

\theoremstyle{definition}
\newtheorem{definition}[theorem]{Definition}

\numberwithin{equation}{section}

\newcommand{\alphabar}{\overline{\alpha}}

\newcommand{\bN}{\mathbb{N}}
\newcommand{\bQ}{\mathbb{Q}}
\newcommand{\bR}{\mathbb{R}}
\newcommand{\bZ}{\mathbb{Z}}
\newcommand{\cA}{{\mathcal{A}}}
\newcommand{\cB}{{\mathcal{B}}}
\newcommand{\cC}{{\mathcal{C}}}
\newcommand{\cD}{{\mathcal{D}}}
\newcommand{\cE}{{\mathcal{E}}}
\newcommand{\cF}{{\mathcal{F}}}
\newcommand{\cFbar}{\overline{\cF}}
\newcommand{\cM}{{\mathcal{M}}}
\newcommand{\cN}{{\mathcal{N}}}
\newcommand{\cO}{{\mathcal{O}}}
\newcommand{\cR}{{\mathcal{R}}}
\newcommand{\cS}{{\mathcal{S}}}
\newcommand{\cT}{{\mathcal{T}}}
\newcommand{\cV}{\mathcal{V}}
\newcommand{\cVbar}{\overline{\cV}}
\newcommand{\disp}{\displaystyle}
\newcommand{\GL}{\mathrm{GL}}
\newcommand{\iotabar}{\overline{\iota}}
\newcommand{\lambdahat}{\widehat{\lambda}}
\newcommand{\matrice}[4]{\begin{pmatrix}#1&#2\\#3&#4\end{pmatrix}}
\newcommand{\Mat}{\mathrm{Mat}}
\newcommand{\norm}[1]{\|\hspace*{1pt}#1\hspace*{1pt}\|}

\newcommand{\phitop}{{\bar{\varphi}}}
\newcommand{\prefixes}{\left[\epsilon,w_\infty\right[}
\newcommand{\prefixesp}{\left]\epsilon,w_\infty\right[}
\newcommand{\SL}{\mathrm{SL}}

\newcommand{\trace}{\mathrm{trace}}
\newcommand{\trM}{{\hspace*{2pt}{}^t\hspace*{-2pt}M}}
\newcommand{\tux}{{ \ux^{\mathrm{alg}} }}
\newcommand{\tw}{\widetilde{w}}
\newcommand{\ue}{\mathbf{e}}
\newcommand{\uf}{\mathbf{f}}
\newcommand{\uL}{\mathbf{L}}
\newcommand{\um}{\mathbf{m}}
\newcommand{\up}{\mathbf{p}}
\newcommand{\uP}{{\mathbf{P}}}
\newcommand{\uq}{\mathbf{q}}
\newcommand{\ur}{{\mathbf{r}}}
\newcommand{\us}{{\mathbf{s}}}
\newcommand{\ux}{\mathbf{x}}
\newcommand{\uy}{\mathbf{y}}
\newcommand{\uz}{\mathbf{z}}
\newcommand{\udeux}{{\mathbf{2}}}
\newcommand{\uun}{{\mathbf{1}}}
\newcommand{\uxi}{{\boldsymbol{\xi}}}

\newcommand{\rien}[1]{}

\newcommand{\et}{\quad\mbox{and}\quad}
\newcommand{\ou}{\quad\mbox{or}\quad}

\newcommand{\trap}{\mathrm{Trap}}
\newcommand{\layer}{\mathrm{Layer}}
\newcommand{\cell}{\mathrm{Cell}}

\begin{document}

\baselineskip=14.7pt 

\title[Multi-parametric geometry of numbers]
{Extremal numbers and\\ multi-parametric geometry of numbers}
\author{Damien ROY}
\address{
   D\'epartement de Math\'ematiques\\
   Universit\'e d'Ottawa\\
   150 Louis Pasteur\\
   Ottawa, Ontario K1N 6N5, Canada}
\email{droy@uottawa.ca}


\dedicatory{In memory of Bertrand Russell, mathematician and philosopher,\\
for his commitment to peace.}


\begin{abstract}
We study weighted simultaneous rational approximation to points of the 
form $(1,\xi,\xi^2)$, for a class of extremal real numbers $\xi$, within the 
framework of multi-parametric geometry of numbers.
\end{abstract}

\maketitle


%
%

\section{Introduction}
 \label{sec:intro}
Fix an integer $n\ge 2$.  For each $A=(a_{i,j})\in\GL_n(\bR)$ and 
each $\uq=(q_1,\dots,q_n)\in\bR^n$, we denote by $\cC_A(\uq)$ 
the parallelepiped of $\bR^n$ made of the points
$\ux=(x_1,\dots,x_n)\in\bR^n$ satisfying
\begin{equation}
 \label{intro:eq:parallelepiped}
 |a_{i,1}x_1+a_{i,2}x_2+\cdots+a_{i,n}x_n|\ \le \exp(-q_i)
 \quad\text{for $i=1,\dots,n$.}
\end{equation} 
For each $j=1,\dots,n$, we also denote by $L_{A,j}(\uq)$ 
the logarithm of its $j$-th minimum with respect to $\bZ^n$, 
namely the smallest $t\in\bR$ such that the product $e^t\cC_A(\uq)$ 
defined by 
\[
 |a_{i,1}x_1+a_{i,2}x_2+\cdots+a_{i,n}x_n| \le \exp(t-q_i) 
 \quad (1\le i\le n)
\] 
contains at least $j$ linearly independent points $\ux=(x_1,\dots,x_n)$ of $\bZ^n$.
The map 
\[
  \begin{array}{rcl}
   \uL_A\,:\, \bR^n &\longrightarrow &\bR^n\\
   \uq &\longmapsto &(L_{A,1}(\uq),L_{A,2}(\uq),\dots, L_{A,n}(\uq)).
  \end{array}
\]
carries much information about Diophantine approximation 
to the matrix $A$ and little is lost in estimating $\uL_A$ up to 
bounded error on $\bR^n$.   

Geometry of numbers imposes constraints on the components of $\uL_A$,
as shown in \cite{S2020, SS2009, SS2013} in one parameter setting, and in
\cite{G2022} in the general case.
For example, since the logarithm of the volume of $\cC_A(\uq)$ is
$-(q_1+\cdots+q_n)+\cO_A(1)$, Minkowski's second convex body theorem
gives
\begin{equation}
 \label{intro:eq:sumLA}
 L_{A,1}(\uq)+\cdots+L_{A,n}(\uq) = q_1+\cdots+q_n +\cO_A(1),
\end{equation}
where $\cO_A(1)$ stands for a function of $\uq$ whose 
absolute value is bounded above by a constant depending only of $A$.
The main open problem is whether or not, for given $n\ge 3$, this and other 
such conditions 
suffice to characterize the set of all maps $\uL_A$ with $A\in\GL_n(\bR)$
modulo the additive group of bounded functions from $\bR^n$ to $\bR^n$
(cf.~\cite{R2015b}).   Even a characterization of the 
maps $\uL_A$ modulo the additive group of functions 
$\uf\colon\bR^n\to\bR^n$ with $\|\uf(\uq)\|/\|\uq\|\to 0$ for 
$\|\uq\|\to\infty$, using say the maximum norm,  would be useful 
in questions related to spectra of exponents of Diophantine approximation.
We do not address this problem here.  
Our goal instead is to estimate the map $\uL_A$ for 
a specific type of matrices $A\in\GL_3(\bR)$.

We first restrict to matrices of the form
 \[
  A = \begin{pmatrix} 1 &0 &0\\ \xi_1 &-1 &0\\ \xi_2 &0 &-1 \end{pmatrix}
     \in \GL_3(\bR)
\]
attached to points $\uxi=(1,\xi_1,\xi_2)\in\bR^3$.   For convenience, we adapt 
the notation as follows.   For each $\uq=(q_1,q_2)\in \bR^2$, we set
$\cC_\uxi(\uq)=\cC_A(0,q_1,q_2)$, and, for each $j=1,2,3$, we 
set $L_{\uxi,j}(\uq)=L_{A,j}(0,q_1,q_2)$, that is the smallest $t\in\bR$ for 
which the inequalities
\begin{equation}
 \label{intro:eq:t}
 |x_0|\le \exp(t), \quad 
 |x_0\xi_1-x_1| \le \exp(t-q_1), \quad  
 |x_0\xi_2-x_2| \le \exp(t-q_2)
\end{equation}
admit at least $j$ linearly independent solutions $\ux=(x_0,x_1,x_2)$
in $\bZ^3$.  Finally, we define 
\begin{equation}
 \label{intro:eq:Lxi}
  \begin{array}{rcl}
   \uL_\uxi\,:\, \bR^2 &\longrightarrow &\bR^3\\
   \uq &\longmapsto &(L_{\uxi,1}(\uq),L_{\uxi,2}(\uq),L_{\uxi,3}(\uq)).
  \end{array}
\end{equation}
No information is lost in the process as we can recover $\uL_A$ from
$\uL_\uxi$ using  
\[
 L_{A,j}(q_0,q_1,q_2) = q_0 +L_{\uxi,j}(q_1-q_0,q_2-q_0) \quad (1\le j\le 3) 
\]
for any $(q_0,q_1,q_2)\in\bR^3$.  Moreover, the estimate 
\eqref{intro:eq:sumLA} becomes
\begin{equation}
 \label{intro:eq:sumLxi}
 L_{\uxi,1}(\uq)+L_{\uxi,2}(\uq)+L_{\uxi,3}(\uq) = q_1+q_2 +\cO_\uxi(1)
\end{equation}
for each $\uq=(q_1,q_2)\in\bR^2$.
As for the maps $\uL_A$ with $A\in\GL_3(\bR)$, one may ask for a 
characterization of the set of all maps $\uL_\uxi\colon\bR^2\to\bR^3$
with $\uxi=(1,\xi_1,\xi_2)\in\bR^3$ modulo bounded functions on
$\bR^2$.  In the next section, we show that this can easily be done 
outside of the critical sector 
\begin{equation}
 \label{intro:eq:D}
 \cD=\{(q_1,q_2)\in\bR^2\,;\, 0\le q_1/2\le q_2\le 2q_1\}
\end{equation}
if we restrict to points $\uxi$ for which $\xi_1$ and $\xi_2$ are badly 
approximable.

In this paper, we further restrict to points $\uxi=(1,\xi,\xi^2)$ where $\xi$ 
belongs to a countably infinite set of extremal real numbers defined in
\cite[Theorem~3.1]{R2003} and shown there to have the strongest possible 
measure of approximation by cubic algebraic integers 
(cf.\ \cite[Theorem~1]{DS1969}).  For each such $\uxi$, 
we construct an explicit approximation of $\uL_\uxi$  
with bounded difference on the set of points $(q_1,q_2)$ of $\bR^2$ with
$q_1\ge 0$ and $q_2\le q_1$, an angular sector which 
covers the lower half of $\cD$.   An outline of the main result
and of its proof is given in section~\ref{sec:main}.  As an application,
we use this to compute exponents of weighted rational approximation
to these points in section~\ref{sec:exp}.  
Numerical experiments, which we do not include, suggest 
a relatively chaotic behaviour for the functions $\uL_\uxi$ on the upper half 
of  $\cD$.

We believe that these are the first examples of points $\uxi$ in $\bR^3$
with $\bQ$-linearly independent coordinates, for which 
$\uL_\uxi$ is estimated up to bounded difference in an angular 
sector of $\bR^2$ with positive angle contained in $\cD$.

%
%

\section{Basic tool and preliminary observations}\
\label{sec:tool}
For each integer $n\ge 1$ and each point $\ux$ of $\bR^n$, we denote by 
$\norm{\ux}$ the maximum norm of $\ux$.  As in \cite[Section~1]{R2015b}, 
we set
\begin{equation}
\label{tool:eq:Delta}
 \Delta_3=\{(p_1,p_2,p_3)\in\bR^3\,;\, p_1\le p_2\le p_3\},
\end{equation}
and denote by $\Phi\colon\bR^3\to\Delta_3$ the continuous map 
which lists the coordinates of a point in non-decreasing order.   Clearly,
the functions $\uL_\uxi$ defined by \eqref{intro:eq:Lxi} take values in
$\Delta_3$ and so $\Phi\circ\uL_\uxi=\uL_\uxi$.  We also note that 
$\Phi$ is $1$-Lipschitz, namely that 
\begin{equation}
 \label{tool:eq:growthPhi}
 \norm{\Phi(\up)-\Phi(\up')}\le \norm{\up-\up'}
 \quad \text{for any $\up, \up'\in\bR^3$.}
\end{equation}

Fix a point $\uxi=(1,\xi_1,\xi_2)\in\bR^3$.  We define the \emph{trajectory}
of a non-zero integer point $\ux=(x_0,x_1,x_2)\in \bZ^3$ as the map
$L_\ux\colon\bR^2\to\bR$ whose value at a point $\uq=(q_1,q_2)\in\bR^2$
is the smallest $t\in\bR$ for which $\ux \in e^t\cC_\uxi(\uq)$ or, equivalently,
for which \eqref{intro:eq:t} holds. Thus,
\begin{equation}
 \label{tool:eq:Lx}
 L_\ux(\uq)  
  =\max\{\log|x_0|, q_1+\log|x_0\xi_1-x_1|, q_2+\log|x_0\xi_2-x_2|\},
\end{equation}
with the convention that we omit a term in the maximum when it involves
$\log 0=-\infty$.  Aside when this happens, the trajectory $L_\ux$ of $\ux$ 
is the maximum of three affine maps 
and Figure~\ref{intro:fig1} shows the angular sectors of $\bR^2$ where 
they realize this maximum.  

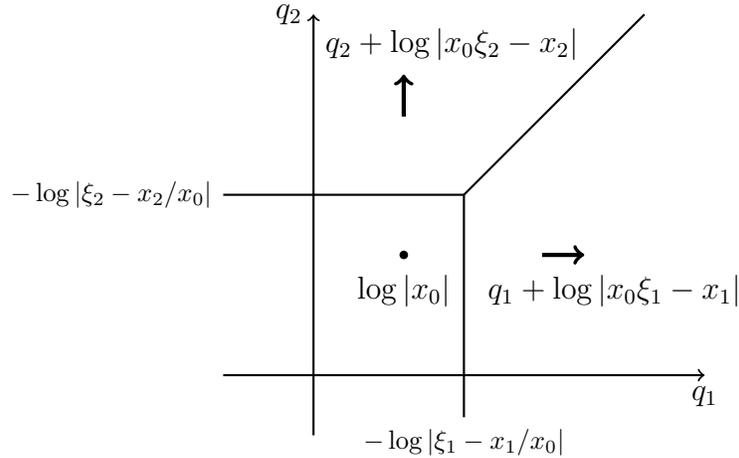
\begin{figure}[h]
\begin{tikzpicture}[scale=0.8]
       \draw[->, thick] (-2,0)--(6,0) node[below]{$q_1$};
       \draw[->, thick] (-0.5,-1)--(-0.5,6) node[left]{$q_2$};
       \draw[-, thick] (2,3) -- (-2,3)
                   node[left]{\footnotesize $-\log|\xi_2-x_2/x_0|$}; 
       \draw[-, thick] (2,3) -- (2, -0.7) 
                    node[below]{\footnotesize $-\log|\xi_1-x_1/x_0|$};
       \draw[-, thick] (2,3) -- (5, 6);
       \draw[->,ultra thick] (3.3,2)--(4,2); 
       \draw node at (4.5,1.4) {$q_1+\log|x_0\xi_1-x_1|$};
       \draw[->,ultra thick] (1,4.3)--(1,5);
       \draw node at (1.8,5.5) {$q_2+\log|x_0\xi_2-x_2|$};
       \node[draw,circle,inner sep=1pt,fill] at (1,2) {};
       \draw node at (1,1.4) {$\log|x_0|$};            
\end{tikzpicture}
\caption{The map $L_\ux$ in the generic case.}
\label{intro:fig1}
\end{figure}

It follows from \eqref{tool:eq:Lx} that  
\begin{equation}
 \label{tool:eq:growthLx}
 \left| L_\ux(\uq)-L_\ux(\uq')\right| \le \norm{\uq-\uq'}
 \quad \text{for any $\uq, \uq'\in\bR^2$.}
\end{equation}
We now show that $\uL_\uxi$ is also $1$-Lipschitz.

\begin{lemma}
\label{tool:lemma1} 
For any $\uq,\uq'\in\bR^2$, we have 
$\norm{\uL_\uxi(\uq)-\uL_\uxi(\uq')} \le \norm{\uq-\uq'}$.
\end{lemma}

\begin{proof}
Let $\uq,\uq'\in\bR^2$.  Choose linearly independent integer points
$\ux_1,\ux_2,\ux_3\in\bZ^3$ such that 
\[
 \uL_\uxi(\uq)=(L_{\ux_1}(\uq),L_{\ux_2}(\uq),L_{\ux_3}(\uq)).
\]
For each $j=1,2,3$, the points $\ux_1,\dots,\ux_j$ are linearly independent, 
and so, using \eqref{tool:eq:growthLx} for each of these points, we find that
\[
 L_{\uxi,j}(\uq')
   \le \max_{1\le i\le j} L_{\ux_i}(\uq')
   \le \max_{1\le i\le j} \big( L_{\ux_i}(\uq)+\|\uq-\uq'\| \big)
   = L_{\uxi,j}(\uq)+\|\uq-\uq'\|.
\]
The result follows as we may permute the roles of $\uq$ and $\uq'$ in the above estimates.
\end{proof}

The main tool for estimating $\uL_\uxi$ is the following
result, similar to \cite[Lemma~4.1]{R2015a}.

\begin{lemma}
\label{tool:lemma2}
Let $\ux_1,\ux_2,\ux_3\in\bZ^3$ be linearly independent integer points.  Let 
$\uq=(q_1,q_2)\in\bR^2$, $\up=(p_1,p_2,p_3)\in\bR^3$ and 
$\delta\ge 0$ such that 
\begin{equation}
\label{tool:lemma2:eq1}
 p_1+p_2+p_3=q_1+q_2 
 \et 
 L_{\ux_j}(\uq)\le p_j+\delta \quad \text{for $j=1,2,3$.}
\end{equation}
Then, we have $\big\|\uL_\uxi(\uq)-\Phi(\up)\big\| \le 5\delta +c_1$
for a constant $c_1$ depending only on $\uxi$.
\end{lemma}

\begin{proof}
Set $\up'=(L_{\ux_1}(\uq),L_{\ux_2}(\uq),L_{\ux_3}(\uq))$.  By definition of
$\uL_\uxi(\uq)$, the difference $\Phi(\up')-\uL_\uxi(\uq)$ has non-negative 
coordinates.  Thus, its norm is bounded above by the sum of its coordinates.
Using \eqref{intro:eq:sumLxi} and the first part of \eqref{tool:lemma2:eq1}, 
this gives
\begin{align*}
 \norm{\Phi(\up')-\uL_\uxi(\uq)}
  &\le \sum_{j=1}^3 L_{\ux_j}(\uq) - \sum_{j=1}^3 L_{\uxi,j}(\uq) \\
  &\le \sum_{j=1}^3 L_{\ux_j}(\uq) - (q_1+q_2) + c
  = \sum_{j=1}^3 (L_{\ux_j}(\uq) - p_j) + c,
\end{align*}
for a constant $c=c(\uxi)\ge 0$.  By the second part of \eqref{tool:lemma2:eq1}, 
this implies that 
\[
 \norm{\Phi(\up')-\uL_\uxi(\uq)}\le 3\delta+c
 \et
 -2\delta-c \le L_{\ux_j}(\uq) - p_j \le \delta  \quad \text{for $j=1,2,3$.} 
\]
Thus, $\norm{\Phi(\up)-\Phi(\up')} \le \norm{\up-\up'} \le 2\delta+c$, and so 
$\big\|\uL_\uxi(\uq)-\Phi(\up)\big\| \le 5\delta +2c$. 
\end{proof}

The next statement provides an estimate for $\uL_\uxi$ outside the angular 
sector $\cD$ defined by \eqref{intro:eq:D} when $\xi_1$ and $\xi_2$ are badly 
approximable, as illustrated in Figure~\ref{tool:fig1}.  Recall that a real 
number $\xi$ is called \emph{badly approximable} when there exists a 
constant $c=c(\xi)>0$ such that $|x_0\xi-x_1|\ge c^{-1}|x_0|^{-1}$ for any 
$(x_0,x_1)\in\bZ^2$ with $x_0\neq 0$.

\begin{lemma}
\label{tool:lemma3}
Let $\uq=(q_1,q_2)\in\bR^2$.  
\begin{itemize}
\item[(i)] If $q_1\le 0$ and $q_2\le 0$, then \  
$\uL_\uxi(\uq)=\Phi(0,q_1,q_2)+ \cO_\uxi(1)$.
\medskip
\item[(ii)] Suppose that $\xi_1$ is badly approximable.  
   If $q_1\ge \max\{0,2q_2\}$, then
\[
\uL_\uxi(\uq)=(q_2,q_1/2,q_1/2)+ \cO_\uxi(1).
\]
\item[(iii)] Suppose that $\xi_2$ is badly approximable.  
   If $q_2\ge \max\{0,2q_1\}$, then
\[
\uL_\uxi(\uq)=(q_1,q_2/2,q_2/2)+ \cO_\uxi(1).
\]
\end{itemize} 
\end{lemma}

\begin{proof}
For the basis $\{\ue_1=(1,0,0),\, \ue_2=(0,1,0),\, \ue_3=(0,0,1)\}$ 
of $\bZ^3$, we find 
\[
 L_{\ue_1}(\uq) = \max\{ 0, q_1+\log|\xi_1|, q_2+\log|\xi_2| \}, \quad
 L_{\ue_2}(\uq) = q_1, \quad
 L_{\ue_3}(\uq) = q_2.
\]
If $q_1\le 0$ and $q_2\le 0$, we have $L_{\ue_1}(\uq)=\cO_\uxi(1)$, 
and Lemma~\ref{tool:lemma2} yields
\[
 \norm{\uL_\uxi(\uq)-\Phi(0,q_1,q_2)}\le 5|L_{\ue_1}(\uq)|+c_1 = \cO_\uxi(1),
\]
which proves (i).

Suppose that $\xi_1$ is badly approximable and that 
$q_1\ge \max\{0, 2q_2\}$.  Set
\begin{equation}
\label{tool:lemma3:eq5}
 (t_1,t_2)=\uL_A(0,q_1)
 \quad\text{where}\quad
 A=\matrice{1}{0}{\xi_1}{-1}.
\end{equation}
Then, choose linearly independent points $(x_{1,0},x_{1,1})$, 
$(x_{2,0},x_{2,1})$ in $\bZ^2$ such that, for $j=1,2$,
\[
 |x_{j,0}| \le \exp(t_j) \et |x_{j,0}\xi_1-x_{j,1}| \le \exp(t_j-q_1). 
\]
Since $\xi_1$ is badly approximable, these estimates imply that 
$t_j \ge q_1/2+\cO_{\xi_1}(1)$.    
As Minkowski's inequality \eqref{intro:eq:sumLA} applied to
\eqref{tool:lemma3:eq5} gives $t_1+t_2=q_1+\cO_{\xi_1}(1)$,
we deduce that
\[
 t_j=q_1/2+\cO_{\xi_1}(1) \quad\text{for $j=1,2$.}
\]
Finally, for $j=1,2$, set $\ux_j=(x_{j,0},x_{j,1},x_{j,2})$ 
for an integer $x_{j,2}$ such that 
\[
 |x_{j,0}\xi_2-x_{j,2}| \le 1\le \exp(q_1/2-q_2),
\]
and set $\ux_3=(0,0,1)$.  Then $\ux_1$, $\ux_2$ and $\ux_3$ 
are linearly independent points of $\bZ^3$ with
\[
 L_{\ux_j}(\uq) \le \max\{t_j, q_1/2\} \le q_1/2+\cO_{\xi_1}(1)
 \quad\text{for $j=1,2$.}
\]  
As $L_{\ux_3}(\uq)=q_2$, Lemma~\ref{tool:lemma2} yields
$\norm{\uL_\uxi(\uq)-\Phi(q_1/2,q_1/2,q_2)}=\cO_\uxi(1)$.  
This proves (ii).  The proof of (iii) is similar.
\end{proof}

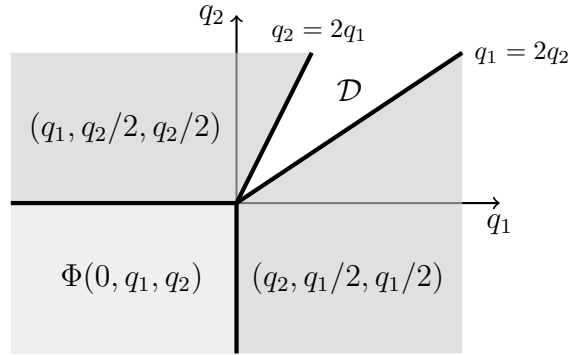
\begin{figure}[h]
\begin{tikzpicture}[scale=0.5]
       \draw[->, thick, black] (-6,0)--(7,0) node[below]{$q_1$};
       \draw[->, thick, black] (0,-4)--(0,5) node[left]{$q_2$};
       \fill[thick,fill=gray!30,opacity=0.4] (-6,0)--(0,0)--(0,-4)--(-6,-4);
       \fill[thick,fill=gray!40,opacity=0.6] (0,0)--(6,4)--(6,-4)--(0,-4)--(0,0);
       \fill[thick,fill=gray!40,opacity=0.6] (0,0)--(2,4)--(-6,4)--(-6,0)--(0,0);
       \draw[-, ultra thick] (-6,0)--(0,0)--(2,4);
       \draw[-, ultra thick, font=\footnotesize] (0,-4)--(0,0)--(6,4) 
              node[right]{$q_1=2q_2$};
       \node[above, font=\footnotesize] at (2.2,4){$q_2=2q_1$};
       \draw node[left] at (-0.6,-2) {$\Phi(0,q_1,q_2)$};
       \draw node[right] at (0.1,-2) {$(q_2,q_1/2,q_1/2)$};
       \draw node[left] at (-0.1,2) {$(q_1,q_2/2,q_2/2)$}; 
       \draw node at (3,3) {$\cD$};            
\end{tikzpicture}
\caption{The map $\uL_\uxi$ up to bounded difference outside of $\cD$, 
 when $\xi_1$ and $\xi_2$ are badly approximable}
\label{tool:fig1}
\end{figure}

Littlewood's conjecture states that, for any choice of $\xi_1,\xi_2\in\bR$ and 
any $\epsilon>0$, there exists $\ux=(x_0,x_1,x_2)\in\bZ^3$ with $x_0\neq 0$ 
such that 
\[
 |x_0|\,|x_0\xi_1-x_1|\,|x_0\xi_2-x_2| \le \epsilon.
\]
We leave the reader check that, in the present setting, an equivalent 
formulation of the conjecture is that there is no pair of badly approximable
numbers $\xi_1,\xi_2\in\bR$ such that the point $\uxi=(1,\xi_1,\xi_2)\in\bR^3$
satisfies
\[
 \uL_\uxi(\uq)=\frac{q_1+q_2}{3}(1,1,1)+\cO_\uxi(1)
 \quad\text{for each $\uq=(q_1,q_2)\in\cD$.}
\]
Although, the construction that we present in the next section is far from 
producing a counterexample, it could eventually inspire one that does so.   
In this respect, we stress that, for each integer $n\ge 2$, there are matrices 
$A\in\GL_n(\bR)$ such that 
\[
 \uL_A(\uq)=\frac{q_1+\cdots+q_n}{n}(1,\dots,1)+\cO_A(1)
 \quad\text{for each $\uq=(q_1,\dots,q_n)\in\bR^n$,}
\]
and so the problem disappears for matrices (this follows for example from 
\cite[Section~4.2, Theorem~1]{GL1987}).

%
%

\section{Notation and main results}
\label{sec:main}


\subsection{Monoid of words} 
\label{main:ssec:E*}
Let $E$ be a non-empty set.  We denote by $E^*$ the monoid of words 
on the alphabet $E$, namely the set of finite (possibly empty) sequences 
of elements of $E$, with product given by concatenation of words.  Its 
neutral element is the empty word $\epsilon$.  Given words $u,w\in E^*$,
we say that $u$ is a \emph{prefix} of $w$ and write $u\le w$ if $w=uv$ for some 
$v\in E^*$.  More generally, we say that $u$ is a \emph{factor} of $w$ 
if $w=vuv'$ for some $v,v'\in E^*$.
We denote by $[u,w]$ the set of words $v\in E^*$ with
$u\le v\le w$.   We write $u<w$ if $u\le w$ and $u\neq w$.  Using 
standard convention, we also denote by $]u,w]$, $[u,w[$ 
and $]u,w[$ the sets obtained by removing respectively $u$, $w$ and 
$\{u,w\}$ from $[u,w]$.  The reverse of a non-empty word $a_1a_2\cdots a_n$ 
of $E^*$ is defined as the word $a_n\cdots a_2a_1$ with letters written in 
reverse order.  The reverse of $\epsilon$ is itself.  Thus $\epsilon$ is an
example of palindrome in $E^*$, that is a word of $E^*$ which coincides 
with its reverse.


\subsection{Fibonacci sequences} 
\label{main:ssec:Fib}
We say that a sequence $(x_i)_{i\ge 1}$ in a monoid $\cM$ is a 
\emph{Fibonacci sequence} if it satisfies $x_{i+2}=x_{i+1}x_i$ for each 
$i\ge 1$.  Such a sequence is uniquely determined by its first two elements 
$x_1$ and $x_2$.  Moreover, if $\varphi\colon\cM\to\cN$ is a morphism 
of monoids, then a Fibonacci sequence $(x_i)_{i\ge 1}$ in $\cM$ yields
a Fibonacci sequence $(\varphi(x_i))_{i\ge 1}$ in $\cN$.

The set $\bN=\{0,1,2,\dots\}$ of non-negative integers is a monoid under 
addition.  We denote by $(F_i)_{i\ge -1}$, the usual Fibonacci sequence
starting with $F_{-1}=0$, $F_0=1$ and obeying $F_{i+2}=F_{i+1}+F_i$ for 
each $i\ge -1$.  It is given in closed form by Binet's formula
\begin{equation}
 \label{main:Binet}
 F_i=(\gamma^{i+1}-(-\gamma)^{-i-1})/\sqrt{5} 
 \quad\text{where}\quad
 \gamma=(1+\sqrt{5})/2.
\end{equation}

Let $E=\{a,b\}$ be an alphabet of two letters.  The Fibonacci sequence 
$(w_i)_{i\ge 1}$ in $E^*$ starting 
with $w_1=a$ and $w_2=ab$, has $w_3=aba$, $w_4=abaab$, etc.  Since 
any word in that sequence is a prefix of the next word, this sequence 
converges pointwise to an infinite word
\begin{equation}
 \label{main:fab}
 f_{a,b}:=w_\infty:=\lim_{i\to\infty} w_i = abaababaabaab\dots
\end{equation}
From now on, we reserve the notation $(w_i)_{i\ge 1}$ for the above 
Fibonacci sequence in $E^*$ and $w_\infty$ for its limit.  We say that
a word $u$ in $E^*$ is a prefix of $w_\infty$ and write $u<w_\infty$, 
if $u\le w_i$ for some $i\ge 1$.  Then we denote by $[u,w_\infty[$ 
(resp.~$]u,w_\infty[$\,) the set of words $v\in E^*$ 
with $u\le v < w_\infty$ (resp.~$u < v < w_\infty$).   We say that
$u$ is a factor of $w_\infty$ if it is a factor of $w_i$ for some $i\ge 1$.

We denote by $\theta\colon E^*\to E^*$ the morphism of monoids 
determined by the conditions $\theta(a)=ab$ and $\theta(b)=a$.  
It satisfies
\begin{equation}
 \label{main:eq:theta}
 \theta(w_i) = w_{i+1}\quad\text{for each $i\ge 1$,}
\end{equation}
and preserves the partial order on $E^*$.   Thus it sends to itself
the set $[\epsilon,w_\infty[$ of prefixes of $w_\infty$.
 
The map which sends a word $w$ to its length $|w|$ yields a 
morphism of monoids from $E^*$ to $\bN$ which restricts to an 
order preserving bijection from $[\epsilon,w_\infty[$ to $\bN$.  It satisfies
\begin{equation}
 \label{main:eq:length}
 |w_i| = F_i \quad\text{for each $i\ge 1$.}
\end{equation}

In general, given any words $u,v$ on any alphabet, we denote by 
$f_{u,v}=uvuuv\dots$ the limit of the Fibonacci sequence starting 
with $(u, uv, uvu, \dots)$, and we call it the \emph{infinite Fibonacci word 
on $(u, v)$}.  


\subsection{The sets $\cV_\ell$} 
\label{main:ssec:V}
Let 
\begin{equation}
 \label{main:F}
 \cF=\{\epsilon, w_1,w_2,w_3,\dots\}=\{\epsilon, a, ab, aba, abaab,\dots\}
\end{equation}
denote the set made of $\epsilon$ and all Fibonacci words $w_i$ with 
$i\ge 1$.  We define a map $\iota$ from $[\epsilon,w_\infty[$ to itself in
the following way.  If a prefix $v$ of $w_\infty$  belongs to $\cF$,
we set $\iota(v)=v$.  Otherwise, we have $v\in\left]w_i,w_{i+1}\right[$ 
for some integer $i\ge 3$ and we define $\iota(v)$ to be the prefix
of $w_\infty$ of length $|v|-2F_{i-2}$.  This makes sense since 
$|v|-2F_{i-2}>F_i-2F_{i-2}=F_{i-3}\ge 1$.  In all cases, we have 
$\iota(v)\le v$ and so, the sequence $(\iota^k(v))_{k\ge 1}$ of iterates 
of $\iota$ at $v$ is eventually constant, equal to some element 
of $\cF$.  We thus obtain a map 
$\alpha\, \colon \left[\epsilon,w_\infty\right[ \to\cF$ by sending $v$ 
to this word:
\begin{equation}
 \label{main:alpha}
 \alpha(v)=\lim_{k\to\infty}\iota^k(v).
\end{equation}
The following table provides the values of the maps $\iota$ and $\alpha$ on 
the first 8 prefixes of $w_\infty$.
\begin{equation}
 \label{main:table}
 \renewcommand{\arraystretch}{1.5}
 \begin{array}{c||c|c|c|c|c|c|c|c|c|c|c|c|c|c}
  v &\epsilon&w_1=a&w_2=ab&w_3=aba&abaa&w_4=abaab&abaaba&abaabab\\
  \hline
  \iota(v) &\epsilon&a&ab&aba&ab&abaab&ab&aba\\
  \hline
  \alpha(v) &\epsilon&a&ab&aba&ab&abaab&ab&aba
\end{array}
\end{equation}

For each integer $\ell\ge 1$, we define
\begin{equation}
 \label{main:eq:V}
 \cV_\ell=\{ v\in \prefixes \,;\, \alpha(v)\ge w_\ell \}.
\end{equation}
This is an infinite set as it contains $w_i$ for each $i\ge \ell$.  We say 
that elements $v_1<v_2<\cdots<v_k$ of $\cV_\ell$ are \emph{consecutive}
in $\cV_\ell$ if, for each index $i$ with $1\le i<k$, there is no element $v$ 
of $\cV_\ell$ with $v_i<v<v_{i+1}$.  Sections~\ref{sec:Vbar} and \ref{sec:V} 
are devoted to the combinatorial properties of the sets $\cV_\ell$.  In 
particular, Corollary~\ref{Vbar:prop2:cor1} shows that, for any consecutive 
elements $u<v$ of $\cV_\ell$ with $\ell\ge 3$, the difference $|v|-|u|$ is 
$F_{\ell-2}$ or $F_{\ell-1}$.  Proposition~\ref{V:prop2} shows that
$\cV_{\ell+1}=\theta(\cV_\ell)$ for each $\ell\ge 4$.


\subsection{The  functions $P_v$} 
\label{main:ssec:Pv}
For each $v\in [\epsilon,w_\infty[$, we define a function $P_v\colon\bR^2\to\bR$
by 
\begin{equation}
 \label{main:Pv}
 P_v(q_1,q_2)=\max\{q_1-|v|, q_2-|\alpha(v)|, |v|\},
\end{equation}
and we denote by $\cA(v)$, $\cB(v)$ and $\cC(v)$ the closed sets made 
of the points $\uq=(q_1,q_2)\in\bR^2$ where $P_v(\uq)$ is respectively 
equal to $q_1-|v|$, $q_2-|\alpha(v)|$ and $|v|$.  In their interior, $P_v$
is differentiable with constant gradient $(1,0)$, $(0,1)$ and $(0,0)$ 
respectively.  Explicitly, these sets are
\begin{itemize}
\item 
 $\cA(v)=\{(q_1,q_2)\in\bR^2\,;\, 
        q_1\ge 2|v| \text{ and } q_2\le q_1-|v|+|\alpha(v)|\}$,
\smallskip
\item 
 $\cB(v)=\{(q_1,q_2)\in\bR^2\,;\, 
    q_2\ge |v|+|\alpha(v)| \text{ and } q_2\ge q_1-|v|+|\alpha(v)|\}$,
\smallskip
\item 
 $\cC(v)=\{(q_1,q_2)\in\bR^2\,;\, 
   q_1\le 2|v| \text{ and } q_2\le |v|+|\alpha(v)|\}$.
\end{itemize} 
They are closed sectors of $\bR^2$ with disjoint interiors and their union is
$\bR^2$.  Figure \ref{main:fig:Pv} shows the three sets together with 
the value of $P_v$ on each of them.  It also shows by an arrow, the gradient 
of $P_v$ in their interior (reduced to a point for $\cC(v)$).   Note in particular 
that 
\begin{equation}
\label{main:eq:A0}
 \cA(\epsilon)=\{(q_1,q_2)\in\bR^2\,;\, q_1\ge 0 \text{ and } q_2\le q_1\}.
\end{equation}

\begin{figure}[h]
\begin{tikzpicture}[xscale=0.7, yscale=0.7] 
       \draw[->, thick, black] (-2,0)--(9,0) node[below]{$q_1$};
       \draw[->, thick, black] (-0.5,-1)--(-0.5,6) node[left]{$q_2$};
       \draw[-, ultra thick] (4,3) -- (-1.5,3) node[left]{$q_2=|v|+|\alpha(v)|$};
       \draw[-, ultra thick] (4,3) -- (4, -0.7) node[below]{$q_1=2|v|$};
       \draw[-, ultra thick] (4,3) -- (7, 6);
       \draw[->,ultra thick] (6,3)--(7,3); 
        \draw[->,ultra thick] (0.5,4.7)--(0.5,5.7);
       \draw node at (6.5,2.3) {$P_v=q_1-|v|$};
       \draw node at (1.9,0.8) {$P_v=|v|$}; 
       \draw node[right] at (0.8,5.2) {$P_v=q_2-|\alpha(v)|$};
       \draw node at (6.5,3.7) {$\cA(v)$};
       \draw node at (2.1,3.9) {$\cB(v)$};
       \draw node at (2.1,2.2) {$\cC(v)$};       
       \node[draw,circle,inner sep=1pt,fill] at (1.9,1.5) {};
\end{tikzpicture}
\caption{The function $P_v$ attached to a prefix $v$ of $w_\infty$.}
\label{main:fig:Pv}
\end{figure}
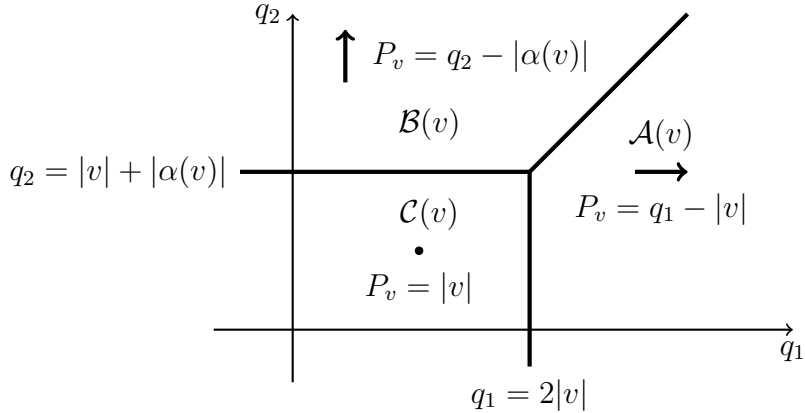


\subsection{The  map $\uP$} 
\label{main:ssec:P}
Using the function $\Phi\colon\bR^3\to\Delta_3$ from section~\ref{sec:tool},
we construct a map $\uP=(P_1,P_2,P_3)\colon\cA(\epsilon)\to\Delta_3$ as 
follows.  For each point $\uq=(q_1,q_2)$ in $\cA(\epsilon)$, we set
\begin{equation}
\label{main:P:eq1}
 \uP(\uq)=\Phi(q_1-|u|,q_2+|u|-|w|,|w|)
\end{equation}
where $u$ is the largest prefix of $w_\infty$ such that $\uq\in \cA(u)$,
and $w$ is the smallest prefix of $w_\infty$ with $u<w$ such that
$\uq\in\cC(w)$.  Such prefixes exist because the condition 
$\uq\in \cA(u)$ requires $2|u|\le q_1$, and if we choose $i\ge 1$
such that $q_2\le q_1\le 2F_i$ then $\uq\in\cC(w_i)=\left]\infty,2F_i\right]^2$.
By choice of $u$ and $w$, we have $\uq\in\cB(v)$ for any word $v$
with $u<v<w$.  In Section~\ref{sec:P}, we will show that $|w|-|u|=F_\ell$
for some integer $\ell\ge 1$ and that, if $\ell>1$, there exists a
unique word $v$ such that $u<v<w$ are consecutive elements of 
$\cV_\ell$ with $|\alpha(v)|=F_\ell=|w|-|u|$ (see Proposition~\ref{P:prop1}).  
In the latter case, we have
\begin{equation}
\label{main:P:eq2}
 \uq\in\cA(u)\cap\cB(v)\cap\cC(w)
 \et  
 \uP(\uq)=\Phi(P_u(\uq),P_v(\uq),P_w(\uq)).
\end{equation}

In Section~\ref{sec:properties}, we view $\uP$ as an example of what we call 
a \emph{integral $2$-parameter $3$-system}.  In particular, 
it is a continuous map whose sum of the components is 
\begin{equation}
\label{main:P:eq3}
 P_1(\uq)+P_2(\uq)+P_3(\uq)=q_1+q_2 \quad 
 \text{for each $\uq=(q_1,q_2)\in\cA(\epsilon)$.}
\end{equation}   
The following additional properties are established in 
sections~\ref{sec:proof1} and \ref{sec:proof2} respectively.

\begin{theorem}
\label{main:P:thm1}
For each integer $k\ge 4$, each 
$\uq=(q_1,q_2)\in\bR^2$ with
\begin{equation}
\label{main:P:thm1:eq1} 
  0\le q_1\le 2F_{k-1}+2F_{k-3} \et 0\le q_2\le \min\{q_1,2F_{k-1}\},
\end{equation}
and each $j=1,2,3$, we have
$P_j(q_1+4F_{k-2},q_2+2F_{k-2}) = P_j(q_1,q_2) + 2F_{k-2}$.
\end{theorem}

\begin{theorem}
\label{main:P:thm2}
For each $\uq\in\cA(\epsilon)$, we have
$\norm{\uP(\gamma\uq) - \gamma\uP(\uq)} \le 40$.
\end{theorem}


\subsection{Main result} 
\label{main:ssec:extremal}
For any matrix $A$ with real coefficients, we denote by $\|A\|$ the maximum 
of the absolute values of its entries.  This agrees with our convention for vectors 
in $\bR^n$, for any positive integer $n$.

We say that a real number $\xi$ is \emph{extremal of $\GL_2(\bZ)$-type} if there
exist an unbounded  Fibonacci sequence of matrices $(W_i)_{i\ge 1}$ 
in $\GL_2(\bZ)$, a matrix $M\in\GL_2(\bQ)$ with ${}^tM\neq \pm M$, and 
a constant $c\ge 1$ which, for each $i\ge 1$, satisfy the following properties:
\begin{itemize}
\item[(E1)] the product $\ux_i:=W_iM_i^{-1}$ is a symmetric matrix, where
  $\disp M_i=\begin{cases} M &\text{if $i$ is even,}\\ {}^tM &\text{if $i$ is odd;}
   \end{cases}$ \\[-5pt]
\item[(E2)] $c^{-1}\|W_{i+1}\|\,\|W_i\| \le \|W_{i+2}\| \le  c\|W_{i+1}\|\,\|W_i\|$; \\[-5pt] 
\item[(E3)] $c^{-1}\|W_i\|^{-1} \le \|(\xi,-1)W_i\| \le c\|W_i\|^{-1}$. \\[-5pt]
\end{itemize}
According to \cite[Theorem~2.2]{R2008}, these numbers $\xi$ are simply those
whose continued fraction expansion coincides, up to its first terms, with an 
infinite Fibonacci word $f_{u,v}$ on two non-commuting 
word $u$ and $v$ in $(\bN\setminus\{0\})^*$.   Thus, they are badly 
approximable and their set is stable under 
the action of $\GL_2(\bZ)$ on $\bR\setminus\bQ$ by fractional linear
transformations.  Note that we may always choose $M$ with relatively prime 
integer coefficients.  Then the products $\det(M)\ux_i$ also have relatively
prime integer coefficients.

For each integer $m\ge 1$, we denote by $\cE_m$ the set of real numbers 
$\xi$ such that the properties (E1)--(E3) hold for an unbounded Fibonacci 
sequence $(W_i)_{i\ge 1}$ in $\GL_2(\bZ)$ and the choice of 
\begin{equation}
\label{main:extremal:eq:M}
 M=\begin{pmatrix} m &1\\ -1 &0 \end{pmatrix} \in \GL_2(\bZ).
\end{equation}
Then the sequence $(\ux_i)_{i\ge 1}$ defined in (E1) is contained in $\GL_2(\bZ)$.
In \cite[Section~3]{R2003}, it is shown that $\cE_m$ is non-empty for each $m$.
We also denote by $\cE_m^+$ the set of elements of $\cE_m$ associated with a Fibonacci sequence $(W_i)_{i\ge 1}$ in $\SL_2(\bZ)$.  The set $\cE_3^+$ is studied 
in \cite{R2011} in connection with Markoff's theory in Diophantine approximation.
In section~\ref{sec:ex}, we show that it contains for example the number 
\begin{equation}
\label{main:extremal:eq:example}
 \xi=[0,1,1,2,2,1,1,2,2,2,2,1,1,\dots]=[0,\uun,f_{\udeux,\uun}]
\end{equation}
where $\uun=(1,1)$, $\udeux=(2,2)$, and $f_{\udeux,\uun}$ is the infinite
Fibonacci word on $\udeux$ and $\uun$ (as defined in section~\ref{main:ssec:Fib}).
However, by \cite[Lemma~3.4]{R2011}, the set $\cE_m^+$ is empty for $m\neq 3$.
The main result of the present paper is the following.

\begin{theorem}
\label{main:extremal:thm}
Let $\xi\in\cE_m$ for some integer $m\ge 1$, and let $\uxi=(1,\xi,\xi^2)$.  
Then there exist $\rho>0$ and $c>0$ such that
\[
 \norm{ \uL_\uxi(\uq)-\rho\uP(\rho^{-1}\uq) } \le c
\]
for each $\uq\in\cA(\epsilon)$.
\end{theorem}


\subsection{A family of points} 
\label{main:ssec:points}
Fix a choice of $\xi\in\cE_m$ for an integer $m\ge 1$ and set 
$\uxi=(1,\xi,\xi^2)$.  Choose
also a corresponding unbounded Fibonacci sequence $(W_i)_{i\ge 1}$
in $\GL_2(\bZ)$ satisfying the conditions (E1)--(E3) of 
section~\ref{main:ssec:extremal} for the matrix $M$ given by 
\eqref{main:extremal:eq:M}.  Let $(\ux_i)_{i\ge 1}$ be the sequence 
of symmetric matrices in $\GL_2(\bZ)$ defined in (E1).
Following \cite{R2004, R2003}, we identify each point 
$\ux=(x_0,x_1,x_2)\in\bQ^3$ with the symmetric matrix 
\[
 \ux=\matrice{x_0}{x_1}{x_1}{x_2}.
\]

We fix an alphabet $E=\{a,b\}$ with two letters and set 
\[
 W_0 = W_1^{-1}W_2.
\]
We denote by $\varphi\colon E^*\to\GL_2(\bZ)$ the morphism of monoids 
determined by the conditions $\varphi(b)=W_0$ and $\varphi(a)=W_1$.  Then, in 
the notation of section~\ref{main:ssec:Fib}, we observe that
\[
 \varphi(w_i)=W_i \quad\text{for each $i\ge 1$.}
\]
For each non-empty word $v\in E^*$, we set
\begin{equation}
\label{main:points:Mv}
M(v) = \begin{cases}
 M &\text{if $v$ ends in $b$,}\\
 {}^tM &\text{if $v$ ends in $a$.}
 \end{cases}
\end{equation}
This yields $M(w_i)=M_i$ for each $i\ge 1$.  We also define
\begin{equation}
\label{main:points:xv}
\ux(v) = \matrice{x_0(v)}{x_1(v)}{x_1(v)}{x_2(v)}
\end{equation}
as the symmetric matrix with integer coefficients that has the same first 
column as the matrix $\varphi(v)M(v)^{-1}$ and satisfies $|x_1(v)\xi-x_2(v)|<1/2$.
The first condition determines $x_0(v)$ and $x_1(v)$, and the second then 
specifies uniquely $x_2(v)$ as $\xi\notin\bQ$.  We also view $\ux(v)$ as 
the point $(x_0(v),x_1(v),x_2(v))\in\bZ^3$ as stated at the beginning 
of the section.  Then, according to \eqref{tool:eq:Lx}, its trajectory 
relative to $\uxi$ is the function $L_{\ux(v)}\colon\bR^2\to\bR$ given by
\begin{equation}
\label{main:points:eq:Lxv}
 L_{\ux(v)}(\uq) 
  = \max\{\log|x_0(v)|,\, q_1+\log|x_0(v)\xi-x_1(v)|,\, q_2+\log|x_0(v)\xi^2-x_2(v)|\}
\end{equation}
for each $\uq=(q_1,q_2)\in\bR^2$.  

The tools of section~\ref{sec:rec} allows us to compute 
$\ux(v)$ for each $v\in\cV_\ell$ with $\ell$ large enough.  In particular,
we find that $\ux(w_i)=\ux_i$
for each large enough $i$.  The following crucial result is proved in 
section~\ref{sec:final}, using the functions $P_v$ of section~\ref{main:ssec:Pv}.

\begin{theorem}
\label{main:points:thm}
With the above notation, there exist an integer $\ell_0\ge 4$ and a constant 
$\rho>0$ such that, for each integer $\ell\ge \ell_0$, the following properties
hold.
\begin{itemize} 
 \item[(i)] For any $v\in\cV_\ell$ and any 
 $\uq\in\cA(\epsilon)$, we have
\[
 L_{\ux(v)}(\uq)=\rho P_v(\rho^{-1}\uq)+\cO_\xi(1),
\] 
where $\cO_\xi(1)$ denotes a function of $v$ and $\uq$ whose absolute
value is bounded above by a constant that depends only on $\xi$.
 \item[(ii)] For any triple of consecutive
 elements $u<v<w$ of $\cV_\ell$, the points $\ux(u)$, $\ux(v)$, $\ux(w)$ 
 are linearly independent if and only if $v\notin\cV_{\ell+1}$.  
 \end{itemize}
\end{theorem}


\subsection{Proof of the main result} 
\label{main:ssec:proof}

Using the results of section~\ref{sec:tool} and taking for granted the 
statements of the preceding sections, Theorem~\ref{main:extremal:thm} 
is proved as follows.

Choose $\rho>0$ and $\ell_0$ as in Theorem \ref{main:points:thm}, and 
fix a point $\uq=(q_1,q_2)\in\cA(\epsilon)$.
Then $\rho^{-1}\uq$ also belongs to $\cA(\epsilon)$, and so, according to 
section~\ref{main:ssec:P}, we have
\begin{equation}
\label{main:proof:eq1}
 \rho\uP(\rho^{-1}\uq) = \Phi(\up) 
 \quad\text{with}\quad
 \up=(q_1-\rho|u|, q_2+\rho|u|-\rho|w|, \rho|w|)
\end{equation}
where $u$ is the largest prefix of $w_\infty$ such that $\rho^{-1}\uq\in\cA(u)$,
and $w$ is the smallest prefix of $w_\infty$ with $u<w$ such that 
$\rho^{-1}\uq\in\cC(w)$.  Moreover, $|w|-|u|=F_\ell$ for some positive
integer $\ell$. Furthermore, if $\ell>1$, then there exists a unique word 
$v\in\left]u,w\right[$ such that $u<v<w$ are consecutive elements of 
$\cV_\ell$ with
\[
 |\alpha(v)| =F_\ell=|w|-|u| \et \rho^{-1}\uq\in\cA(u)\cap\cB(v)\cap\cC(w),
\]
thus, by definition of $P_u$, $P_v$ and $P_w$, we have
\[
 \up = \big( 
     \rho P_u(\rho^{-1}\uq), 
     \rho P_v(\rho^{-1}\uq), 
     \rho P_w(\rho^{-1}\uq) \big).  
\]

Suppose first that $\ell\ge \ell_0$.  Then, since $\ell_0\ge 4$, the above 
formula for $\up$ holds and, by Theorem~\ref{main:points:thm}(i), 
we have
\[
 \up = \big( L_{\ux(u)}(\uq), L_{\ux(v)}(\uq), L_{\ux(w)}(\uq) \big) +\cO_\xi(1).  
\]
Since $\Phi$ is $1$-Lipschitz, this yields
\[
 \rho\uP(\rho^{-1}\uq) 
 = \Phi(\up) 
 =\Phi \big( L_{\ux(u)}(\uq), L_{\ux(v)}(\uq), L_{\ux(w)}(\uq) \big) +\cO_\xi(1).
\]
Moreover, as $v\notin\cV_{\ell+1}$, the integer points $\ux(u), \ux(v), \ux(w)$ are linearly independent by Theorem~\ref{main:points:thm}(ii), and so we have
\[
 \uL_\uxi(\uq) \le \Phi \big( L_{\ux(u)}(\uq), L_{\ux(v)}(\uq), L_{\ux(w)}(\uq) \big)
\]
componentwise.  Altogether, this means that
\[
 L_{\uxi,j}(\uq) \le \rho P_j(\rho^{-1}\uq) + \cO_\xi(1)
 \quad (1\le j\le 3).
\]
We also observe from \eqref{main:proof:eq1} that the sum of the 
coordinates of $\up$ is $q_1+q_2$.  Since the sum of the coordinates of 
$\Phi(\up)=\rho\uP(\rho^{-1}\uq)$ is the same, we conclude from
Lemma~\ref{tool:lemma2} that
\begin{equation}
\label{main:proof:eq8}
 \norm{\uL_{\uxi}(\uq) - \rho \uP(\rho^{-1}\uq)} = \cO_\xi(1).
\end{equation}

Suppose now that $\ell<\ell_0$.   As $\rho^{-1}\uq\in\cA(u)\cap\cC(w)$, 
we have
\[
 2|u|\le \rho^{-1}q_1 \le 2|w|
 \et
 \rho^{-1}q_2\le \min\{\rho^{-1}q_1-|u|+|\alpha(u)|,\, |w|+|\alpha(w)|\}.
\]
Since $|w|-|u|=F_\ell$, we also have $\{u,w\}\not\subseteq\cV_{\ell+3}$
according to the remark at the end of section~\ref{main:ssec:V} (or see
Corollary~\ref{Vbar:prop2:cor1}). Thus,
\[
  |w|-|u|=F_\ell=\cO_\xi(1) 
  \et
 \min\{|\alpha(u)|,|\alpha(w)|\}\le F_{\ell+2}=\cO_\xi(1).
\]
So, the preceding estimates yield
\[
 q_2\le q_1/2+\cO_\xi(1) 
 \et 
 \up=(q_1/2,q_2,q_1/2)+\cO_\xi(1),
\]
using the formula for $\up$ given in \eqref{main:proof:eq1}.  Thus,
\[
 \rho\uP(\rho^{-1}\uq)=\Phi(q_2,q_1/2,q_1/2)+\cO_\xi(1).
\]
To estimate $\uL_\uxi(\uq)$, we use the fact that $\xi$ is badly 
approximable, as stated in section~\ref{main:ssec:extremal}.   Set 
$\uq'=(q_1,q'_2)$ where $q'_2=\min\{q_2,q_1/2\}$.  By the above, we have
$\|\uq-\uq'\|=|q_2-q'_2|=\cO_\xi(1)$.   Then, Lemmas~\ref{tool:lemma1}
and \ref{tool:lemma3}(ii) yield
\[
 \uL_\uxi(\uq)=\uL_\uxi(\uq')+\cO_\xi(1)
  =\Phi(q'_2,q_1/2,q_1/2)+\cO_\xi(1)
  =\Phi(q_2,q_1/2,q_1/2)+\cO_\xi(1).
\]
Thus, \eqref{main:proof:eq8} also holds in this case.  This completes the 
proof of Theorem~\ref{main:extremal:thm}.

%
%

\section{The maps $\iotabar$ and $\alphabar$, and the sets $\cVbar_\ell$}
\label{sec:Vbar}

In Section~\ref{main:ssec:V}, we defined two maps $\iota$ and $\alpha$ from 
$\prefixes$ to itself.  Using the natural bijection from $\prefixes$ to $\bN$ 
given by the length of a word, these yield maps $\iotabar$ and $\alphabar$ 
from $\bN$ to $\bN$.   Similarly, the subsets $\cV_\ell$ of $\prefixes$ 
correspond to subsets $\cVbar_\ell$ of $\bN$.  This section is devoted to the 
combinatorial properties of these simpler objects.

Here, the set $\cF$, given by \eqref{main:F}, is replaced by the set 
\begin{equation}
 \label{Vbar:Fbar}
 \cFbar=\{F_i\,;\,i\ge -1\}=\{0,1,2,3,5,8,13,\dots\}
\end{equation}
of all Fibonacci numbers, including $0$.  We define 
$\iotabar \colon \bN\to\bN$ by
\[
 \iotabar(x)
=\begin{cases}
   x &\text{if $x\in\cFbar$,}\\
   x-2F_{i-2} &\text{if $F_i<x<F_{i+1}$ for an integer $i\ge 3$.}
 \end{cases}
\]
We note that, for each $x\in\bN$, we have $\iotabar(x)\le x$ with equality if
and only if $x\in\cFbar$.  Thus, the sequence $(\iotabar^k(x))_{k\ge 1}$ of 
iterates of $\iotabar$ evaluated at $x$ is non-increasing and so, 
from some point on, it is constant, equal to some element of $\cFbar$.
We thus obtain a map $\alphabar\colon\bN\to\cFbar$ by sending $x$ 
to this value:
\begin{equation}
 \label{Vbar:alphabar}
 \alphabar(x)=\lim_{k\to\infty}\iotabar^k(x).
\end{equation}
Clearly, we have $\alphabar\circ\iotabar=\alphabar$.  We also note that, 
for each $x\in\bN$, we have $\alphabar(x)\le x$ with equality if and 
only if $x\in\cFbar$.  The next table gives the values of the maps 
$\iotabar$ and $\alphabar$ on 
$\bN\cap[0,13]$.
\begin{equation}
 \label{Vbar:table}
 \renewcommand{\arraystretch}{1.5}
 \begin{array}{c||c|c|c|c|c|c|c|c|c|c|c|c|c|c}
  x &0&1&2&3&4&5&6&7&8&9&10&11&12&13\\
  \hline
  \iotabar(x) &0&1&2&3&2&5&2&3&8&3&4&5&6&13\\
  \hline
  \alphabar(x) &0&1&2&3&2&5&2&3&8&3&2&5&2&13
\end{array}
\end{equation}
For each integer $\ell\ge 1$, we also define
\[
\cVbar_\ell=\{x\in\bN\,;\, \alphabar(x)\ge F_\ell\}.
\]
We say that elements $x_1<x_2<\cdots<x_k$ of $\cVbar_\ell$ are 
\emph{consecutive} in $\cVbar_\ell$ if $x_1,x_2,\dots,x_k$ are all
elements of $\cVbar_\ell\cap[x_1,x_k]$.  By construction, we have 
\begin{equation}
 \label{Vbar:eq:bar}
  \iotabar(|v|) = |\iota(v)|, \quad 
  \alphabar(|v|) = |\alpha(v)| \et
  \cVbar_\ell=\{|v|\,;\,v\in\cV_\ell\}
\end{equation}
for each $v\in\prefixes$ and each positive integer $\ell$.

As the table \eqref{Vbar:table} suggests, the map $\alphabar$ has the 
following property.

\begin{lemma}
\label{Vbar:lemma1}
Let $x\in\bN$ with $x\ge 2$.  Then, $\alphabar(x)\ge 2$.
\end{lemma}

\begin{proof}
It suffices to show that $\iotabar(x)\ge 2$.  This is clear if $x\in\cFbar$ because
then $\iotabar(x)=x\ge 2$.  Otherwise, we have $F_i<x<F_{i+1}$ for some
$i\ge 3$, and then $\iotabar(x)=x-2F_{i-2}>F_i-2F_{i-2}=F_{i-3}\ge 1$, thus
$\iotabar(x)\ge 2$.
\end{proof}

\begin{lemma}
\label{Vbar:lemma2}
We have $\cVbar_1=\bN\setminus\{0\}$ and $\cVbar_2=\bN\setminus\{0,1\}$.
\end{lemma}

\begin{proof}
By Lemma \ref{Vbar:lemma1}, we have 
$\bN\setminus\{0,1\} \subseteq \cVbar_2 \subseteq \cVbar_1 \subseteq \bN$.  
The result follows since 
$0\notin\cVbar_1$ while $1\in\cVbar_1\setminus\cVbar_2$.
\end{proof}

\begin{lemma}
\label{Vbar:lemma3}
Let $\ell\ge 2$ be an integer.  Then, the smallest element of $\cVbar_\ell$ is
$F_\ell$. Moreover, for each integer $i\ge \ell$, the map $\iotabar$ restricts to an 
order preserving bijection from $\cVbar_\ell \cap \left]F_i,F_{i+1}\right[$ to 
$\cVbar_\ell \cap \left]F_{i-3},F_{i-1}+F_{i-3}\right[$.
\end{lemma}

\begin{proof}
For each $x\in\cVbar_\ell$, we have $x\ge \alphabar(x)\ge F_\ell$.  Thus, 
$F_\ell$ is the smallest element of $\cVbar_\ell$.   Let $i\ge \ell$ be an 
integer.  By construction, the map $\iotabar$ restricts to an order 
preserving bijection
\[
\begin{array}{rcl}
\bN \cap \left]F_i,F_{i+1}\right[ 
 &\longrightarrow  &\bN \cap \left]F_{i-3},F_{i-1}+F_{i-3}\right[ \\
 x &\longmapsto &x-2F_{i-2}
\end{array}
\]
(with empty domain and codomain if $i=\ell=2$).
This yields the last assertion because $\iotabar^{-1}(\cVbar_\ell)=\cVbar_\ell$. 
\end{proof}

\begin{lemma}
\label{Vbar:lemma4}
Let $\ell\ge 2$ be an integer.  Then, $\cVbar_\ell \cap \left[0,F_{\ell+3}\right]$ 
consists of the 7 numbers
\[
 F_\ell < F_{\ell+1} < F_{\ell+1}+F_{\ell-1} < F_{\ell+2} < F_{\ell+2}+F_{\ell-2}
  < F_{\ell+2}+F_{\ell} < F_{\ell+3}.
\]
\end{lemma}

\begin{proof}
According to Lemma \ref{Vbar:lemma3}, we have
\[
 \cVbar_\ell \cap \left[0,F_{\ell}\right] = \{F_\ell\},
\]
and $\iotabar$ maps 
$\cVbar_\ell \cap \left]F_{\ell}, F_{\ell+1}\right[$ bijectively to 
$\cVbar_\ell \cap \left]F_{\ell-3}, F_{\ell-1}+F_{\ell-3}\right[ 
=\emptyset$, thus  
\[
 \cVbar_\ell \cap \left]F_{\ell}, F_{\ell+1}\right] = \{F_{\ell+1}\}.
\]
Similarly, $\iotabar$ maps 
$\cVbar_\ell \cap \left]F_{\ell+1}, F_{\ell+2}\right[$ bijectively to 
$\cVbar_\ell \cap \left]F_{\ell-2}, F_{\ell}+F_{\ell-2}\right[ = \{F_\ell\}$, 
thus
\[
 \cVbar_\ell \cap \left]F_{\ell+1}, F_{\ell+2}\right]
   = \{2F_{\ell-1}+F_{\ell}, F_{\ell+2}\}
   = \{F_{\ell+1}+F_{\ell-1}, F_{\ell+2}\}.
\]
Finally, $\iotabar$ maps 
$\cVbar_\ell \cap \left]F_{\ell+2}, F_{\ell+3}\right[$ bijectively to 
$\cVbar_\ell \cap \left]F_{\ell-1}, F_{\ell+1}+F_{\ell-1}\right[ = \{F_\ell, F_{\ell+1}\}$, 
thus
\[
 \cVbar_\ell \cap \left]F_{\ell+2}, F_{\ell+3}\right]
   =\{3F_\ell, 2F_\ell+F_{\ell+1}, F_{\ell+3}\} 
   =\{F_{\ell+2}+F_{\ell-2}, F_{\ell+2}+F_\ell, F_{\ell+3}\}.
   \qedhere
\]
\end{proof}

The next lemma presents a central property of the map $\alphabar$.  
It is the key to the proof of Theorem~\ref{main:P:thm1} in 
section~\ref{sec:proof1}.  Below, we use it to derive several 
properties of the sets $\cVbar_\ell$. 

\begin{lemma}
\label{Vbar:lemma6}
Let $k\ge 3$ be an integer.  We have 
\[
 \alphabar(x)\le \alphabar(x+2F_{k-2}) \quad
 \text{for each $x\in \bN  \cap\left[0, F_{k-1}+F_{k-3}\right]$,}
\]
with equality if $x\notin\cFbar$ and $x\neq F_{k-1}+F_{k-3}$.
\end{lemma}

\begin{proof}
We proceed by induction on $k$.  If $k=3$, the range for $x$ is 
$\{0,1,2,3\}$, contained in $\cFbar$, and, using the table \eqref{Vbar:table}, 
we find, as claimed, that $\alphabar(x)\le\alphabar(x+2)$ for each $x$ 
in that range.  Now, suppose that $k\ge 4$.  Let 
$x\in \bN\cap[0,F_{k-1}+F_{k-3}]$, and set $y=x+2F_{k-2}$.   We note that
\begin{itemize}
\item[(i)] if $x=F_{k-3}$, then $y=F_k$, so
 $\alphabar(y)=F_k > F_{k-3}=\alphabar(x)$;
\smallskip
\item[(ii)] if $x=F_{k-1}+F_{k-3}$, then $y=F_{k+1}$, so
 $\alphabar(y)=F_{k+1} > F_{k-2}=\alphabar(x)$;
\smallskip
\item[(iii)] if $x \in \left]F_{k-3},F_{k-1}+F_{k-3}\right[$, then 
 $y\in \left]F_k, F_{k+1}\right[$\/, so $\iotabar(y)=x$,
 thus $\alphabar(y)=\alphabar(x)$.
\end{itemize} 
So, we may assume that $x\in \left[0, F_{k-3}\right[$.  Then,
$y\in \left]F_{k-1}, F_k\right[$ and thus
\[
 \iotabar(y)=y-2F_{k-3}=x+2F_{k-4}
 \quad \Longrightarrow \quad
 \alphabar(y)=\alphabar(x+2F_{k-4}).
\]
If $k=4$, we are done because $x=0\in\cF$ and 
$\alphabar(y)=2 > \alphabar(x)=0$.  Hence, we may also assume that 
$k\ge 5$.  Finally, as $x \in \left[0,F_{k-3}+F_{k-5}\right[$, 
we may further assume, by induction, that 
$\alphabar(x)\le \alphabar(x+2F_{k-4}) = \alphabar(y)$ 
with equality if $x\notin \cF$, and we are done.
\end{proof}

\begin{proposition}
\label{Vbar:prop1}
Let $k,\ell\in\bN$ with $k-2\ge\ell\ge 1$.  Then, we have a 
bijection
\[
\begin{array}{rcl}
 \cVbar_\ell\cap[F_\ell,F_{k-1}+F_{k-3}] &\longrightarrow 
                            &\cVbar_\ell\cap[F_\ell+2F_{k-2},F_{k+1}] \\
 x &\longmapsto &x+2F_{k-2}.
\end{array}
\]
\end{proposition}

\begin{proof}
Set $A=[F_\ell,F_{k-1}+F_{k-3}]$ and $B=2F_{k-2}+A =[F_\ell+2F_{k-2},F_{k+1}]$.  
By Lemma \ref{Vbar:lemma6}, translation by $2F_{k-2}$ maps $\cVbar_\ell\cap A$
injectively into $\cVbar_\ell\cap B$.  To prove that it is surjective, choose 
$y\in\cVbar_\ell\cap B$ and let $x\in \bN\cap A$ such that $y=x+2F_{k-2}$.
If $x\in\cFbar$, then $\alphabar(x)=x\ge F_\ell$.  If $x=F_{k-1}+F_{k-3}$ and 
$x\notin\cFbar$, then $k\ge 4$ and $\alphabar(x)=F_{k-2}\ge F_\ell$.  
For any other value of $x$, Lemma~\ref{Vbar:lemma6} yields 
$\alphabar(x)=\alphabar(y)\ge F_\ell$.  Thus, $x\in\cVbar_\ell\cap A$
in all cases.   
\end{proof}

\begin{proposition}
\label{Vbar:prop2}
Let $\ell\ge 3$ be an integer and let $x_1=F_\ell<x_2<x_3<\cdots$ be the 
elements of $\cVbar_\ell$ listed in increasing order.  Then, the
sequence
\begin{equation}
\label{Vbar:prop2:eq}
(x_{i+1}-x_i)_{i\ge 1}
  = \left(F_{\ell-1},F_{\ell-1},F_{\ell-2},F_{\ell-2},F_{\ell-1},F_{\ell-1},\dots\right)
\end{equation}
is the infinite Fibonacci word constructed on the words 
$(F_{\ell-1},F_{\ell-1})$ and $(F_{\ell-2},F_{\ell-2})$.
\end{proposition}

This is also true for $\ell=2$ but not interesting since, 
by Lemma~\ref{Vbar:lemma2}, we have $\cVbar_2=\bN\setminus\{0,1\}$
and so \eqref{Vbar:prop2:eq} is the constant sequence $(1,1,1,\dots)$ 
when $\ell=2$.

\begin{proof}[Proof of Proposition~\ref{Vbar:prop2}]
For each integer $j\ge 1$, let $\lambda(j)$ 
denote the cardinality of the set 
$\cVbar_\ell\cap\left[F_\ell, F_\ell+2F_{\ell+j-2}\right]$, and let
\[
 m_j=(x_{i+1}-x_i)_{1\le i<\lambda(j)}.
\]
Using Lemma \ref{Vbar:lemma4}, we find that $x_1=F_\ell$, 
$\lambda(1)=3$, $\lambda(2)=5$,
\begin{equation}
\label{Vbar:prop2:eq1}
 m_1=(F_{\ell-1}, F_{\ell-1}) 
 \et
 m_2=(F_{\ell-1},F_{\ell-1},F_{\ell-2}, F_{\ell-2}).
\end{equation}
For $j\ge 1$, Proposition~\ref{Vbar:prop1} applied with $k=\ell+j+1$ 
yields a bijection
\begin{equation}
\label{Vbar:prop2:eq2}
\begin{array}{rcl}
 \cVbar_\ell \cap \left[F_\ell, F_{\ell}+2F_{\ell+j-2}\right] 
            &\longrightarrow 
             &\cVbar_\ell\cap\left[F_\ell+2F_{\ell+j-1}, F_\ell+2F_{\ell+j}\right] \\
 x &\longmapsto &x+2F_{\ell+j-1}.
\end{array}
\end{equation}
Since $F_\ell\in\cVbar_\ell$, this implies that $F_\ell+2F_{\ell+j-1}\in\cVbar_\ell$
for each $j\ge 1$.  This also extends to $j=0$ since $F_\ell+2F_{\ell-1}
=F_{\ell+1}+F_{\ell-1}\in\cVbar_\ell$.  Thus, for each $j\ge 1$, we have  $F_\ell+2F_{\ell+j-2}\in\cVbar_\ell$, so  
\[
 x_{\lambda(j)}=F_\ell+2F_{\ell+j-2},
\]
and the bijection \eqref{Vbar:prop2:eq2} amounts to
\[
(x_i+2F_{\ell+j-1})_{1\le i\le \lambda(j)}
 = (x_i)_{\lambda(j+1)\le i\le \lambda(j+2)}.
\]
In turn, this implies that
\[
(x_{i+1}-x_i)_{1\le i< \lambda(j)}
 = (x_{i+1}-x_i)_{\lambda(j+1)\le i< \lambda(j+2)},
\]
which translates into the relation
\[ 
 m_{j+2}=m_{j+1}m_j  \quad (j\ge 1).
\] 
The conclusion follows from this recurrence relation together with
\eqref{Vbar:prop2:eq1}.
\end{proof}

\begin{corollary}
\label{Vbar:prop2:cor1}
Let $\ell\ge 3$ be an integer.  For any pair of consecutive numbers 
$x<y$ in $\cVbar_\ell$, we have $y-x\in\{F_{\ell-2},F_{\ell-1}\}$.  There are 
infinitely many arithmetic progressions of length $3$ with difference $F_{\ell-2}$
made of consecutive numbers in $\cVbar_\ell$, but no longer ones.  
There are infinitely many arithmetic progressions of length $5$ with 
difference $F_{\ell-1}$ made of consecutive numbers in $\cVbar_\ell$, 
but no longer ones. 
\end{corollary}

\begin{proof}
The first assertion is a direct consequence of the proposition.  The others
follow from the proposition and the fact that the infinite Fibonacci word 
$f_{a,b}=abaababa\cdots$ on two letters $a$ and $b$ contains infinitely 
many letters $b$ but no factor of the form $bb$, and that it contains 
infinitely factors $aa$ but no factor $aaa$.  Thus, the infinite Fibonacci word 
on $(F_{\ell-1},F_{\ell-1})$ and $(F_{\ell-2},F_{\ell-2})$ contains infinitely many
subsequences of two consecutive terms equal to $F_{\ell-2}$ but no longer 
ones, and infinitely many subsequences of four consecutive terms equal to
$F_{\ell-1}$ but no longer ones.
\end{proof}

\begin{corollary}
\label{Vbar:prop2:cor2}
Let $x<y$ be consecutive elements of $\cVbar_\ell$ for some integer 
$\ell\ge 2$.  Then, 
\[
 \{x,y\}\not\subseteq\cVbar_{\ell+2}, \quad
 \{x,y\}\cap\cVbar_{\ell+1}\neq\emptyset \et
 |\alphabar(y)-\alphabar(x)| \ge y-x.
\]
\end{corollary}

\begin{proof}
If $\ell\ge 3$, Corollary~\ref{Vbar:prop2:cor1} yields $y-x\in\{F_{\ell-2}, F_{\ell-1}\}$.
This also holds if $\ell=2$ because then $y-x=1$ by Lemma~\ref{Vbar:lemma2}.  
Since  $y-x\le F_{\ell-1}< F_{\ell}$, it follows from 
Corollary~\ref{Vbar:prop2:cor1} that $x<y$ are not consecutive elements 
of $\cVbar_{\ell+2}$, thus $\{x,y\}\not\subseteq\cVbar_{\ell+2}$.  

Suppose that $\{x,y\}\cap\cVbar_{\ell+1}=\emptyset$.  As the
first two elements of $\cVbar_\ell$ are $F_\ell$ and $F_{\ell+1}$, we must have
that $x>F_{\ell+1}$ and so there are consecutive $x'<y'$ in 
$\cVbar_{\ell+1}$ such that $x'<x<y<y'$.  By the above, the differences 
$x-x'$, $y-x$ and $y'-y$ are all bounded below by $F_{\ell-2}$.  Moreover,
at least one of them is bounded below by $F_{\ell-1}$ because, otherwise,
we would have $F_{\ell-1}>F_{\ell-2}$, thus $\ell\ge 3$, and $x'<x<y<y'$
would be $4$ consecutive elements of $\cVbar_\ell$ in arithmetic 
progression with difference $F_{\ell-2}$, against 
Corollary~\ref{Vbar:prop2:cor1}.  We conclude that 
$y'-x'\ge 2F_{\ell-2}+F_{\ell-1}>F_\ell$.  This is impossible since the same 
corollary yields $y'-x'\le F_\ell$.  Thus, $\{x,y\}\cap\cVbar_{\ell+1}\neq\emptyset$. 

Let $m$ be the largest integer such that $\{x,y\}\subseteq\cVbar_m$.
By the above, we have $m\in\{\ell,\ell+1\}$,  and the preceeding reasoning 
yields $\{x,y\}\cap\cVbar_{m+1}\neq\emptyset$.  Thus $\alphabar(x)$
and $\alphabar(y)$ are distinct Fibonacci numbers, the smallest of 
which is equal to $F_\ell$ or $F_{\ell+1}$. Consequently, we have 
$| \alphabar(y)-\alphabar(x)|\ge F_{\ell-1}\ge y-x$.
\end{proof}

The next result characterizes the elements of $\cVbar_{\ell+1}$ within 
$\cVbar_\ell$, for any integer $\ell\ge 3$.

\begin{proposition}
\label{Vbar:prop3}
Let $x<y<z$ be consecutive numbers in $\cVbar_\ell$ for some integer 
$\ell\ge 3$.  Then, we have
\begin{equation}
\label{Vbar:prop3:eq0}
 y\notin\cVbar_{\ell+1} 
 \quad\Longleftrightarrow\quad
 y-x\neq z-y.
\end{equation}
When these equivalent conditions hold, we further have $z-x=\alphabar(y)=F_\ell$.
\end{proposition}

Note that \eqref{Vbar:prop3:eq0} is false for $\ell=2$.  A counterexample 
is provided by any triple of consecutive integers $x<y<z$ with $x>1$ and 
$\alphabar(y)=2$.  For then, $x<y<z$ are consecutive elements of $\cVbar_2$
with $y\notin\cVbar_3$ and $y-x=z-y=1$.

\begin{proof}
Let $(x_i)_{i\ge 1}$ denote the sequence of elements of $\cVbar_\ell$ listed 
in increasing order, and let  $(y_i)_{i\ge 1}$ denote the subsequence obtained 
by removing its first term $x_1$ and each $x_i$ with $i\ge 2$ such 
that $x_i-x_{i-1}\neq x_{i+1}-x_i$.  Since $x_1=F_\ell\notin\cVbar_{\ell+1}$, 
proving the first part of the proposition amounts to showing that 
$(y_i)_{i\ge 1}$ lists the elements of $\cVbar_{\ell+1}$ in increasing order.  

By Proposition~\ref{Vbar:prop2}, we have $x_2-x_1=x_3-x_2=F_{\ell-1}$, thus
$y_1=x_2=F_{\ell+1}$ is indeed the smallest element of $\cVbar_{\ell+1}$. 
To go further, set 
\[
 X=(x_{i+1}-x_i)_{i\ge 1}, \quad 
 X'=(x_{i+1}-x_i)_{i\ge 2} \et
 Y=(y_{i+1}-y_i)_{i\ge 1}.
\]
By Proposition~\ref{Vbar:prop2}, we have $X=\lim_{i\to\infty}u_i$ where 
$(u_i)_{i\ge 1}$ is the Fibonacci sequence in $\{F_{\ell-2},F_{\ell-1}\}^*$ 
starting with
\[
 u_1=F_{\ell-1}F_{\ell-1} \et u_2=F_{\ell-1}F_{\ell-1}F_{\ell-2}F_{\ell-2}.
\]
Since each $u_i$ starts with $F_{\ell-1}$, we can write
\[
 u_iF_{\ell-1}=F_{\ell-1}v_i
\]
for some $v_i\in\{F_{\ell-2},F_{\ell-1}\}^*$.  Then, $(v_i)_{i\ge 1}$ is 
the Fibonacci sequence in $\{F_{\ell-2},F_{\ell-1}\}^*$ with
\[
 v_1=F_{\ell-1}F_{\ell-1} \et v_2=F_{\ell-1}F_{\ell-2}F_{\ell-2}F_{\ell-1}.
\] 
We deduce that $X'=\lim_{i\to\infty}v_i = v_2v_1v_2v_2v_1\cdots$ 
is the infinite Fibonacci word on $v_2$ and $v_1$.  Hence, the factors 
$F_{\ell-1}F_{\ell-2}$ and $F_{\ell-2}F_{\ell-1}$ do not overlap in 
$X'$, and so $Y$ is the infinite word on $\{F_{\ell-1},F_\ell\}$ 
obtained from $X'$ by replacing each of these factors by $F_\ell$.  
Thus, $Y$ is the infinite Fibonacci word on $F_\ell F_\ell$ 
and $F_{\ell-1}F_{\ell-1}$.  As $y_1=F_{\ell+1}$, we conclude from 
Proposition~\ref{Vbar:prop2} that $(y_i)_{i\ge 1}$ lists the
elements of $\cVbar_{\ell+1}$ in increasing order. 

The second part of the proposition follows from this.  Indeed, when 
both conditions hold in \eqref{Vbar:prop3:eq0}, we have $\alphabar(y)=F_\ell$
because $y\in\cVbar_\ell\setminus\cVbar_{\ell+1}$, and $z-x=F_\ell$ since
$(y-x,z-y)$ is a permutation of $(F_{\ell-2},F_{\ell-1})$.
\end{proof}

We conclude this section with two additional observations.

\begin{proposition}
\label{Vbar:prop4}
Let $\ell\ge 2$ be an integer and let $x<z$ be positive integers.  Then, the 
following conditions are equivalent.
\begin{itemize}
\item[(i)] $x<z$ are consecutive elements of\/ $\cVbar_{\ell+1}$ but not
 consecutive elements of\/ $\cVbar_{\ell}$;
 \smallskip
 \item[(ii)] there exists $y\in\cVbar_\ell\setminus\cVbar_{\ell+1}$ such that
 $x<y<z$ are consecutive elements of\/ $\cVbar_{\ell}$;
 \smallskip
 \item[(iii)] $x<z$ are consecutive elements of\/ $\cVbar_{\ell+1}$ and 
 $z-x=F_\ell$.
\end{itemize}
\end{proposition}

\begin{proof}
(i) $\Rightarrow$ (ii): Suppose that (i) holds, and let $y$ be the successor 
of $x$ in $\cVbar_\ell$.  Since $x<y<z$, this number $y$ does not belong 
to $\cVbar_{\ell+1}$.  So, by Corollary~\ref{Vbar:prop2:cor2},
its successor in $\cVbar_\ell$ belongs to $\cVbar_{\ell+1}$, hence it must 
be $z$.  So (ii) holds. 

(ii) $\Rightarrow$ (iii): Under the condition (ii), Corollary~\ref{Vbar:prop2:cor2}
implies that $\{x,z\}\subseteq\cVbar_{\ell+1}$.  Thus, $x<z$ are 
consecutive in $\cVbar_{\ell+1}$.  Moreover, if $\ell\ge 3$,
Proposition~\ref{Vbar:prop3} yields $z-x=F_{\ell}$.  This still holds when 
$\ell=2$ for then $x<y<z$ are consecutive integers and so $z-x=2=F_\ell$.

(iii) $\Rightarrow$ (i): If (iii) holds, then $z-x=F_\ell\notin\{F_{\ell-2},F_{\ell-1}\}$.
So, $x<z$ cannot be consecutive in $\cVbar_\ell$ by 
Corollary~\ref{Vbar:prop2:cor1} if $\ell\ge 3$, and by 
Lemma~\ref{Vbar:lemma2} if $\ell=2$.
\end{proof}

\begin{proposition}
\label{Vbar:prop5}
Let $\ell\ge 2$ be an integer, and let $x<y<z$ be consecutive numbers 
in $\cVbar_\ell$, not all contained in $\cVbar_{\ell+1}$.  Then, exactly
one of the following situations holds:\\[5pt]
\quad
$\disp 
 \begin{array}{rlllll}
 \mathrm{(i)}  &\alphabar(x)\ge F_{\ell+1}, &\alphabar(y)=F_{\ell}, 
  &\alphabar(z)\ge F_{\ell+1} &\text{and} &z-x=F_\ell;\\[5pt]
 \mathrm{(ii)} &\alphabar(x)=F_{\ell}, &\alphabar(y)\ge F_{\ell+1}, 
  &\alphabar(z)=F_{\ell} &\text{and} &y-x=z-y;\\[5pt]
 \mathrm{(iii)} &\alphabar(x)=F_{\ell}, &\alphabar(y)=F_{\ell+1}, 
  &\alphabar(z)\ge F_{\ell+2} &\text{and} &y-x=z-y=F_{\ell-1};\\[5pt]
 \mathrm{(iv)} &\alphabar(x)\ge F_{\ell+2}, &\alphabar(y)=F_{\ell+1}, 
  &\alphabar(z)=F_{\ell} &\text{and} &y-x=z-y=F_{\ell-1}.
\end{array}$
\end{proposition}

\begin{proof}
We first note that the four conditions (i) to (iv) are mutually exclusive.
So, we only have to show that they exhaust all possibilities.
We also note that, if $\ell=2$, then $x<y<z$ are consecutive integers,
thus $y-x=z-y=1=F_1$ and $z-x=2=F_2$.

If $y\notin\cVbar_{\ell+1}$, then case (i) applies by 
Proposition~\ref{Vbar:prop4}.

Suppose from now on that $y\in\cVbar_{\ell+1}$.  If $\ell\ge 3$, 
we have
\[
 y-x=z-y\in\{F_{\ell-2},F_{\ell-1}\}
\]
by Corollary~\ref{Vbar:prop2:cor1} and Proposition~\ref{Vbar:prop3}.
This is also true for $\ell=2$ as noted above.  Thus, if 
$x\notin\cVbar_{\ell+1}$ and $z\notin\cVbar_{\ell+1}$, then
case (ii) applies.  

Suppose that $x\notin\cVbar_{\ell+1}$ and $z\in\cVbar_{\ell+1}$. 
Then $y<z$ are consecutive both in $\cVbar_\ell$ and in $\cVbar_{\ell+1}$.
So, Corollary~\ref{Vbar:prop2:cor1} yields $z-y=F_{\ell-1}$.  
We also note that $x\neq F_\ell$ because 
otherwise, by Lemma~\ref{Vbar:lemma4}, we would have $y=F_{\ell+1}$ and 
$z=F_{\ell+1}+F_{\ell-1}$, thus $\alphabar(z)=F_\ell$ against the 
hypothesis that $z\in\cVbar_{\ell+1}$.  Hence, $x$ admits a predecessor 
$x'$ in $\cVbar_\ell$.  Since $x\notin\cVbar_{\ell+1}$, 
Corollary~\ref{Vbar:prop2:cor2} shows that $x'\in\cVbar_{\ell+1}$, 
thus $x'<y<z$ are consecutive elements of $\cVbar_{\ell+1}$ 
with $y-x'>y-x=z-y$.  By Proposition~\ref{Vbar:prop3}, this implies that
$y\notin\cV_{\ell+2}$ and so $z\in\cV_{\ell+2}$ by 
Corollary~\ref{Vbar:prop2:cor2}.  Thus, case (iii) applies.

Finally, suppose that $x\in\cVbar_{\ell+1}$.  Then the hypothesis implies that
$z\notin\cVbar_{\ell+1}$, and we argue similarly as above. 
As $x<y$ are consecutive in $\cVbar_{\ell+1}$, we have $y-x=F_{\ell-1}$.  
Let $z'$ be the successor of $z$ in $\cVbar_\ell$.  Since 
$z\notin\cVbar_{\ell+1}$, we have $z'\in\cVbar_{\ell+1}$, thus $x<y<z'$ are 
consecutive in $\cVbar_{\ell+1}$ with $y-x=z-y<z'-y$.  
This implies that $y\notin\cVbar_{\ell+2}$ and so $x\in\cVbar_{\ell+2}$.  
Thus, case (iv) applies.
\end{proof}

%
%

\section{The maps $\iota$ and $\alpha$, and the sets $\cV_\ell$}
\label{sec:V}

Let the notation be as in section~\ref{main:ssec:V}.   In this 
section, we give an explicit description of the map 
$\iota\colon\prefixes\to\prefixes$.  Then, we use it to refine some 
of the results of the previous section, and we prove the last assertion 
of section~\ref{main:ssec:V}.

Before doing this, we note that, by Lemma~\ref{Vbar:lemma2}, we have 
$\cV_1=\left[w_1,w_\infty\right[$ and $\cV_2=\left[w_2,w_\infty\right[$.
We also note the following direct consequence of 
Lemma~\ref{Vbar:lemma4}.

\begin{lemma}
\label{V:lemma1}
Let $\ell\ge 3$ be an integer.  Then, $\cV_\ell \cap \left[\epsilon,w_{\ell+3}\right]$ 
consists of the 7 words
\[
 w_\ell < w_{\ell+1} < w_{\ell+1}w_{\ell-1} < w_{\ell+2} < w_{\ell+2}w_{\ell-2}
  < w_{\ell+2}w_{\ell} < w_{\ell+3}.
\]
\end{lemma}

For each non-empty word $v$ in $E^*$,
we denote by $v^*$ the word $v$ deprived of its last letter, that is
the prefix of $v$ of length $|v|-1$.  Induction based on the recurrence 
formula $w_{i+2}=w_{i+1}w_i$ ($i\ge 1$) shows that
\begin{equation}
\label{V:eq:f}
 w_i=w_i^{**}f_i
 \quad\text{where}\quad
 f_i=\begin{cases}  ab &\text{if $i\ge 2$ is even,}\\ 
                              ba &\text{if $i\ge 2$ is odd.}     \end{cases}
\end{equation}
For each $i\ge 2$, we define $\tw_i$ to be the word obtained from $w_i$
by permuting its last two letters.  In view of the above, this means that
\begin{equation}
\label{V:eq:tw}
 \tw_i=w_i^{**}f_{i+1} \quad (i\ge 2).
\end{equation}
It is a well known fact, attributed to J.~Berstel, that $w_i^{**}$ is a palindrome 
for each $i\ge 2$.  This has the following useful consequence. 

\begin{lemma}
\label{V:lemma2}
For each integer $i\ge 2$, the word $w_i^{**}$ is a palindrome and we 
have 
\begin{equation}
\label{V:lemma2:eq}
 w_{i+1}=w_{i-1}\tw_{i}.
\end{equation}
\end{lemma}

\begin{proof}
We proceed by induction on $i$.  We first note that $w_2^{**}=\epsilon$ and
$w_3^{**}=a$ are palindromes, that $w_3=aba=w_1\tw_2$,  and that 
$w_4=abaab=w_2\tw_3$.  Suppose that, for some $i\ge 3$, 
the words $w_{i-1}^{**}$ and $w_i^{**}$ are palindromes
and that \eqref{V:lemma2:eq} holds.  This yields 
$w_{i+1}^{**}=w_{i-1}^{**}f_{i-1}w_{i}^{**}$.  
Thus, the reverse of $w_{i+1}^{**}$ is 
$w_{i}^{**}f_{i}w_{i-1}^{**}=w_{i}w_{i-1}^{**}=w_{i+1}^{**}$, showing that 
$w_{i+1}^{**}$ is a palindrome.  Moreover, since \eqref{V:lemma2:eq}  
yields $\tw_{i+1}=w_{i-1}w_{i}$, we also find that 
$w_{i+2}=w_{i+1}w_{i}=w_{i}w_{i-1}w_{i}=w_{i}\tw_{i+1}$, which 
completes the induction step.
\end{proof}

It follows from Lemma~\ref{V:lemma2} that, for 
each $i\ge 2$, we have $w_{i+1}^{**}=w_{i-1}w_{i}^{**}$, hence 
$w_i^{**}$ is the largest prefix $v$ of $w_\infty$ such that 
$w_{i-1}v$ is a prefix of $w_\infty$.

\begin{lemma}
 \label{V:lemma3}
Let $k\ge 4$ be an integer.  Then,  
\begin{equation}
\label{V:lemma3:eq}
 \begin{array}{rcl}
  \left]w_k,w_{k+1}\right[ &\longrightarrow &\left]w_{k-3},w_{k-1}w_{k-3}\right[\\
  w &\longmapsto &\iota(w).
 \end{array}
\end{equation}
is an order preserving bijection. 
For $w\in\left]w_k,w_{k+1}\right[$, we can compute $\iota(w)$ as follows.
\begin{itemize}
\item[(i)] If $w < w_kw^*_{k-2}$, then $w=w_ku$ for some
  $u\in \left] \epsilon, w^{**}_{k-2} \right]$ and $\iota(w)=w_{k-3}u$.
  \smallskip
\item[(ii)] If $w = w_kw^*_{k-2}$, then $\iota(w)=w^*_{k-1}$.
  \smallskip
\item[(iii)] If $w > w_kw^*_{k-2}$, then $w=w_kw_{k-2}u$
  for some $u\in \left[ \epsilon, w_{k-3} \right[$ and $\iota(w)=w_{k-1}u$.
\end{itemize}
\end{lemma}

\begin{proof}
For each $w\in\left]w_k,w_{k+1}\right[$, we have 
$|\iota(w)|=\iotabar(|w|)=|w|-2F_{k-2}$.  So, \eqref{V:lemma3:eq}
is an order preserving bijection.  To justify the explicit formulas, it 
suffices to note that, in each case, the value given for $\iota(w)$ 
is a prefix of $w_\infty$ of length $|w|-2F_{k-2}$.  This is a direct
computation for the length.  That it is a prefix is immediate
in cases (ii) and (iii).  In case (i), it follows from the observation made
right before the lemma.
\end{proof}

\begin{lemma}
 \label{V:lemma4}
We have
$\{\alpha(w^*_k), \alpha(w^*_{k+1}), \alpha(w_kw^*_{k-2}) \} 
\subseteq \{w_2, w_3\}$ for each integer $k\ge 3$.
\end{lemma}

\begin{proof}
We proceed by induction.
Table \eqref{main:table} gives $\alpha(w^*_3) =\alpha(w^*_4)=w_2$
and $\alpha(w_3w^*_1)=w_3$.  Suppose that $\alpha(w^*_{k-1})$, 
$\alpha(w^*_k)$ and $\alpha(w_{k-1}w^*_{k-3})$ lie in $\{w_2, w_3\}$ 
for some $k\ge 4$.  Lemma \ref{V:lemma3} gives
\[
 \iota(w^*_{k+1})=w_{k-1}w_{k-3}^*
 \et
 \iota(w_kw^*_{k-2})=w^*_{k-1}.
\]
Thus, $\alpha(w^*_{k+1})=\alpha(w_{k-1}w_{k-3}^*)$ and 
$\alpha(w_kw^*_{k-2})=\alpha(w^*_{k-1})$ also belong to $\{w_2, w_3\}$.
\end{proof}

\begin{proposition}
\label{V:prop1}
Let $\ell\ge 4$ be an integer, and let $u<v$ be consecutive elements 
of $\cV_\ell$.  Then, we have $v=us$ for some suffix 
$s$ in the set $\cS_\ell:=\{w_{\ell-2},\, \tw_{\ell-2},\, w_{\ell-1},\, \tw_{\ell-1}\}$.
\end{proposition}

\begin{proof}
By Lemma~\ref{V:lemma1}, the smallest element of $\cV_\ell$ is $w_\ell$
and the next six elements of $\cV_\ell$ are obtained  by multiplying $w_\ell$
on the right successively by $w_{\ell-1}$, $w_{\ell-1}$, $w_{\ell-2}$, 
$w_{\ell-2}$, $\tw_{\ell-1}$ and $w_{\ell-1}$.  Thus, we may 
assume that $v>w_{\ell+3}$ and so $v\in\left]w_k,w_{k+1}\right]$ for 
some integer $k\ge \ell+3$.  Since $w_k\in\cV_\ell$, we deduce that
$\{u,v\}\subseteq \left[w_k,w_{k+1}\right]$.  Since $w_{k-3}$ and 
$w_{k-1}w_{k-3}$ belong to $\cV_\ell$, the map 
\[
\begin{array}{rcl}
 \pi\colon \cV_\ell \cap \left[w_k,w_{k+1}\right] 
   &\longrightarrow &\cV_\ell \cap \left[w_{k-3},w_{k-1}w_{k-3}\right]\\
  w &\longmapsto
 &\pi(w) 
  = \begin{cases} 
    w_{k-3} &\text{if $w=w_k$,}\\
    \iota(w) &\text{if $w_k<w<w_{k+1}$,}\\
    w_{k-1}w_{k-3} &\text{if $w=w_{k+1}$.}
      \end{cases}
\end{array}
\]
is an order preserving bijection.  In particular, $\pi(u)<\pi(v)$ are 
consecutive elements of $\cV_\ell$ and so we may assume, by induction
on $|v|$, that 
\begin{equation}
\label{V:prop1:eq}
 \pi(v)=\pi(u)s 
 \quad\text{for some $s \in \cS_\ell$.}  
\end{equation}
 We consider three cases.

(i) Suppose that $v<w_kw_{k-2}$.  Since $\ell\ge 4$, Lemma~\ref{V:lemma4}
shows that $w_kw^*_{k-2}\notin\cV_\ell$.  So, we have 
$u=w_ku'$ and $v=w_kv'$ for some $u',v'$ in $[\epsilon, w^{**}_{k-2}]$.
Then Lemma~\ref{V:lemma3} gives $\pi(u)=w_{k-3}u'$ and $\pi(v)=w_{k-3}v'$.
By \eqref{V:prop1:eq}, this yields $v'=u's$, thus $v=us$.

(ii) Suppose that $v>w_kw_{k-2}$.  Then, we have $u\ge w_kw_{k-2}$. So, 
$u=w_kw_{k-2}u'$ and $v=w_kw_{k-2}v'$ for some  
$u',v'$ in $[\epsilon, w_{k-3}]$.  Applying Lemma~\ref{V:lemma3} gives 
$\pi(u)=w_{k-1}u'$ and $\pi(v)=w_{k-1}v'$.  By \eqref{V:prop1:eq}, this implies 
that $v'=u's$, thus $v=us$.  

(iii) Finally, suppose that $v=w_kw_{k-2}$.  Since 
$w_kw^*_{k-2}\notin\cV_\ell$, we have $u=w_ku'$ for some word 
$u' \in [\epsilon, w^{**}_{k-2}]$.  Then, Lemma~\ref{V:lemma3} gives 
$\pi(u)=w_{k-3}u'$ and $\pi(v)=w_{k-1}$.  Thus, $\pi(v)=w_{k-3}\tw_{k-2}$ 
by Lemma~\ref{V:lemma2}.   By \eqref{V:prop1:eq}, we deduce that
$\tw_{k-2}=u's$, thus $w_{k-2}=u's'$ where $s'$ is the reverse of $s$.
We conclude that $v=w_ku's'=us'$ with $s' \in \cS_\ell$.
\end{proof}

In Section~\ref{main:ssec:Fib}, we defined $\theta$ to be the endomorphism
of $E^*$ determined by $\theta(a)=ab$ and $\theta(b)=a$.  It has the 
following property.

\begin{lemma}
 \label{V:lemma:theta}
For each integer $k\ge 2$, we have $\theta(w_k)=w_{k+1}$ and 
$\theta(\tw_k)=\tw_{k+1}$.  Moreover, $\theta$ restricts to an order 
preserving map from $\left[\epsilon,w_\infty\right[$ to itself.
\end{lemma}

\begin{proof}
The first formula holds for each $k\ge 1$, as stated in \eqref{main:eq:theta}.  
This follows from the fact that $(\theta(w_i))_{i\ge 1}$ is a Fibonacci 
sequence in $E^*$ with 
$\theta(w_1)=\theta(a)=ab=w_2$ and $\theta(w_2)=\theta(ab)=aba=w_3$.
Since $\theta$ preserves the relation of prefix, we deduce that it maps
$[\epsilon,w_k]$ to $[\epsilon,w_{k+1}]$ for each $k\ge 1$.  Thus, it maps
$\left[\epsilon,w_\infty\right[$ to itself.

To prove the second formula, we note that $\theta(ab)=aba$ and 
$\theta(ba)=aab$.  This implies that, for each $k\ge 2$, we have 
$\theta(f_k)=af_{k+1}$, thus 
$w_{k+1}=\theta(w_k)=\theta(w^{**}_k)af_{k+1}$,
and so 
\[
 \tw_{k+1}=\theta(w^{**}_k)af_k=\theta(w^{**}_kf_{k+1})=\theta(\tw_k).
 \qedhere
\]
\end{proof}

In particular, the above lemma implies that, with the notation of 
Proposition~\ref{V:prop1}, we have $\theta(\cS_\ell)=\cS_{\ell+1}$
for each integer $\ell\ge 4$.  We can now prove the last 
assertion of Section~\ref{main:ssec:V}.

\begin{proposition}
\label{V:prop2}
Let $\ell\ge 4$ be an integer.  Then, $\theta$ restricts to an order 
preserving bijection from $\cV_\ell$ to $\cV_{\ell+1}$.
\end{proposition}

\begin{proof}
Let $(v_i)_{i\ge 1}$ be the increasing sequence of elements of $\cV_\ell$, 
and let $(v'_i)_{i\ge 1}$ be that of $\cV_{\ell+1}$.
We need to show that $\theta(v_i)=v'_i$ for each $i\ge 1$.
As $v_i$ and $v'_i$ are prefixes of $w_\infty$, both
$\theta(v_i)$ and $v'_i$ are prefixes of $w_\infty$ by 
Lemma~\ref{V:lemma:theta}.
So, it suffices to show that $|\theta(v_i)|=|v'_i|$.  

For each $i\ge 1$, write $v_{i+1}=v_is_i$ and $v'_{i+1}=v'_is'_i$ for words 
$s_i$ and $s'_i$.   By Proposition~\ref{Vbar:prop2}, the sequence 
$(|s_i|)_{i\ge 1}$ is the infinite Fibonacci word on  $(F_{\ell-1},F_{\ell-1})$ 
and $(F_{\ell-2},F_{\ell-2})$, while $(|s'_i|)_{i\ge 1}$ is the infinite 
Fibonacci word on $(F_{\ell},F_{\ell})$ and $(F_{\ell-1},F_{\ell-1})$.   
Thus, for each $i\ge 1$, there is an integer $k\in\{\ell-2,\ell-1\}$ 
such that $|s_i|=F_k$ and $|s'_i|=F_{k+1}$.  By Proposition~\ref{V:prop1}, 
we have $s_i\in\{w_k,\tw_k\}$ since $|s_i|=F_k$.  Then, 
Lemma~\ref{V:lemma:theta} gives $\theta(s_i)\in\{w_{k+1},\tw_{k+1}\}$,
thus $|\theta(s_i)|=F_{k+1}=|s'_i|$. 

Since $v_1=w_\ell$ and $v'_1=w_{\ell+1}$, we have $\theta(v_1)=v'_1$,
thus $|\theta(v_i)|=|v'_i|$ for $i=1$.  We deduce by induction 
that $|\theta(v_i)|=|v'_i|$ for each $i\ge 1$ because, if this holds for
some $i\ge 1$, then
\[
 |\theta(v_{i+1})|=|\theta(v_i)|+|\theta(s_i)|=|v'_i|+|s'_i|=|v'_{i+1}|. 
 \qedhere
\] 
\end{proof}

\begin{remark}
The statement of the proposition fails for $\ell=2$ and for $\ell=3$.
For example, for $v=w_6w^*_4=w_6w_3a$, we have 
$\theta(v)=w_7w_4ab=w_7w^*_5$.  Using Lemma~\ref{V:lemma3}, 
we find that $\alpha(v)=\iota^2(v)=w_3$ and
$\alpha(\theta(v))=\iota^3(\theta(v))=w_2$.  So, $v\in\cV_3$
while $\theta(v)\notin\cV_3$.  This shows that 
$\theta(\cV_3)\not\subseteq\cV_3$.  A fortiori, we have 
$\theta(\cV_2)\not\subseteq\cV_3$ and  
$\theta(\cV_3)\not\subseteq \cV_4$.
\end{remark}

\begin{corollary}
\label{V:prop2:cor}
For each $v\in\cV_4$, we have $\alpha(\theta(v))=\theta(\alpha(v))$. 
\end{corollary}

\begin{proof}
Let $v\in\cV_4$.  We have $\alpha(v)=w_\ell$ for some integer 
$\ell\ge 4$.  Then $v\in\cV_\ell\setminus\cV_{\ell+1}$, and the
proposition gives $\theta(v)\in\cV_{\ell+1}\setminus\cV_{\ell+2}$.  So,
$\alpha(\theta(v))=w_{\ell+1}=\theta(\alpha(v))$.
\end{proof}


%
%

\section{A partition of the domain of $\uP$}
\label{sec:P} 

In this section, we study the map $\uP$ of section \ref{main:ssec:P}
by decomposing its domain $\cA(\epsilon)$ as a union of closed polygons with
disjoint interiors, and by giving explicit formulas for $\uP$ on each of these
polygons.  To this end, we introduce several definitions.  

We define an \emph{admissible line segment} to be a closed 
line segment in $\bR^2$ with non-empty interior that is parallel to 
$(1,0)$, $(1,1)$ or $(0,1)$,  with end points in $\bZ^2$, i.e.~a set of the form 
\[
 \{(i+t,j)\,;\,t\in I\} \ou \{(i+t,j+t)\,;\,t\in I\} \ou \{(i,j+t)\,;\,t\in I\}
\]
where $(i,j)\in\bZ^2$ and where $I$ is a closed subinterval of $\bR$ 
with non-empty interior, possibly unbounded. 

We define an \emph{admissible polygon} to be a non-empty 
subset of $\bR^2$ which is the closure of its interior, whose
boundary (possibly empty) is a polygonal line made of admissible line 
segments intersecting at most in their end-points.  These line segments 
are the \emph{sides} of the polygon and their end-points are 
its \emph{vertices}. 

We say that two distinct admissible polygons are \emph{compatible} if 
their intersection is either empty, or a common vertex of each, 
or a common side of each.   We define a \emph{polygonal partition}
of an admissible polygon $\cA$ to be a set $S$ of admissible 
pairwise compatible polygons whose union is $\cA$.  This implies
that each edge of $\cA$ is an union of edges of polygons of $S$, however
not necessarily the edge of a single polygon of $S$.
 
For example, the sector $\cA(\epsilon)$ is an admissible polygon
with a single vertex $(0,0)$, a vertical side
$\{(0,t)\,;\,-\infty<t\le 0\}$, and a side $\{(t,t)\,;\,0\le t<\infty\}$
of slope $1$.  The main result of this section is the following.

\begin{theorem}
\label{P:thm}
For each pair of consecutive words $u<v$ in $[\epsilon,w_\infty[$,
the set
\begin{equation}
\label{P:thm:eq1}
 \trap(u,v):=\cA(u)\cap\cC(v) 
\end{equation}
is an unbounded admissible trapeze with $2$ vertices.  For each point
$\uq=(q_1,q_2)$ in it, we have 
\begin{equation}
\label{P:thm:eq2}
 \uP(\uq):=\Phi(P_u(\uq),q_2-1,P_v(\uq))=\Phi(q_1-|u|,q_2-1,|v|).
\end{equation}
For each integer $\ell\ge 2$ and each triple of consecutive words 
$u<v<w$ in $\cV_\ell$ with $v\notin\cV_{\ell+1}$, the set
\begin{equation}
\label{P:thm:eq3}
 \cell(u,v,w):=\cA(u)\cap\cB(v)\cap\cC(w) 
\end{equation}
is a bounded admissible polygon with $4$ or $5$ vertices. 
For each $\uq=(q_1,q_2)$ in it, we have 
\begin{equation}
\label{P:thm:eq4}
 \uP(\uq):=\Phi(P_u(\uq),P_v(\uq),P_w(\uq))=\Phi(q_1-|u|,q_2-|\alpha(v)|,|w|).
\end{equation}
The polygons in \eqref{P:thm:eq1} and \eqref{P:thm:eq3} are all distinct 
and form a partition of $\cA(\epsilon)$.
\end{theorem}

Note that the sets $\trap(u,v)$ in \eqref{P:thm:eq1} and $\cell(u,v,w)$ in
\eqref{P:thm:eq3} are uniquely determined by $v$ which, 
as a prefix of $w_\infty$, is in turn determined by its length.  So, it makes 
sense to denote them respectively as $\cT_i$ and $\cR_i$ where $i=|v|$.   
For $\cT_i$, the range of $i$ is $\bN\setminus\{0\}$, while for $\cR_i$, it is 
$\bN\setminus\cFbar$ (for each $\ell\ge 1$, we cannot have $v=w_\ell$ 
because $w_\ell\in\cV_\ell\setminus\cV_{\ell+1}$ has no predecessor in
$\cV_\ell$).    This shorter notation is used in Figure~\ref{P:fig1} below
to illustrate the partition of $\cA(\epsilon)$ given by the theorem in the
range $0\le q_1\le 26$.

\begin{figure}[ht]
\begin{tikzpicture}[xscale=0.5, yscale=0.4]
       \draw[->, thick] (0,0)--(28,0) node[below]{$q_1$};
       \draw[->, thick] (0,-1.5)--(0,28) node[left]{$q_2$};
       \foreach \x in {2,4,...,26}
       {\draw[-,line width=0.5mm] (\x,0.1)--(\x,-0.1) %
              node[below left]{$\scriptstyle \x$}; }  
       \foreach \y in {0,2,...,26}
         {\draw[-,line width=0.5mm] (0.1,\y)--(-0.1,\y) 
              node[left]{$\scriptstyle \y$};
          \draw[-,dashed] (0,\y)--(\y,\y);}
       \foreach \x in {0,1,3,5,8,13}
         \draw[-,thick] (2*\x,-1.5)--(2*\x,2*\x);
       \foreach \x in {2,4,6,10,12}
         \draw[-,thick] (2*\x,-1.5)--(2*\x,\x+2);
       \foreach \x in {3,7,9}
         \draw[-,thick] (2*\x,-1.5)--(2*\x,\x+3);
      \draw[-,thick] (0,0)--(6,6)--(8,6)--(10,8)--(12,8)--(16,12)--(20,12)
             --(22,14)--(24,14)--(26,16);
      \draw[-,thick] (6,6)--(26,26);
      \draw[-,thick] (22,-1.5)--(22,16);
      \draw[-,thick] (10,10)--(14,10);
      \draw[-,thick] (18,12)--(22,16)--(26,20);
      \draw[-,thick] (16,16)--(22,16);
       \foreach \x in {2,4,6,10,12}
        \node[draw,circle,inner sep=1.4pt, fill] at (2*\x,\x+2){};
       \foreach \x in {3,7,9}
        \node[draw,circle,inner sep=1.4pt, fill] at (2*\x,\x+3){};
       \foreach \x in {5,11}
        \node[draw,circle,inner sep=1.4pt, fill] at (2*\x,\x+5){};
       \foreach \x in {0,1,8,13}
        \node[draw,circle,inner sep=1.4pt, fill] at (2*\x,2*\x){};
       \node at (8,7){$\cR_4$};  
       \node at (20,13){$\cR_{10}$};  
       \node at (12,9){$\cR_6$}; 
       \node at (24,16){$\cR_{12}$};   
       \node at (14,12){$\cR_7$};  
       \node at (18,14){$\cR_9$};
       \node at (22,19){$\cR_{11}$};
      \foreach \x in {2,3,...,13}
       \node at (2*\x-1,1.5){$\cT_{\x}$};
      \node at (1.4,0.6){$\cT_1$};
      \draw[very nearly transparent, fill] (0,-1.5)--(0,0)--(6,6)--(8,6)--(10,8)--%
        (12,8)--(16,12)--(20,12)--(22,14)--(24,14)--(26,16)--(26,-1.5)--(0,-1.5);
\end{tikzpicture}
\caption{Partition of $\cA(\epsilon)$ into admissible polygons for $q_1\le 26$.}
\label{P:fig1}
\end{figure}
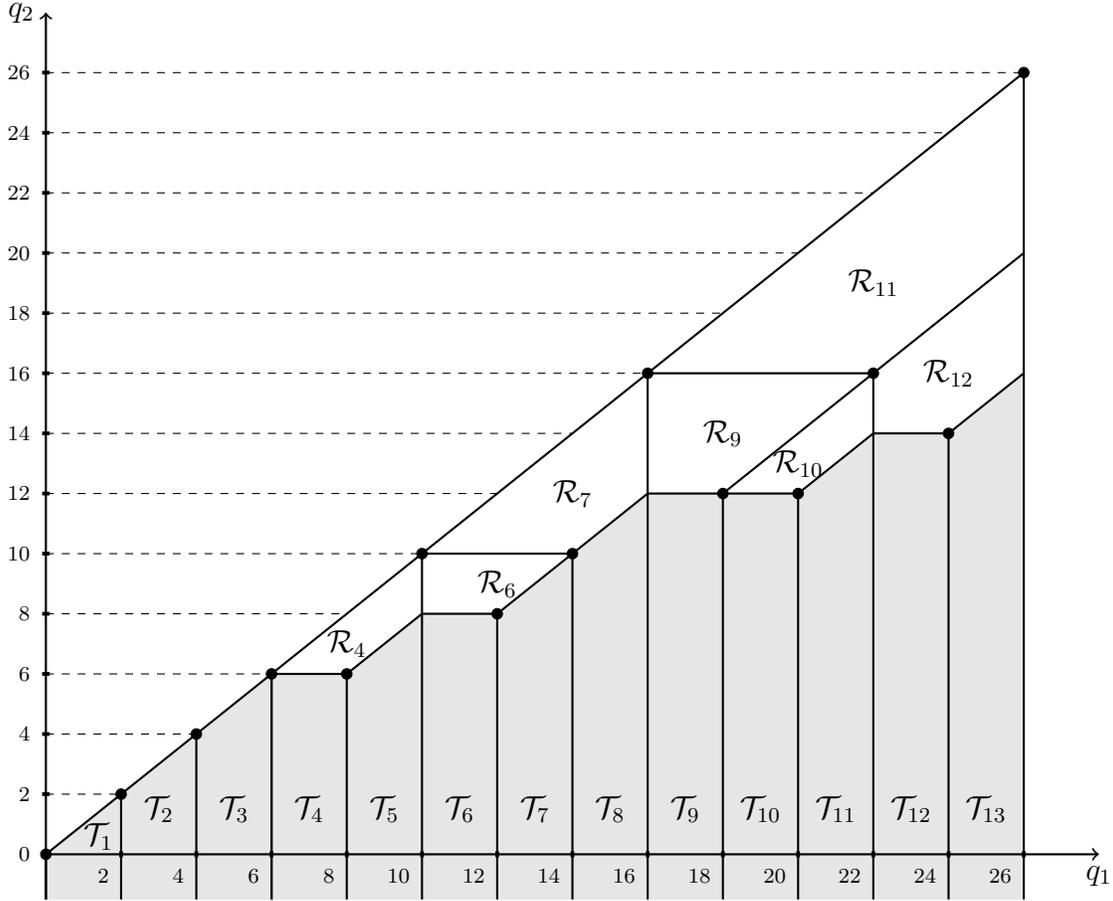

The proof of Theorem~\ref{P:thm} requires several steps, leaving the formulas 
\eqref{P:thm:eq2} and \eqref{P:thm:eq4} for the end.  In the process, we also 
prove, as Proposition~\ref{P:prop1}, the assertion made at the beginning
of section~\ref{main:ssec:P}.  We start by introducing some additional notation.

\begin{definition}
\label{P:def:pi1}
We denote by $\pi_1\colon\bR^2\to\bR$ the projection on the first coordinate:
\[
 \pi_1(\uq)=q_1 \quad \text{for each $\uq=(q_1,q_2)\in\bR^2$.}
\]
\end{definition}

\begin{definition}
\label{P:def:q}
For each $v\in\left[\epsilon,w_\infty\right[$, we denote by
\[
 \uq(v) = (q_1(v),q_2(v)) = (2|v|, |v|+|\alpha(v)|)
\]
the common vertex of $\cA(v)$, $\cB(v)$ and $\cC(v)$.
\end{definition}

The $14$ points $\uq(v)$ with $v\in\left[\epsilon, w_6\right]$ are displayed 
as large dots in Figure~\ref{P:fig1}. 

\begin{definition}
\label{P:def:trap}
Let $\ell\ge 1$ be an integer.  For consecutive $u<v$ in 
$\{\epsilon\}\cup\cV_\ell$, we define
\begin{equation}
\label{P:def:trap:eq1}
 \trap(u,v):=\cA(u)\cap\cC(v). 
\end{equation}
We denote by $T_\ell$ the collection of these sets, and denote their union by
\begin{equation}
\label{P:def:trap:eq2}
 \layer(\ell)=\trap(\epsilon,w_\ell)\cup\trap(w_\ell,w_{\ell+1})\cup\cdots
\end{equation} 
\end{definition}

When $\ell=1$, we have $\{\epsilon\}\cup\cV_\ell=[\epsilon,w_\infty[$, 
and so the sets $\trap(u,v)$ defined above are the same as those in 
Theorem~\ref{P:thm}.  The next lemma justifies this notation by showing 
that the latter are trapezes for any choice of $\ell$.  For $\ell=1$, it proves 
the first assertion of Theorem~\ref{P:thm}.

\begin{lemma}
\label{P:lemma:T}
Let $\ell\ge 1$ be an integer, let $u<v$ be consecutive elements 
of $\{\epsilon\}\cup\cV_\ell$, and let $\cT=\trap(u,v)$.  Then, we 
have $\alpha(u)\neq \alpha(v)$ and $\cT$ consists of the points 
$\uq=(q_1,q_2)\in\bR^2$ satisfying
\begin{equation}
\label{P:lemma:T:eq1}
 q_1(u)\le q_1\le q_1(v) \et 
 q_2\le 
  \begin{cases} 
   q_1-q_1(u)+q_2(u) &\text{if $\alpha(u)<\alpha(v)$,}\\
   q_2(v) &\text{if $\alpha(u)>\alpha(v)$.}  
  \end{cases} 
\end{equation}
In both cases, $\cT$ is a closed unbounded admissible trapeze 
with three sides and two vertices: two vertical sides, unbounded 
from below, each one ending in a vertex, and a line segment of slope 
$0$ or $1$ joining the two vertices.    Moreover, 
$\pi_1(\cT)=\left[q_1(u),q_1(v)\right]$.
\end{lemma}

We define the \emph{left vertex} (resp.~\emph{right vertex}) of $\cT=\trap(u,v)$ 
to be the vertex of its left (resp.~right) vertical side, and we define its 
\emph{top side} to be the side of $\cT$ joining these two vertices.

\begin{proof}[Proof of Lemma~\ref{P:lemma:T}]
It suffices to prove the first assertion.  By definition, $\cT$ consists 
of the points $(q_1,q_2)\in\bR^2$ satisfying
\begin{equation}
\label{P:lemma:T:eq2}
 q_1(u)\le q_1\le q_1(v) \et 
 q_2\le   \min\{ q_1-q_1(u)+q_2(u),  q_2(v) \}.
\end{equation}
Let $q_1 \in \left[ q_1(u), q_1(v) \right]$.  We need to show that the upper 
bounds for $q_2$ are the same in \eqref{P:lemma:T:eq1} and 
\eqref{P:lemma:T:eq2}.  To this end, we note that, if $u>\epsilon$, the numbers 
$|u|<|v|$ are consecutive elements of $\cVbar_\ell$ and 
Corollary~\ref{Vbar:prop2:cor2} gives 
\begin{equation}
\label{P:lemma:T:eq3}
 \big| |\alpha(v)|-|\alpha(u)| \big| \ge |v|-|u|.
\end{equation}
This also holds if $u=\epsilon$ for then $v=w_\ell$, so $\alpha(u)=u$ 
and $\alpha(v)=v$.  Thus, we have $\alpha(u) \neq \alpha(v)$.  
If $\alpha(u)<\alpha(v)$, we deduce from \eqref{P:lemma:T:eq3} that 
\[
q_1-q_1(u)+q_2(u) \le 2|v|-|u|+|\alpha(u)| \le |v|+|\alpha(v)| = q_2(v).
\]
Otherwise, we have $\alpha(u)>\alpha(v)$, and so
\[
q_1-q_1(u)+q_2(u) \ge |u|+|\alpha(u)| \ge |v|+|\alpha(v)| = q_2(v).
\qedhere
\]
\end{proof}

The next lemma proves the third assertion of Theorem~\ref{P:thm}.

\begin{lemma}
\label{P:lemma:C}
Let $\ell\ge 2$ be an integer, let $u<v<w$ be consecutive elements 
of $\cV_\ell$ with $v\notin\cV_{\ell+1}$, and let $\cR=\cell(u,v,w)$.  Then, 
$u<w$ are consecutive elements of $\cV_{\ell+1}$ with 
$|w|-|u|=F_\ell=|\alpha(v)|$, 
the point $\uq(v)$ lies in the interior $\trap(u,w)$, and we have
\begin{equation}
\label{P:lemma:C:eq}
 \begin{aligned}
  &\trap(u,w)\cap\cA(v)=\trap(v,w),\\
  &\trap(u,w)\cap\cB(v)=\cell(u,v,w),\\
  & \trap(u,w)\cap\cC(v)=\trap(u,v).
 \end{aligned}
\end{equation}
Moreover, these three sets form a partition of $\trap(u,w)$ into admissible 
polygons.  In particular, $\cR$ is a bounded convex admissible 
polygon with $4$ or $5$ sides: the top sides of $\trap(u,v)$, $\trap(v,w)$ 
and $\trap(u,w)$, the vertical line segment joining the left 
vertices of $\trap(u,v)$ and $\trap(u,w)$ when distinct, and the vertical
line segment joining the right  vertices of $\trap(v,w)$ and $\trap(u,w)$ 
when distinct.  Moreover, $\pi_1(\cR)=\left[q_1(u),q_1(w)\right]$.
\end{lemma}

Thus, $\cR=\cell(u,v,w)$ has exactly three non-vertical sides.  One of them is 
the top side of $\trap(u,w)$.  We call it the \emph{top side} of $\cR$.
The other two are the top sides of $\trap(u,v)$ and $\trap(v,w)$.  We 
call them the \emph{bottom sides} of $\cR$.  The point
$\uq(v)$ is their common vertex.
This is illustrated in Figure~\ref{P:fig2} when the top side of $\trap(u,w)$ 
has slope $0$.  The configuration is similar when it has slope $1$.

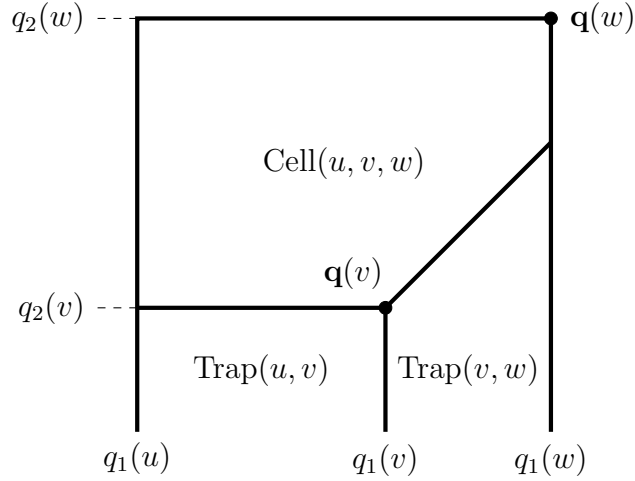
\begin{figure}[h]
\begin{tikzpicture}[scale=1.1]
     \draw[-,ultra thick] (0,0.5)--(0,5.5)--(5,5.5)--(5,0.5) node[below]{$q_1(w)$};
     \draw[-,ultra thick] (5,4)--(3,2)--(3,0.5) node[below]{$q_1(v)$};
     \draw[-,ultra thick] (3,2)--(0,2);
     \draw[-,dashed] (0,2)--(-0.5,2) node[left]{$q_2(v)$};
      \node[below] at (0,0.5){$q_1(u)$};
     \draw[-,dashed] (0,5.5)--(-0.5,5.5) node[left]{$q_2(w)$};
     \node[draw,circle,inner sep=1.7pt,fill] at (5,5.5){};
     \node[right] at (5.1,5.5){$\uq(w)$};
     \node[draw,circle,inner sep=1.7pt,fill] at (3,2){};
     \node[above left] at (3.1,2.1) {$\uq(v)$};
     \node at (1.5,1.25) {$\trap(u,v)$};
     \node at (4,1.25) {$\trap(v,w)$};
     \node at (2.5,3.75) {$\cell(u,v,w)$};
\end{tikzpicture}
\caption{Partition of $\trap(u,w)$ when $\alpha(w)<\alpha(u)$}
\label{P:fig2} 
\end{figure}

\begin{proof}[Proof of Lemma~\ref{P:lemma:C}]
By Proposition~\ref{Vbar:prop4}, $u<w$ are consecutive words in 
$\cV_{\ell+1}$, so we may form the trapeze $\cT=\trap(u,w)$.  We also have 
$|w|-|u|=F_\ell$.  Since $\alpha(u)>\alpha(v)$ and $\alpha(w)>\alpha(v)$,
Lemma~\ref{P:lemma:T} shows that the top side of $\trap(u,v)$ 
is horizontal, that the top side of $\trap(v,w)$ has slope $1$,
and that they share $\uq(v)$ as a vertex.  

If $\alpha(u)>\alpha(w)$, the top 
side of $\cT$ is horizontal with right vertex $\uq(w)$ above the right 
vertex of $\trap(v,w)$, on the same vertical line.  Thus, $\uq(v)$ is an 
interior point of $\cT$. This case is illustrated in Figure~\ref{P:fig2}.  
Otherwise, we have $\alpha(u)<\alpha(w)$, so the top 
side of $\cT$ has slope $1$, with left vertex $\uq(u)$ above the left 
vertex of $\trap(u,v)$, on the same vertical line.  Again, $\uq(v)$ is 
an interior point of $\cT$.

Since $\uq(v)\in\cT\subseteq\cA(u)$, we have $\cA(v)\subseteq\cA(u)$ 
which yields the first formula in \eqref{P:lemma:C:eq}.  Similarly, since 
$\uq(v)\in\cT\subseteq\cC(w)$, we have $\cC(v)\subseteq\cC(w)$ 
which yields the 
third formula in \eqref{P:lemma:C:eq}.  The middle formula follows 
directly from the definition of $\cell(u,v,w)$.  

Since $\cA(v)$, $\cB(v)$ and $\cC(v)$ form a partition of $\bR^2$ into 
admissible sectors with common vertex $\uq(v)$ in the interior of $\cT$, it follows
that the three sets in \eqref{P:lemma:C:eq} form a partition of $\cT$ 
into admissible polygons, and the remaining assertions follow.
\end{proof}

We can now prove the assertion made at the beginning of section~\ref{main:ssec:P}.

\begin{proposition}
\label{P:prop1}
Let $\uq\in\cA(\epsilon)$, let $u$ be the largest element of $[\epsilon,w_\infty[$
for which $\uq\in\cA(u)$, and let $w$ be the smallest element of $]u,w_\infty[$
for which $\uq\in\cC(w)$.  Then, exactly one of the following holds.
\begin{itemize}
 \item[{\rm (i)}]  The words $u<w$ are consecutive elements of 
 $[\epsilon,w_\infty[$.  We have $|w|-|u|=F_1=1$, $\uq\in\trap(u,w)$ and 
 $\uq\notin\cA(w)$.
\smallskip
 \item[{\rm (ii)}] There is an integer $\ell\ge 2$ such that $u<w$ are consecutive 
 elements of $\cV_{\ell+1}$ with $|w|-|u|=F_\ell$, and there exists 
 $v\in\cV_\ell\setminus\cV_{\ell+1}$ such that $u<v<w$ are consecutive  
 in $\cV_\ell$.  The point $\uq$ belongs to $\cell(u,v,w)$ but not to 
 its bottom sides, nor to $\cA(w)$. 
\end{itemize}
\end{proposition}

\begin{proof}
As $(F_\ell)_{\ell\ge 1}$ is strictly increasing, the two assertions are mutually 
exclusive. So, it suffices to show that one of them applies.

Suppose first that $|u|\ge 2$.  Then $\{u,w\} \subseteq \cV_2 = 
\left[w_2,w_\infty\right[$, and so there is a largest integer $m\ge 2$ 
for which $\{u,w\} \subseteq \cV_m$.  We claim that $u$ and $w$ are
consecutive elements of $\cV_m$.

To prove this claim, suppose first that
$u\notin\cV_{m+1}$ and let $v$ be the successor of $u$ in 
$\cV_m$.  By Lemma~\ref{P:lemma:T}, we have $\alpha(v) >
\alpha(u)=w_m$ and the top side of $\trap(u,v)$ has slope $1$.
Moreover, by the choice of $u$, we have $\uq\notin\cA(v)$, thus
\begin{equation}
 \label{P:prop1:eq1}
 \uq\in\cA(u)\setminus\cA(v)=\trap(u,v)\setminus\cA(v) \subseteq \cC(v).
\end{equation}
Since $v\le w$, this implies that $w=v$, so $u<w$ are consecutive in
$\cV_m$, as claimed.   Suppose now that $u\in\cV_{m+1}$ and 
let $v\ge u$ be the predecessor of $w$ in $\cV_m$.  By the choice 
of $m$, we have $w\notin\cV_{m+1}$.  So, Lemma~\ref{P:lemma:T} 
yields $\alpha(v) > \alpha(w)=w_m$ and shows that the top side 
of $\trap(v,w)$ is horizontal.  If $v>u$, we also have $\uq\notin\cC(v)$ 
by the choice of $w$, and so
\[
 \uq\in\cC(w)\setminus\cC(v)=\trap(v,w)\setminus\cC(v) \subseteq \cA(v)
\]
against the choice of $u$.  Thus $v=u$, and the claim holds once 
again.

Let $\ell$ be the smallest integer with $1\le \ell \le m-1$ such that
$u$ and $w$ are consecutive elements of $\cV_{\ell+1}$.  If $\ell=1$, 
then assertion (i) holds.  If $\ell\ge 2$, then $u<w$ are not consecutive 
in $\cV_\ell$.  So, by Proposition~\ref{Vbar:prop4}, we have $|w|-|u|=F_\ell$
and there exists $v\in\cV_\ell\setminus\cV_{\ell+1}$ such that $u<v<w$ 
are consecutive in $\cV_\ell$.  Then (ii) holds because $\uq\notin
\cA(v)\cup\cC(v)\cup\cA(w)$.

Finally, suppose that $|u|<2$.  If $u=\epsilon$, then \eqref{P:prop1:eq1}
holds with $v=w_1$ since $\alpha(\epsilon)<\alpha(w_1)$,
and so $w=w_1$.  Otherwise, we have $u=w_1$ and \eqref{P:prop1:eq1}
holds with $v=w_2$ since $\alpha(w_1)<\alpha(w_2)$, hence
$w=w_2$.  Thus, (i) applies.
\end{proof}

\begin{definition}
\label{P:def:S}
We denote by $S_1$ the set $T_1$ of all trapezes $\trap(u,v)$ 
where $u<v$ are consecutive words in $[\epsilon,w_\infty[$.  
For each integer $\ell\ge 2$, we denote by $S_\ell$ the set of all 
polygons $\cell(u,v,w)$ where $u<v<w$ 
are consecutive words in $\cV_\ell$ with $v\notin\cV_{\ell+1}$.   
We also set  $S = \bigcup_{\ell\ge 1} S_\ell$.
\end{definition}

Our next goal is to show that $S$ is a partition of $\cA(\epsilon)$,
as asserted in Theorem~\ref{P:thm}.  By Proposition~\ref{P:prop1},
$\cA(\epsilon)$ is the union of the polygons of $S$.  So, it remains to 
prove that any pair of distinct polygons of $S$ are compatible.  
The next result shows that the writing of an element of $S$,
as in the above definition, is unique.  It also shows that distinct 
polygons of $S$ have distinct projections on the first coordinate axis.

\begin{lemma}
\label{P:lemma:unicity}
Let $\cR\in S$.  Then, there is a unique integer $\ell\ge 1$ such 
that $\cR\in S_\ell$.  
\begin{itemize} 
  \item[{\rm (i)}] If $\ell=1$, there is a unique pair of consecutive elements
  $u<w$ of $[\epsilon,w_\infty[$ such that $\cR=\trap(u,w)$.
  \smallskip
  \item[{\rm (ii)}] If $\ell\ge 2$, there is a unique triple of consecutive 
  elements $u<v<w$ of $\cV_\ell$ with $v\notin\cV_{\ell+1}$ such that
  $\cR=\cell(u,v,w)$.
\end{itemize}  
In both cases, we have 
\begin{equation}
\label{P:lemma:unicity:eq1}
 |w|-|u|=F_\ell \et \pi_1(\cR)=[q_1(u),q_2(w)].
\end{equation}
Distinct polygons of $S$ have distinct projections under $\pi_1$.
\end{lemma}

For example, the polygon $\cR_9$ in Figure~\ref{P:fig1} has 
projection $\pi_1(\cR_9)=[16,22]$.  Thus, it is $\cell(u,v,w)$ where
$|u|=8$, $|v|=9$ and $|w|=11$.  Since $|w|-|u|=3=F_3$, it belongs 
to $S_3$.  Indeed, $8<9<11$ are consecutive elements of $\cVbar_3$
with $\alpha(9)=3$, as table \eqref{Vbar:table} shows.

\begin{proof}[Proof of Lemma~\ref{P:lemma:unicity}]
Suppose first that $\cR\in S_1$.  Then, $\cR=\trap(u,w)$ for consecutive 
elements $u<w$ of $[\epsilon,w_\infty[$.  We have $|w|-|u|=1=F_1$ and 
Lemma~\ref{P:lemma:T} yields $\pi_1(\cR)=[q_1(u),q_2(w)]$.  

Suppose now that $\cR\in S_\ell$ for some $\ell\ge 2$.  Then, 
$\cR=\cell(u,v,w)$ for consecutive 
elements $u<v<w$ of $\cV_\ell$ with $v\notin\cV_{\ell+1}$.  By
Proposition~\ref{Vbar:prop3}, we have $|w|-|u|=F_\ell$.  
By Lemma~\ref{P:lemma:C}, the top side of $\cR$ is the top side
of $\trap(u,w)$.  Hence, $\cR$ has the same projection as $\trap(u,w)$ under
$\pi_1$, which is $[q_1(u),q_2(w)]$ by Lemma~\ref{P:lemma:T}.

Thus, \eqref{P:lemma:unicity:eq1} holds in all cases.  Hence,
$u$ and $w$ are uniquely determined by $\pi_1(\cR)$, which in 
turn determine $\ell$.  If $\ell\ge 2$, then $v$ is also 
determined as the successor of $u$ in $\cV_\ell$.  
\end{proof}

\begin{proposition}
\label{P:prop2}
Let $\ell\ge 1$ be an integer, and let $u<v<w$ be consecutive elements 
of $\{\epsilon\}\cup\cV_\ell$.  Then $\trap(u,v)$ and $\trap(v,w)$ are 
compatible trapezes: their intersection is both  the right vertical side of 
$\trap(u,v)$ and the left vertical side of $\trap(v,w)$.  
\end{proposition}

\begin{proof}
By Lemma~\ref{P:lemma:T}, the right vertex of 
$\trap(u,v)$ is $(q_1(v),r)$ for some $r\in\bR$ and the left vertex of
$\trap(v,w)$ is $(q_1(v),s)$ for some $s\in\bR$.  The lemma also allows us
to compute $r$ and $s$ in terms of the numbers $x=|u|$, $y=|v|$, $z=|w|$,
$\alphabar(x)=|\alpha(u)|$, $\alphabar(y)=|\alpha(v)|$ and 
$\alphabar(z)=|\alpha(w)|$.  We simply need to show that $r=s$.
Without loss of generality, we may assume that $u<v<w$ are not 
consecutive elements of $\{\epsilon\}\cup\cV_{\ell+1}$.

If $|u|\ge 2$, we have $\ell\ge 2$, and $x<y<z$ are consecutive numbers 
in $\cVbar_{\ell}$, not all contained in $\cVbar_{\ell+1}$.  So,
Proposition~\ref{Vbar:prop5} applies and yields four cases (i)--(iv)
to consider.  Figure~\ref{P:fig3} shows the trapezes $\trap(u,v)$ and 
$\trap(v,w)$ in each of these cases.  Using Lemma~\ref{P:lemma:T}, 
we find in all cases that $r=s$:
\begin{center}
\begin{tabular}{rl}
Case (i): &$r=y+F_\ell=s$;\\[3pt]
Case (ii): &$r=(x+F_\ell)+2(y-x)=z+F_\ell=s$;\\[3pt]
Case (iii): &$r=(x+F_\ell)+2(y-x)=y+F_{\ell+1}=s$;\\[3pt]
Case (iv): &$r=y+F_{\ell+1}=z+F_{\ell}=s$.
\end{tabular}
\end{center}
If $u=\epsilon$, then $v=w_\ell$, $w=w_{\ell+1}$, and we find $r=2F_\ell=s$.
Finally, if $u=w_1$, then $\ell=1$, $v=w_2$, $w=w_3$, and $r=s=4$.
\end{proof}

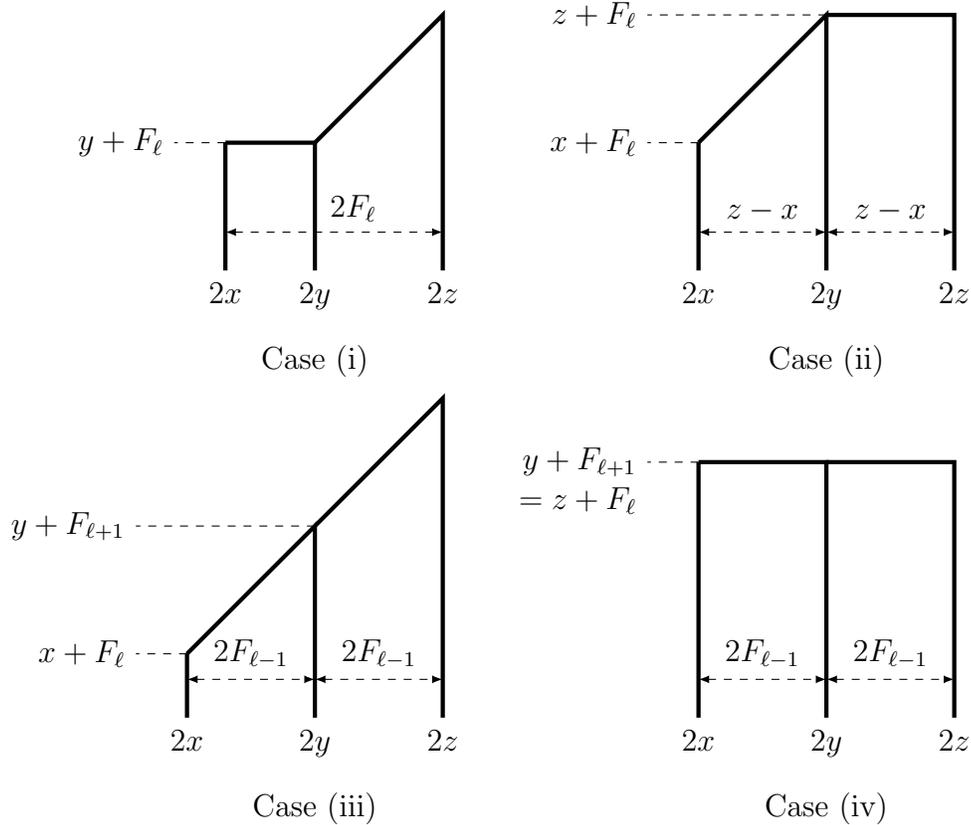
\begin{figure}[h]
\hspace*{8pt}
\begin{tikzpicture}[scale=1.7]
       \draw[-, ultra thick] (0.3,4.5)--(0.3,3.5) node[below]{$2x$};
       \draw[-, ultra thick] (0.3,4.5) -- (1,4.5)--(1,3.5) node[below]{$2y$};
       \draw[-, ultra thick] (1,4.5) -- (2,5.5) -- (2,3.5) node[below]{$2z$};
       \draw[-, dashed] (0.3,4.5) -- (-0.1,4.5) node[left] {$y+F_\ell$};
       \draw[latex-latex, dashed] (0.3,3.8) -- (2,3.8);
       \draw node[above] at (1.3,3.8) {$2F_\ell$};
       \draw node[below] at (1,3) {Case (i)};
       \draw[-, ultra thick] (4,4.5)--(4,3.5) node[below]{$2x$};
       \draw[-, ultra thick] (4,4.5) -- (5,5.5)--(5,3.5) node[below]{$2y$};
       \draw[-, ultra thick] (5,5.5) -- (6,5.5) -- (6,3.5) node[below]{$2z$};
       \draw[-, dashed] (4,4.5) -- (3.6,4.5) node[left] {$x+F_\ell$};
       \draw[-, dashed] (5,5.5) -- (3.6,5.5) node[left] {$z+F_{\ell}$};
       \draw[latex-latex, dashed] (4,3.8) -- (5,3.8);
       \draw node[above] at (4.5,3.8) {$z-x$};
       \draw[latex-latex, dashed] (5,3.8) -- (6,3.8);
       \draw node[above] at (5.5,3.8) {$z-x$};
       \draw node[below] at (5,3) {Case (ii)};
       \draw[-, ultra thick] (0,0.5)--(0,0) node[below]{$2x$};
       \draw[-, ultra thick] (0,0.5) -- (1,1.5)--(1,0) node[below]{$2y$};
       \draw[-, ultra thick] (1,1.5) -- (2,2.5) -- (2,0) node[below]{$2z$};
       \draw[-, dashed] (0,0.5) -- (-0.4,0.5) node[left] {$x+F_\ell$};
       \draw[-, dashed] (1,1.5) -- (-0.4,1.5) node[left] {$y+F_{\ell+1}$};
       \draw[latex-latex, dashed] (0,0.3) -- (1,0.3);
       \draw node[above] at (0.5,0.3) {$2F_{\ell-1}$};
       \draw[latex-latex, dashed] (1,0.3) -- (2,0.3);
       \draw node[above] at (1.5,0.3) {$2F_{\ell-1}$};
       \draw node[below] at (1,-0.5) {Case (iii)};
      \draw[-, ultra thick] (4,2)--(4,0) node[below]{$2x$};
       \draw[-, ultra thick] (4,2) -- (5,2)--(5,0) node[below]{$2y$};
       \draw[-, ultra thick] (5,2) -- (6,2) -- (6,0) node[below]{$2z$};
       \draw[-, dashed] (4,2) -- (3.6,2);
       \draw node[left] at (3.6,2) {$y+F_{\ell+1}$};
       \draw node[left] at (3.6,1.7) {$=z+F_\ell$};
       \draw[latex-latex, dashed] (4,0.3) -- (5,0.3);
       \draw node[above] at (4.5,0.3) {$2F_{\ell-1}$};
       \draw[latex-latex, dashed] (5,0.3) -- (6,0.3);
       \draw node[above] at (5.5,0.3) {$2F_{\ell-1}$};
       \draw node[below] at (5,-0.5) {Case (iv)};
\end{tikzpicture} 
\caption{Four possible configurations for $\trap(u,v)$ and $\trap(v,w)$
for consecutive words $u<v<w$ in $\cV_\ell$ with $\ell\ge 1$.}
\label{P:fig3}
\end{figure}

\begin{corollary}
\label{P:prop2:cor1}
Let $\ell\ge 1$ be an integer.  Any pair of distinct trapezes of $T_\ell$ are
compatible.  Any pair of distinct polygons of $S_\ell$ are
compatible.
\end{corollary}

\begin{proof}
The first assertion is clear since any pair of distinct 
trapezes of $T_\ell$ with non-empty intersection are as in the 
proposition.

For the second assertion, we may assume that $\ell\ge 2$ since $S_1=T_1$.
Let $\cR$ and $\cR'$ be distinct elements of $S_\ell$.  Write
$\cR=\cell(u,v,w)$ and $\cR'=\cell(u',v',w')$ where $u<v<w$ and
$u'<v'<w'$ are triples of consecutive elements of $\cV_\ell$ with
$v\notin\cV_{\ell+1}$ and $v'\notin\cV_{\ell+1}$.  Without loss 
of generality, we may assume that $v<v'$ and so $w\le u'$.  
Then $\cR\cap \cR'$ is empty if $w<u'$, and is contained 
in the vertical line $L$ of abscissa $q_1(w)=q_1(u')$ if $w=u'$.  Suppose 
that $w=u'$.   By Lemma~\ref{P:lemma:C}, 
$\cR\cap L$ is the line segment joining the right vertex $\ur$ of $\trap(v,w)$ 
and the right vertex $\us$ of $\trap(u,w)$, while $\cR'\cap L$ is the line 
segment joining the left vertex $\ur'$ of $\trap(u',v')$ and the left vertex
$\us'$ of $\trap(u',w')$.  These two line segments are the same because, by
the proposition, we have $\ur=\ur'$ since $v<w=u'<v'$ are consecutive 
elements of $\cV_\ell$, and $\us=\us'$ since $u<w=u'<w'$ are consecutive 
elements of $\cV_{\ell+1}$.   Thus, $\cR\cap \cR'=\cR\cap L=\cR'\cap L$
is a common side of $\cR$ and $\cR'$ if $\ur\neq\us$, and a common vertex if
$\ur=\us$.  In all cases, $\cR$ and $\cR'$ are compatible.
\end{proof}

\begin{corollary}
\label{P:prop2:cor2}
Let $\ell\ge 1$ be an integer.  
Then, $\layer(\ell)$ is an admissible polygon and $T_\ell$ is a partition of it.  
The boundary of $\layer(\ell)$ is a polygonal line made of the vertical side 
of $\cA(\epsilon)$ and the top sides of the trapezes in $T_\ell$.
\end{corollary}

This follows immediately from the preceding corollary and the definition 
of $\layer(\ell)$.  For $\ell=1$, this is illustrated in 
Figure~\ref{P:fig1} where $\layer(1)$ is displayed in grey.   Note that, if
$u<v<w$ are consecutive words in $\{\epsilon\}\cup\cV_\ell$ and if the top
sides of $\trap(u,v)$ and $\trap(v,w)$ have the same slope, both of them
are contained in the same side of $\layer(\ell)$.  The next result provides 
an example of this.

\begin{corollary}
\label{P:prop2:cor3}
Let $\ell\ge 1$ be an integer. Then, $\trap(\epsilon,w_\ell)$ and 
$\trap(w_\ell,w_{\ell+1})$ form a partition of $\trap(\epsilon,w_{\ell+1})$, 
and so $\trap(\epsilon,w_{\ell+1})\subseteq\layer(\ell)$.  We also
have $\layer(1)=\layer(2)$, and $S_1$ is a partition of $\layer(2)$.
\end{corollary}

\begin{proof}
Since $\epsilon<w_\ell<w_{\ell+1}$ are consecutive elements of 
$\{\epsilon\}\cup\cV_\ell$, the proposition shows that $\trap(\epsilon,w_\ell)$ 
and $\trap(w_\ell,w_{\ell+1})$ are compatible.  Since 
$\alpha(\epsilon)<\alpha(w_\ell)<\alpha(w_{\ell+1})$, their top sides have 
slope $1$, like the top side of $\trap(\epsilon,w_{\ell+1})$.  So, the former 
trapezes form a partition of the latter.  This proves the first assertion.
For $\ell=1$, this gives $\trap(\epsilon,w_2)\subseteq\layer(1)$.  As
$\cV_2=[w_2,w_\infty[$, we conclude that $\layer(2)=\layer(1)$.  Since
$S_1=T_1$ is a partition of $\layer(1)$, it is therefore a partition of $\layer(2)$.
\end{proof}

\begin{lemma}
\label{P:lemma:bottom}
Let $\ell\ge 2$ be an integer and let $u<v<w$ be consecutive elements of 
$\cV_\ell$ with $v\notin\cV_{\ell+1}$.  Then $\cell(u,v,w)\cap\layer(\ell)$ is the 
union of the bottom sides of $\cell(u,v,w)$.  
\end{lemma}

\begin{proof}
This follows immediately from Lemma~\ref{P:lemma:C} since the portion
of $\layer(\ell)$ between the vertical line segments of abscissa $q_1(u)$ 
and $q_1(w)$ is $\trap(u,v)\cup\trap(v,w)$.
\end{proof}

\begin{lemma}
\label{P:lemma:top}
Let $u<w$ be consecutive elements of $\cV_{\ell+1}$ for some integer 
$\ell\ge 1$, and let $\cE$ be the top side of $\cT=\trap(u,w)$.  Then $\cE$ is a
side of a unique polygon $\cR \in S_1\cup\cdots\cup S_\ell$, and it is the top 
side of $\cR$.
\end{lemma}

\begin{proof}
If $\cE$ is a side of some $\cR \in S_1\cup\cdots\cup S_\ell$ then it is its 
top side because $\cR$ is contained in $\layer(\ell+1)$ and $\cE$ lies on 
the boundary of $\layer(\ell+1)$.  Thus, $\pi_1(\cR)=\pi_1(\cT)$
and so, by Lemma~\ref{P:lemma:unicity}, there is at most one such $\cR$.

To show the existence of $\cR$, we may assume
without loss of generality that $\ell$ is the smallest positive integer such 
that $u<w$ are consecutive in $\cV_{\ell+1}$.  If $\ell=1$, we may take 
$\cR=\cT$.  If $\ell\ge 2$, then $u<w$ are not consecutive in
$\cV_\ell$ and so, by Proposition~\ref{Vbar:prop4}, there exists 
$v\in\cV_\ell\setminus\cV_{\ell+1}$ such that $u<v<w$ are consecutive 
in $\cV_\ell$.  Then Lemma~\ref{P:lemma:C} shows that the top side of
$\cR=\cell(u,v,w)$ is $\cE$.
\end{proof}

The next result proves the last assertion of Theorem~\ref{P:thm}.

\begin{proposition}
\label{P:prop3}
The set $S_1\cup\cdots\cup S_{\ell-1}$ is a partition of $\layer(\ell)$ for each 
integer $\ell\ge 2$.  Moreover, $S=\bigcup_{\ell=1}^\infty S_\ell$ is a partition 
of $\cA(\epsilon)$.
\end{proposition}

\begin{proof}
Since the sets $\layer(\ell)$ with $\ell\ge 2$ form an increasing sequence 
whose union is $\cA(\epsilon)$, it suffices to show the first assertion.  
We do this by induction on $\ell$.

For $\ell=2$, the statement follows from Corollary~\ref{P:prop2:cor3}.

Now, suppose that $S_1\cup\cdots\cup S_{\ell-1}$ is a partition of 
$\layer(\ell)$ for some $\ell\ge 2$.  By Lemma~\ref{P:lemma:C},
each trapeze in $T_{\ell+1}$ either belongs to $T_\ell$ or decomposes 
as the union of two trapezes of $T_\ell$ and a polygon of $S_\ell$.
Thus $\layer(\ell+1)$ is the union of $\layer(\ell)$ and of the polygons
of $S_\ell$.  Hence, it is the union of the polygons in 
$S_1\cup\cdots\cup S_\ell$. 

To complete the induction step, it remains to show that any pair of
polygons $\cR\neq \cR'$ in $S_1\cup\cdots\cup S_\ell$ are compatible.
By hypothesis, this is true if they both belong to $S_1\cup\cdots\cup S_{\ell-1}$.
By Corollary~\ref{P:prop2:cor1}, this is also true if they both belong to $S_\ell$.
So, we may assume that $\cR\in S_\ell$ and that $\cR'\in S_1\cup\cdots\cup S_{\ell-1}$.  We may further assume that $\cR$ and $\cR'$ intersect.

Since $\cR'\subseteq\layer(\ell)$, the set  $\cR \cap \cR'$ is contained in
$\cR\cap\layer(\ell)$ which, by Lemma~\ref{P:lemma:bottom}, is the 
union of the two bottom sides of $\cR$.  As these have distinct slopes 
and as $\cR \cap \cR'$ is convex, that intersection is contained in a 
single bottom side $\cE$ of $\cR$.  By Lemma~\ref{P:lemma:top}, 
$\cE$ is a side of a unique $\cR''\in S_1\cup\cdots\cup S_{\ell-1}$.  
If $\cR'=\cR''$, then  $\cR \cap \cR'=\cE$ is a common side of $\cR$ 
and $\cR'$, and we are done.  Otherwise, $\cE$ is 
not a side of $\cR'$.  As $\cR'$ and $\cR''$ are 
compatible polygons by our induction hypothesis, it follows that 
$\cR\cap\cR'=\cE\cap\cR'$ is a common vertex of $\cR'$ and $\cE$,
thus also a vertex of $\cR$, and again we are done.
\end{proof}

We also have a similar result for trapezes.

\begin{proposition}
\label{P:prop4}
Let $\cT\in \bigcup_{\ell=1}^\infty T_\ell$.
Then, $\{\cR\in S\,;\, \cR\subseteq\cT\}$ is a partition of $\cT$.
\end{proposition}

\begin{proof}
By Proposition~\ref{P:prop3},  the polygons in $S$ are pairwise compatible.
Thus, it suffices to show that $\cT$ is a union of polygons of $S$.  
This is automatic if $\cT\in T_1$ because $T_1=S_1\subseteq S$.  Suppose 
that $\cT\in T_{\ell+1}$ for some $\ell \ge 1$.  
Then, we have $\cT\in T_\ell$, or $\cT=\trap(\epsilon,w_{\ell+1})$ is the 
union of two trapezes in $T_\ell$ by Corollary~\ref{P:prop2:cor2},
or we have $\ell\ge 2$ and $\cT$ is the union of an element of $S_\ell$
and two elements of $T_\ell$ by Lemma~\ref{P:lemma:C}.   As we may 
assume, by induction, that each 
element of $T_\ell$ is a union of polygons of $S$, the same is true for $\cT$.
\end{proof}

By Lemmas~\ref{P:lemma:T} and \ref{P:lemma:C}, the following result 
completes the proof of Theorem~\ref{P:thm}.

\begin{proposition}
\label{P:prop5}
Let $\cR\in S$, let $u,w\in\left[\epsilon,w_\infty\right[$ such that 
$\pi_1(\cR)=\left[q_1(u),q_1(w)\right]$, and let $\uq=(q_1,q_2)\in\cR$. 
Then, we have
\begin{equation}
\label{P:prop3:eq1}
\uP(\uq)=\Phi(q_1-|u|,q_2+|u|-|w|,|w|).
\end{equation}
\end{proposition}

\begin{proof}
We proceed by induction on the integer $\ell\ge 1$ for which $\cR\in S_\ell$.

Suppose first that $\ell=1$.  Then, $u<w$ are consecutive in 
$[\epsilon,w_\infty[$, and $\cR=\trap(u,w)$.  If $\uq\notin\cA(w)$,
then $u$ is the largest element of $[\epsilon,w_\infty[$ for which $\uq\in\cA(u)$
and $w$ is the smallest element of $]u,w_\infty[$ for which $\uq\in\cC(w)$.
Hence, formula~\eqref{main:P:eq1} for $\uP(\uq)$ applies and gives 
\eqref{P:prop3:eq1}.  If $\uq\in\cA(w)$, then $\uq$ belongs to 
$\trap(w,w')\setminus\cA(w')$ where $w'$ is the successor of $w$ in 
$[\epsilon,w_\infty[$, and the preceding yields
\[
 \uP(\uq)=\Phi(q_1-|w|,q_2+|w|-|w'|,|w'|).
\]
This yields \eqref{P:prop3:eq1} because $q_1=2|w|$ and $|w'|-|w|=1=|w|-|u|$, 
thus
\[
 (q_1-|w|,q_2+|w|-|w'|,|w'|) = (|w|,q_2+|u|-|w|,q_1-|u|).
\]

Suppose now that $\ell\ge 2$.  Then, by Lemma~\ref{P:lemma:unicity}, 
we have $|w|-|v|=F_\ell$ and $\cR=\cell(u,v,w)$ for some 
$v\in\cV_\ell\setminus\cV_{\ell+1}$ such that $u<v<w$ are consecutive 
elements of $\cV_\ell$.  Let $\cE_0$ denote the top side of $\trap(u,v)$, 
and $\cE_1$, the top side of $\trap(v,w)$.  By Lemma~\ref{P:lemma:C}, 
$\cE_0$ and $\cE_1$ are the bottom sides of $\cR$.  Moreover, $u<w$ are 
consecutive elements of $\cV_{\ell+1}$, and $\uq\in\trap(u,w)$.  
We consider several possibilities.

1. Suppose that $\uq\notin\cE_0\cup\cE_1$.  Then we have 
$\uq\notin\layer(\ell)$ by Lemma~\ref{P:lemma:bottom}.  Let $u'$ be 
the largest element of $[\epsilon,w_\infty[$ such that $\uq\in\cA(u')$ and let 
$w'$ be the smallest element of $]u',w_\infty[$ such that $\uq\in\cC(w')$.
Since $\uq\in\trap(u,w)$, we have $u\le u'<w'$.
By Proposition~\ref{P:prop1}, $u'<w'$ are consecutive elements of $\cV_{k+1}$
for some integer $k\ge 1$ (because $u'\ge u\ge w_2$), and so
$\uq\in\trap(u',w')\subseteq\layer(k+1)$.  Thus, we must have $k\ge \ell$
and therefore $\{u',w'\}\subseteq\cV_{\ell+1}$.   We divide this case into 
two sub-cases.

a) If $\uq\notin\cA(w)$, then we have $u\le u'<w'\le w$.
As $u<w$ are consecutive elements of $\cV_{\ell+1}$, we deduce that 
$u'=u$ and $w'=w$, and then \eqref{P:prop3:eq1} holds by definition of $\uP$.

b) If instead $\uq\in\cA(w)$, then $u'=w$ and, by 
Proposition~\ref{P:prop2}, $\uq$ belongs to 
the left vertical side of $\trap(w,w'')$ where $w''$ is the successor 
of $w$ in $\cV_{\ell+1}$.   Hence, we must have $w'=w''$.   
So, $w<w'$ are consecutive in $\cV_{\ell+1}$, but not consecutive 
in $\cV_\ell$ because $\uq\notin\layer(\ell)$, thus $|w'|-|w|=F_\ell=|w|-|u|$
by Proposition~\ref{Vbar:prop4}.  By definition of $\uP$, we thus have  
\[
 \uP(\uq)=\Phi(q_1-|w|,q_2+|w|-|w'|,|w'|)=\Phi(|w|,q_2+|u|-|w|,q_1-|u|)
\]
where the second equality uses $q_1=q_1(w)=2|w|$.  Thus, 
\eqref{P:prop3:eq1} also holds in this case.

2. Suppose that $\uq\in\cE_0$.  By Lemma~\ref{P:lemma:top}, $\cE_0$ is
the top side of a unique $\cR'\in S_1\cup\cdots\cup S_{\ell-1}$.  Moreover,
we have $\pi_1(\cR')=\pi_1(\cE_0)=[q_1(u),q_1(v)]$.  So, by induction, we 
may assume that
\[
 \uP(\uq)=\Phi(q_1-|u|,q_2+|u|-|v|,|v|).
\]
This yields \eqref{P:prop3:eq1} because $q_2=q_2(v)=|v|+F_\ell=|v|+|w|-|u|$, 
and so
\[
 (q_1-|u|,q_2+|u|-|v|,|v|) = (q_1-|u|,|w|,q_2+|u|-|w|). 
\]
 
3. Finally, suppose that $\uq\in\cE_1$.  Then, similarly as in the previous case,
$\cE_1$ is the top side of a unique $\cR''\in S_1\cup\cdots\cup S_{\ell-1}$
and we have $\pi_1(\cR'')=\pi_1(\cE_1)=[q_1(v),q_1(w)]$.  So, by induction, 
we may assume that
\[
 \uP(\uq)=\Phi(q_1-|v|,q_2+|v|-|w|,|w|).
\]
Then \eqref{P:prop3:eq1} follows because 
$q_1-q_2=q_1(v)-q_2(v)=|v|-F_\ell=|v|+|u|-|w|$, and so
\[
 (q_1-|v|,q_2+|v|-|w|,|w|) = (q_2+|u|-|w|, q_1-|u|,|w|).
 \qedhere 
\]
\end{proof}

%
%

\section{Additional properties of the map $\uP$}
 \label{sec:properties}
 
In this section, we study in more detail the map $\uP$ 
and we look more closely at its first component
$P_1$.  We also introduce a notion of integral $2$-parameter
$3$-system which applies to $\uP$ and extends that of integral 
$3$-system from \cite{R2003}.   
 
\begin{proposition}
 \label{properties:prop1}
Let $\uP=(P_1,P_2,P_3)\colon\cA(\epsilon)\to\bR^3$ be as in 
section~\ref{sec:P}, and let $\cR\in S$.  Then, there is a unique 
triple $(a,b,c)\in\bZ^3$ with $c=a+b$ such that 
\begin{equation}
\label{properties:prop1:eq1}
 \uP(\uq)=\Phi(q_1-a,q_2-b,c) 
 \quad
 \text{for each $\uq=(q_1,q_2)\in\cR$.}
\end{equation}
If $\cR\neq\trap(\epsilon,w_1)$, there is also a unique point 
$\ur$ of $\cR$ such that $P_1(\ur)=P_2(\ur)=P_3(\ur)$.  If moreover
$\cR\neq\trap(w_1,w_2)$, then $\ur$ is an interior point of $\cR$,
and the sets 
\[
 \cR_{i,j,k}=\{\uq=(q_1,q_2)\in\cR\,;\, P_i(\uq)=q_1-a, P_j(\uq)=q_2-b, P_k(\uq)=c\}
\]
attached to the six permutations $(i,j,k)$ of $(1,2,3)$ form a partition
of $\cR$ into admissible convex polygons with $\ur$ as a common 
vertex.   If $\cR$ is $\trap(\epsilon,w_1)$ or $\trap(w_1,w_2)$, then
the $\cR_{i,j,k}$ with non-empty interior provide a partition 
of $\cR$ into admissible convex polygons. 
\end{proposition}

\begin{proof}
We will simply treat the case where $\cR\in S_\ell$ for an integer $\ell\ge 2$.  
The reasoning is similar and simpler for the trapezes of $S_1$.  Thus,
we assume that $\cR=\cell(u,v,w)$ for consecutive elements $u<v<w$
of $\cV_\ell$ with $v\notin\cV_{\ell+1}$.  Then, Theorem~\ref{P:thm} 
yields \eqref{properties:prop1:eq1} with $(a,b,c)=(|u|,|\alpha(v)|,|w|)$.
By Lemma~\ref{P:lemma:C}, we have $|w|-|u|=F_\ell=|\alpha(v)|$, thus
$c=a+b$.  Since $\cR$ has non-empty interior, the condition
\eqref{properties:prop1:eq1} uniquely determines $(a,b,c)$.  
As shown in Figure~\ref{properties:fig1} when the top side of 
$\cR$ is horizontal, the lines of equation
$q_1-a=c$, $q_2-b=c$ and $q_1-a=q_2-b$ cut respectively in their
middle the top sides of $\trap(u,w)$, $\trap(v,w)$ and $\trap(u,v)$.
Their intersection point $\ur=(a+c,b+c)$ is therefore an interior  
point of $\cR$ and the only point of $\cR$ where $P_1$, $P_2$ and
$P_3$ coincide.  These lines induce a partition of $\cR$ into six convex 
polygons on which the differences $(q_1-a)-c$, $(q_2-b)-c$ and  
$(q_1-a)-(q_2-b)$ are everywhere $\ge 0$ or $\le 0$.  By definition of $\uP$,
this is equivalent to $(q_1-a,q_2-b,c)=(P_i(\uq),P_j(\uq),P_k(\uq))$ 
for a fixed permutation $(i,j,k)$ of $(1,2,3)$.  We get a different permutation 
for each polygon, as indicated in Figure~\ref{properties:fig1}.
\end{proof}

\begin{figure}[ht]
 \begin{tikzpicture}[scale=0.6]
   \draw[-, thick] (2,2) -- (12,2) -- (18,8) -- (18,10) -- (2,10) -- (2,2);
   \draw[-,very thick]  (7,2) -- (15,10);
   \draw[-,very thick]  (15,5) -- (2,5);
   \draw[-,very thick]  (10,2) -- (10,10);
   \draw[-,dotted,thick] (2,2)--(2,1) node[below]{$2|u|$};
   \draw[-,dotted,thick] (12,2)--(12,1) node[below]{$\ 2|v|$};
   \draw[-,dotted,thick] (18,8)--(18,1) node[below]{$2|w|$};
   \draw[-,dotted,thick] (7,2)--(7,1) node[below]{$|u|+|v|$};
   \draw[-,dotted,thick] (10,2)--(10,1) node[below]{$|u|+|w|$};
   \draw[-,dotted,thick] (2,2)--(1,2) node[left]{$|v|+|\alpha(v)|$};
   \draw[-,dotted,thick] (2,5)--(1,5) node[left]{$|w|+|\alpha(v)|$};
   \node at (8.9,2.5){$\cR_{2,1,3}$};
   \node at (11.9,3.4){$\cR_{3,1,2}$};
   \node at (13.5,6.2){$\cR_{3,2,1}$};
   \node at (11.4,8){$\cR_{2,3,1}$};
   \node at (7.4,7.2){$\cR_{1,3,2}$};
   \node at (7,3.7){$\cR_{1,2,3}$};
   \node[draw,circle,inner sep=1.4pt, fill] at (10,5){};
   \node[above left] at (10,5) {$\ur$};
 \end{tikzpicture}
\caption{Partition of $\cR$ into six polygons on which $\uP$ is affine}
\label{properties:fig1}
\end{figure}
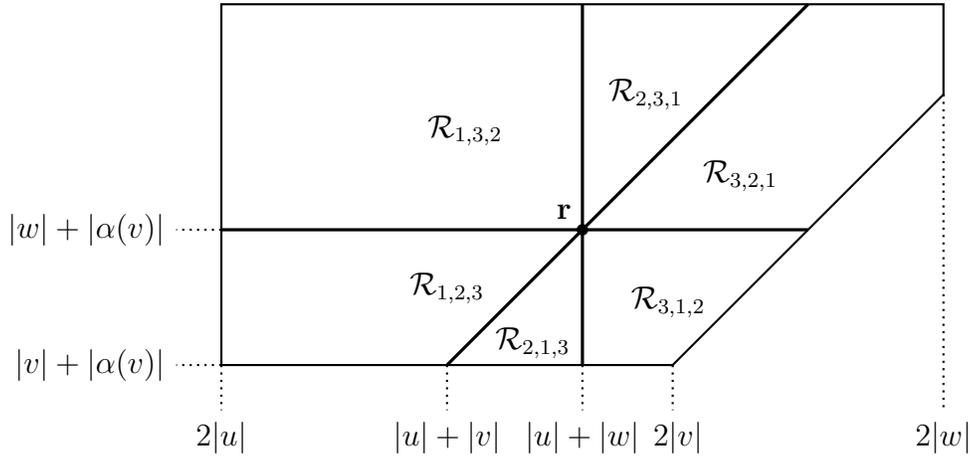

In the notation of the proposition, we have 
\[
 P_1(\uq)
   =\begin{cases}
   q_1-a &\text{if $\uq\in \cR_{1,2,3}\cup\cR_{1,3,2}$,}\\
   q_2-b &\text{if $\uq\in \cR_{2,1,3}\cup\cR_{3,1,2}$,}\\
          c &\text{if $\uq\in \cR_{2,3,1}\cup\cR_{3,2,1}$.}
      \end{cases}
\]
Thus, we obtain a $3$-dimensional picture of the 
graph of $P_1$ over a trapeze $\trap(\epsilon, w_\ell)$ by partitioning 
this trapeze into polygons $\cR\in S$ and then by colouring the
corresponding $\cR_{i,j,k}$ in light grey when $i=1$, in medium grey when 
$j=1$ and in dark grey when $k=1$.  The result is shown
in Figure~\ref{properties:fig2} for $\ell=7$.  On the connected regions 
in light grey, $P_1(\uq)-q_1$ is constant; on those in medium grey,
$P_1(\uq)-q_2$ is constant; on those in dark grey, $P_1(\uq)$ is constant.  
This picture shows some symmetries.  One notes for example 
that the colouring is the same, up to translation, in the region 
surrounded by solid lines as in the region surrounded by dashed 
lines.  This is an illustration of Theorem~\ref{main:P:thm1} with $k=6$.

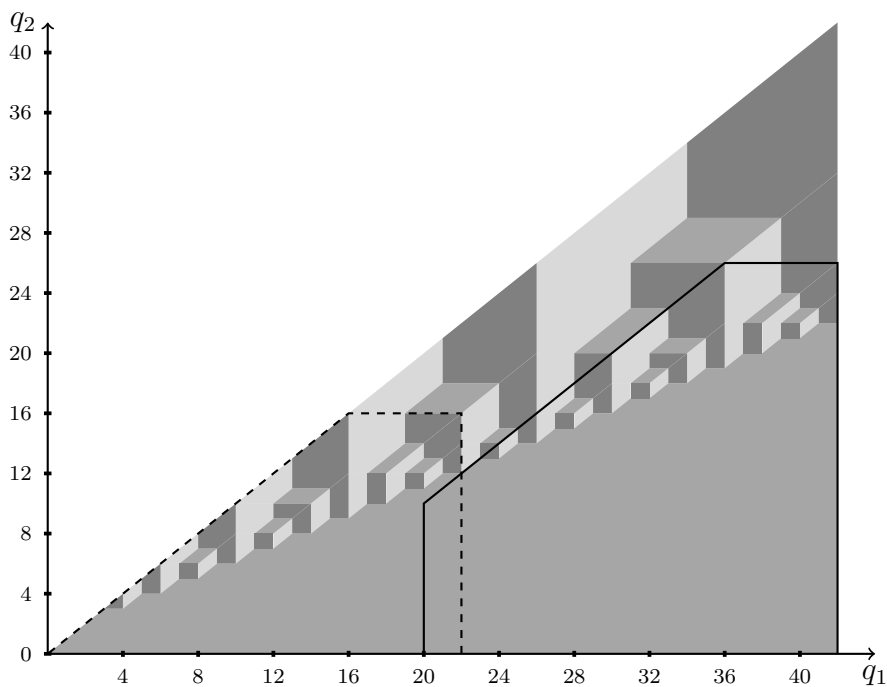
\begin{figure}[h]
\begin{tikzpicture}[xscale=0.25, yscale=0.2]
\fill[opacity=0.3, gray] (6, 6)--(6, 6)--(7, 6)--(8, 7)--(8, 8)--(6, 6); 
\fill[opacity=0.3, gray] (10, 10)--(10, 8)--(11, 8)--(12, 9)--(12, 10)--(10, 10); 
\fill[opacity=0.3, gray] (18, 12)--(18, 12)--(19, 12)--(20, 13)--(20, 14)--(18, 12); 
\fill[opacity=0.3, gray] (22, 16)--(22, 14)--(23, 14)--(24, 15)--(24, 18)--(22, 16); 
\fill[opacity=0.3, gray] (26, 20)--(26, 16)--(27, 16)--(28, 17)--(28, 20)--(26, 20); 
\fill[opacity=0.3, gray] (30, 20)--(30, 18)--(31, 18)--(32, 19)--(32, 20)--(30, 20); 
\fill[opacity=0.3, gray] (38, 22)--(38, 22)--(39, 22)--(40, 23)--(40, 24)--(38, 22); 
\fill[opacity=0.3, gray] (10, 10)--(10, 10)--(12, 10)--(13, 11)--(13, 13)--(10, 10); 
\fill[opacity=0.3, gray] (16, 16)--(16, 12)--(17, 12)--(19, 14)--(19, 16)--(16, 16); 
\fill[opacity=0.3, gray] (30, 20)--(30, 20)--(32, 20)--(33, 21)--(33, 23)--(30, 20); 
\fill[opacity=0.3, gray] (36, 26)--(36, 22)--(37, 22)--(39, 24)--(39, 29)--(36, 26); 
\fill[opacity=0.3, gray] (16, 16)--(16, 16)--(19, 16)--(21, 18)--(21, 21)--(16, 16); 
\fill[opacity=0.3, gray] (26, 26)--(26, 20)--(28, 20)--(31, 23)--(31, 26)--(26, 26); 
\fill[opacity=0.3, gray] (26, 26)--(26, 26)--(31, 26)--(34, 29)--(34, 34)--(26, 26); 
\fill[opacity=0.3, gray] (4, 4)--(4, 3)--(5, 4)--(5, 5)--(4, 4); 
\fill[opacity=0.3, gray] (6, 6)--(6, 4)--(7, 5)--(7, 6)--(6, 6); 
\fill[opacity=0.3, gray] (8, 6)--(8, 5)--(9, 6)--(9, 7)--(8, 6); 
\fill[opacity=0.3, gray] (10, 8)--(10, 6)--(11, 7)--(11, 8)--(10, 8); 
\fill[opacity=0.3, gray] (12, 8)--(12, 7)--(13, 8)--(13, 9)--(12, 8); 
\fill[opacity=0.3, gray] (14, 10)--(14, 8)--(15, 9)--(15, 11)--(14, 10); 
\fill[opacity=0.3, gray] (16, 12)--(16, 9)--(17, 10)--(17, 12)--(16, 12); 
\fill[opacity=0.3, gray] (18, 12)--(18, 10)--(19, 11)--(19, 12)--(18, 12); 
\fill[opacity=0.3, gray] (20, 12)--(20, 11)--(21, 12)--(21, 13)--(20, 12); 
\fill[opacity=0.3, gray] (22, 14)--(22, 12)--(23, 13)--(23, 14)--(22, 14); 
\fill[opacity=0.3, gray] (24, 14)--(24, 13)--(25, 14)--(25, 15)--(24, 14); 
\fill[opacity=0.3, gray] (26, 16)--(26, 14)--(27, 15)--(27, 16)--(26, 16); 
\fill[opacity=0.3, gray] (28, 16)--(28, 15)--(29, 16)--(29, 17)--(28, 16); 
\fill[opacity=0.3, gray] (30, 18)--(30, 16)--(31, 17)--(31, 18)--(30, 18); 
\fill[opacity=0.3, gray] (32, 18)--(32, 17)--(33, 18)--(33, 19)--(32, 18); 
\fill[opacity=0.3, gray] (34, 20)--(34, 18)--(35, 19)--(35, 21)--(34, 20); 
\fill[opacity=0.3, gray] (36, 22)--(36, 19)--(37, 20)--(37, 22)--(36, 22); 
\fill[opacity=0.3, gray] (38, 22)--(38, 20)--(39, 21)--(39, 22)--(38, 22); 
\fill[opacity=0.3, gray] (40, 22)--(40, 21)--(41, 22)--(41, 23)--(40, 22);
\fill[opacity=0.7, gray] (7, 6)--(8, 6)--(9, 7)--(8, 7)--(7, 6); 
\fill[opacity=0.7, gray] (11, 8)--(12, 8)--(13, 9)--(12, 9)--(11, 8); 
\fill[opacity=0.7, gray] (19, 12)--(20, 12)--(21, 13)--(20, 13)--(19, 12); 
\fill[opacity=0.7, gray] (23, 14)--(24, 14)--(25, 15)--(24, 15)--(23, 14); 
\fill[opacity=0.7, gray] (27, 16)--(28, 16)--(29, 17)--(28, 17)--(27, 16); 
\fill[opacity=0.7, gray] (31, 18)--(32, 18)--(33, 19)--(32, 19)--(31, 18); 
\fill[opacity=0.7, gray] (39, 22)--(40, 22)--(41, 23)--(40, 23)--(39, 22); 
\fill[opacity=0.7, gray] (12, 10)--(14, 10)--(15, 11)--(13, 11)--(12, 10); 
\fill[opacity=0.7, gray] (17, 12)--(18, 12)--(20, 14)--(19, 14)--(17, 12); 
\fill[opacity=0.7, gray] (32, 20)--(34, 20)--(35, 21)--(33, 21)--(32, 20); 
\fill[opacity=0.7, gray] (37, 22)--(38, 22)--(40, 24)--(39, 24)--(37, 22); 
\fill[opacity=0.7, gray] (19, 16)--(22, 16)--(24, 18)--(21, 18)--(19, 16); 
\fill[opacity=0.7, gray] (28, 20)--(30, 20)--(33, 23)--(31, 23)--(28, 20); 
\fill[opacity=0.7, gray] (31, 26)--(36, 26)--(39, 29)--(34, 29)--(31, 26); 
\fill[opacity=0.7, gray] (0,0)--(3,3)--(4, 3)--(5, 4)
   --(6, 4)--(7, 5)--(8, 5)--(9, 6)--(10, 6)--(11, 7)--(12, 7)
   --(13, 8)--(14, 8)--(15, 9)--(16, 9)--(17, 10)--(18, 10)--(19, 11)--(20, 11)
   --(21, 12)--(22, 12)--(23, 13)--(24, 13)--(25, 14)--(26, 14)
   --(27, 15)--(28, 15)--(29, 16)--(30, 16)
   --(31, 17)--(32, 17)--(33, 18)--(34, 18)--(35, 19)--(36, 19)--(37, 20)--(38, 20)
   --(39, 21)--(40, 21)--(41, 22)--(42, 22)--(42, 0)--(0, 0);
\fill[opacity=1, gray] (8, 8)--(8, 7)--(9, 7)--(10, 8)--(10, 10)--(8, 8); 
\fill[opacity=1, gray] (12, 10)--(12, 9)--(13, 9)--(14, 10)--(14, 10)--(12, 10); 
\fill[opacity=1, gray] (20, 14)--(20, 13)--(21, 13)--(22, 14)--(22, 16)--(20, 14); 
\fill[opacity=1, gray] (24, 18)--(24, 15)--(25, 15)--(26, 16)--(26, 20)--(24, 18); 
\fill[opacity=1, gray] (28, 20)--(28, 17)--(29, 17)--(30, 18)--(30, 20)--(28, 20); 
\fill[opacity=1, gray] (32, 20)--(32, 19)--(33, 19)--(34, 20)--(34, 20)--(32, 20); 
\fill[opacity=1, gray] (40, 24)--(40, 23)--(41, 23)--(42, 24)--(42, 26)--(40, 24); 
\fill[opacity=1, gray] (13, 13)--(13, 11)--(15, 11)--(16, 12)--(16, 16)--(13, 13); 
\fill[opacity=1, gray] (19, 16)--(19, 14)--(20, 14)--(22, 16)--(22, 16)--(19, 16); 
\fill[opacity=1, gray] (33, 23)--(33, 21)--(35, 21)--(36, 22)--(36, 26)--(33, 23); 
\fill[opacity=1, gray] (39, 29)--(39, 24)--(40, 24)--(42, 26)--(42, 32)--(39, 29); 
\fill[opacity=1, gray] (21, 21)--(21, 18)--(24, 18)--(26, 20)--(26, 26)--(21, 21); 
\fill[opacity=1, gray] (31, 26)--(31, 23)--(33, 23)--(36, 26)--(36, 26)--(31, 26); 
\fill[opacity=1, gray] (34, 34)--(34, 29)--(39, 29)--(42, 32)--(42, 42)--(34, 34); 
\fill[opacity=1, gray] (3, 3)--(4, 3)--(4, 4)--(3, 3); 
\fill[opacity=1, gray] (5, 5)--(5, 4)--(6, 4)--(6, 6)--(5, 5); 
\fill[opacity=1, gray] (7, 6)--(7, 5)--(8, 5)--(8, 6)--(7, 6); 
\fill[opacity=1, gray] (9, 7)--(9, 6)--(10, 6)--(10, 8)--(9, 7); 
\fill[opacity=1, gray] (11, 8)--(11, 7)--(12, 7)--(12, 8)--(11, 8); 
\fill[opacity=1, gray] (13, 9)--(13, 8)--(14, 8)--(14, 10)--(13, 9); 
\fill[opacity=1, gray] (15, 11)--(15, 9)--(16, 9)--(16, 12)--(15, 11); 
\fill[opacity=1, gray] (17, 12)--(17, 10)--(18, 10)--(18, 12)--(17, 12); 
\fill[opacity=1, gray] (19, 12)--(19, 11)--(20, 11)--(20, 12)--(19, 12); 
\fill[opacity=1, gray] (21, 13)--(21, 12)--(22, 12)--(22, 14)--(21, 13); 
\fill[opacity=1, gray] (23, 14)--(23, 13)--(24, 13)--(24, 14)--(23, 14); 
\fill[opacity=1, gray] (25, 15)--(25, 14)--(26, 14)--(26, 16)--(25, 15); 
\fill[opacity=1, gray] (27, 16)--(27, 15)--(28, 15)--(28, 16)--(27, 16); 
\fill[opacity=1, gray] (29, 17)--(29, 16)--(30, 16)--(30, 18)--(29, 17); 
\fill[opacity=1, gray] (31, 18)--(31, 17)--(32, 17)--(32, 18)--(31, 18); 
\fill[opacity=1, gray] (33, 19)--(33, 18)--(34, 18)--(34, 20)--(33, 19); 
\fill[opacity=1, gray] (35, 21)--(35, 19)--(36, 19)--(36, 22)--(35, 21); 
\fill[opacity=1, gray] (37, 22)--(37, 20)--(38, 20)--(38, 22)--(37, 22); 
\fill[opacity=1, gray] (39, 22)--(39, 21)--(40, 21)--(40, 22)--(39, 22); 
\fill[opacity=1, gray] (41, 23)--(41, 22)--(42, 22)--(42, 24)--(41, 23);
\draw[-, dashed, thick](22,0)--(22,16)--(16,16)--(0,0);
\draw[-, thick](42,0)--(42,26)--(36,26)--(20,10)--(20,0);
\draw[->, thick] (0,0)--(44,0) node[below]{$q_1$};
\draw[->, thick] (0,0)--(0,42) node[left]{$q_2$};
\foreach \x in {4,8,...,40}
       {\draw[-,line width=0.5mm] (\x,0.2)--(\x,-0.2) %
              node[below]{$\scriptstyle \x$}; }  
\foreach \y in {0,4,...,40}
         {\draw[-,line width=0.5mm] (0.2,\y)--(-0.2,\y) 
              node[left]{$\scriptstyle \y$};}
\end{tikzpicture}
\caption{The graph of $P_1(q_1,q_2)$ for $0\le q_2\le q_1\le 2F_7=42$}
\label{properties:fig2}
\end{figure}

The smallest admissible polygons are the admissible triangles with horizontal 
and vertical sides of length $1$, namely the triangles with set of vertices
 \[
  V=\{(m,n), (m+1,n),(m+1,n+1)\} 
  \quad\text{or}\quad
  V'=\{(m,n), (m,n+1),(m+1,n+1)\}
\]
for some $(m,n)\in\bZ^2$.  We call them the \emph{basic triangles}.
We say that a basic triangle $\cT$ is respectively of \emph{lower} or 
\emph{upper} type if its set of vertices is of the form $V$ or $V'$
respectively, that is if $\cT$ lies respectively below or above its 
hypotenuse.

The basic triangles are pairwise compatible, each basic triangle 
of a type sharing common sides with 3 basic triangles of 
the other type (see Figure~\ref{properties:fig3}).  They form a partition 
of $\bR^2$ as defined in section~\ref{sec:P}.  In general,  the set 
of basic triangles contained in any given admissible polygon $\cA$ 
is a partition of $\cA$.

\begin{lemma}
\label{properties:lemma1}
Let $\cT$ be a basic triangle and let $(a,b,c)\in\bZ^3$.  Then the map 
from $\cT$ to $\bR^3$ sending each $\uq=(q_1,q_2)\in\cT$ to
$\Phi(q_1-a,q_2-b,c)$ is affine.
\end{lemma}

\begin{proof}
For any interior point $\uq=(q_1,q_2)$ of $\cT$, none of the 
numbers $q_1$, $q_2$ and $q_1-q_2$ is an integer, so the coordinates of 
$(q_1-a,q_2-b,c)$ are all distinct, and their order is independent of
$\uq$.    Hence, there is a linear map $\sigma\colon\bR^3\to\bR^3$ that 
permutes the coordinates in $\bR^3$ such that 
$\Phi(q_1-a,q_2-b,c)=\sigma(q_1-a,q_2-b,c)$ for each interior point
of $\cT$, and thus, by continuity, for each point of $\cT$.
\end{proof}

By Proposition~\ref{properties:prop1}, the next two results apply
to the map $\uP\colon\cA(\epsilon)\to\bR^3$ of section~\ref{sec:P}.

\begin{lemma}
\label{properties:lemma2}
Let $\cA$ be an  admissible polygon, let $T$ denote the set of 
basic triangles contained in $\cA$, and let 
$\uP=(P_1,P_2,P_3)\colon\cA\to\bR^3$ be a function.
Suppose that, for each $\cT\in T$, there exists $(a,b,c)\in\bZ^3$ with 
$c=a+b$ such that
\begin{equation}
\label{properties:lemma2:eq1}
 \uP(q_1,q_2)=\Phi(q_1-a,q_2-b,c) \quad \text{for any $(q_1,q_2)\in\cT$.}
\end{equation}
Then $\uP$ is continuous and satisfies 
\begin{equation}
\label{properties:lemma2:eq2}
 P_1(\uq)+P_2(\uq)+P_3(\uq)=q_1+q_2 \quad
 \text{for any $\uq=(q_1,q_2)\in\cA$.}
\end{equation}
If $\cA$ is convex then $\uP$ is $1$-Lipschitz.
\end{lemma}

\begin{proof}
By \eqref{properties:lemma2:eq1}, the map $\uP$ is continuous on each $\cT\in T$ 
(for the relative topology of $\cT$).  Thus, it is continuous on $\cA$ because, for 
each $\uq\in\cA$, there are at most six triangles $\cT\in T$ with $\uq\in\cT$,
and their union is a neighbourhood of $\uq$.   
Formula \eqref{properties:lemma2:eq1} also implies 
\eqref{properties:lemma2:eq2} for each $\uq\in\cT$ with $\cT\in T$, thus for each
$\uq\in \cA$.  Finally, suppose that $\cA$ is convex and let $\uq,\uq'\in\cA$.  
Write $\uq'-\uq=(u_1,u_2)$.  Then, for each $i=1,2,3$, the function 
$f_i\colon[0,1]\to\cA$ given by $f_i(t)=P_i(\uq+t(\uq'-\uq))$ is continuous
and piecewise linear with slopes $0$, $u_1$ or $u_2$, thus
\[
 |P_i(\uq')-P_i(\uq)| = |f_i(1)-f_i(0)| \le \max\{|u_1|,|u_2|\} = \norm{\uq'-\uq},
\]
and so $\norm{\uP(\uq')-\uP(\uq)}\le \norm{\uq'-\uq}$.
\end{proof}

\begin{lemma}
\label{properties:lemma3}
Let the notation and hypotheses be as in Lemma~\ref{properties:lemma2}. 
Suppose that $\cT\in T$ is a lower basic triangle, that $\cT'\in T$ is an 
upper basic triangle, and that they share a common side $\cE$.    Let
$(a,b,c)\in\bZ^3$ with $c=a+b$ such that \eqref{properties:lemma2:eq1}
holds, and let $(a',b',c')\in\bZ^3$ with $c'=a'+b'$ such that
\[
 \uP(q_1,q_2)=\Phi(q_1-a',q_2-b',c') \quad \text{for each $(q_1,q_2)\in\cT'$.}
\]
Suppose further that $(a',b',c')\neq (a,b,c)$.  Then,
\begin{itemize} 
\item[(i)] $(a',b',c')=(m-c,b,m-a)$ if $\cE$ is contained in the vertical 
     line $q_1=m$; 
\item[(ii)] $(a',b',c')=(a,m-c,m-b)$ if $\cE$ is contained in the horizontal 
     line $q_2=m$;
\item[(ii)] $(a',b',c')=(b+m,a-m,c)$ if $\cE$ is contained in the line $q_1-q_2=m$. 
\end{itemize}
\end{lemma}

The three cases are illustrated in Figure~\ref{properties:fig3}.  
In each, the formula given for $(a',b',c')$ follows 
from the hypotheses that $(a',b',c')\neq (a,b,c)$ and that 
$\Phi(q_1-a,q_2-b,c)=\Phi(q_1-a',q_2-b',c')$ for all $(q_1,q_2)\in\cE$.

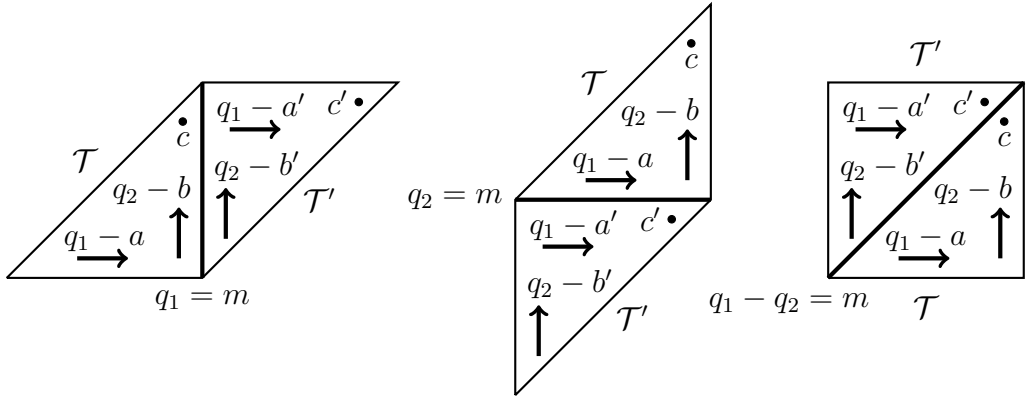
\begin{figure}[ht]
\begin{tikzpicture}[scale=0.52]
 \draw[-, thick] (-1,3)--(4,3)--(9,8)--(4,8)--(-1,3);
 \draw[-, ultra thick] (4,8)--(4,3) node[below]{$q_1=m$};
 \node at (1,6){$\cT$};
 \node at (7,5){$\cT'$};
 \node[draw,circle,inner sep=1pt, fill] at (3.5,7){};
 \node[below] at (3.5,7) {$c$};
 \draw[->, ultra thick] (0.8,3.5)--(2.1,3.5){}; 
 \node[above] at (1.5,3.4) {$q_1-a$};
 \draw[->, ultra thick] (3.4,3.5)--(3.4,4.8){};  
 \node[above] at (2.7,4.6) {$q_2-b$};  
 \node[draw,circle,inner sep=1pt, fill] at (8,7.5){};
 \node[left] at (8,7.5) {$c'$};
 \draw[->, ultra thick] (4.7,6.8)--(6,6.8){};
 \node[above] at (5.5,6.7) {$q_1-a'$};
 \draw[->, ultra thick] (4.6,4.0)--(4.6,5.3){};
 \node[above] at (5.4,5.2) {$q_2-b'$};
 \draw[-, thick] (12,0)--(17,5)--(17,10)--(12,5)--(12,0);
 \draw[-, ultra thick] (17,5)--(12,5) node[left]{$q_2=m$};
 \node at (14,8){$\cT$};
 \node at (15,2){$\cT'$};
  \node[draw,circle,inner sep=1pt, fill] at (16.5,9){};
 \node[below] at (16.5,9) {$c$};
 \draw[->, ultra thick] (13.8,5.5)--(15.1,5.5){};
 \node[above] at (14.5,5.4) {$q_1-a$};
 \draw[->, ultra thick] (16.4,5.5)--(16.4,6.8){};
 \node[above] at (15.7,6.6) {$q_2-b$};
 \node[draw,circle,inner sep=1pt, fill] at (16,4.5){};
 \node[left] at (16,4.5) {$c'$};
 \draw[->, ultra thick] (12.7,3.8)--(14,3.8){};  
 \node[above] at (13.5,3.7) {$q_1-a'$};
 \draw[->, ultra thick] (12.6,1.0)--(12.6,2.3){};
 \node[above] at (13.4,2.2) {$q_2-b'$};
 \draw[-, thick] (20,3)--(25,3)--(25,8)--(20,8)--(20,3);
 \draw[-, ultra thick] (25,8)--(20,3);
 \node[below] at (19,3) {$q_1-q_2=m$};
 \node[above] at (22.5,8.2){$\cT'$};
 \node[below] at (22.5,2.8){$\cT$};
 \node[draw,circle,inner sep=1pt, fill] at (24.5,7){};
 \node[below] at (24.5,7) {$c$};
 \draw[->, ultra thick] (21.8,3.5)--(23.1,3.5){};
 \node[above] at (22.5,3.4) {$q_1-a$};
 \draw[->, ultra thick] (24.4,3.5)--(24.4,4.8){};
 \node[above] at (23.7,4.6) {$q_2-b$};
 \node[draw,circle,inner sep=1pt, fill] at (24,7.5){};
 \node[left] at (24,7.5) {$c'$};
 \draw[->, ultra thick] (20.7,6.8)--(22,6.8){};
 \node[above] at (21.5,6.7) {$q_1-a'$};
 \draw[->, ultra thick] (20.6,4.0)--(20.6,5.3){};
 \node[above] at (21.4,5.2) {$q_2-b'$};
\end{tikzpicture}
\caption{Basic triangles meeting in a common side}
\label{properties:fig3}
\end{figure}

We propose the following notion.

\begin{definition}
\label{properties:def}
An integral  $2$-parameter $3$-system is a function 
$\uP\colon\cA\to\bR^3$ as in Lemma~\ref{properties:lemma2},
such that, under the hypotheses of Lemma~\ref{properties:lemma3},
we have $c'>c$ in case (i), $b'<b$ in case (ii), and $a'<a$ in case (iii).
\end{definition}

It is easily seen that, if $\uP\colon\bR^2\to\bR^3$ is an integral $2$-parameter 
$3$-system on $\bR^2$, then the maps $q\mapsto P(q,0)$ and $q\mapsto P(0,q)$
are both integral $3$-systems in the sense of \cite{R2003}.   We conclude
with the following result.

\begin{proposition}
\label{properties:prop2}
The map $\uP\colon\cA(\epsilon)\to\bR^3$ of section~\ref{sec:P} is an integral 
$2$-parameter $3$-system. 
\end{proposition}

\begin{proof} 
Let the notation and hypotheses be as in Lemma~\ref{properties:lemma3},
and let $\cR,\cR'\in S$ such that $\cT\subseteq\cR$ and $\cT'\subseteq\cR'$.  
Since $(a',b',c')\neq(a,b,c)$, Proposition~\ref{properties:prop1} gives 
$\cR\neq \cR'$.  So the common side $\cE$ of $\cT$ and $\cT'$ is contained 
in a common side of $\cR$ and $\cR'$.  
If $\cR=\trap(\epsilon, w_1)$, we have $\cR'=\trap(w_1,w_2)$,
$(a,b,c)=(0,1,1)$, $(a',b',c')=(1,1,2)$, and case (i) applies with $m=2$.
If $\cR=\trap(w_1, w_2)$, we have $\cR'=\trap(w_2,w_3)$,
$(a,b,c)=(1,1,2)$, $(a',b',c')=(2,1,3)$, and case (i) applies with $m=4$.
In both instances, we note that $c'>c$, as needed.    
If $\cR$ is not one of these trapezes,
then, by Proposition~\ref{properties:prop1}, the functions
$q_1-a$, $q_2-b$ and $c$ coincide in an interior point of $\cR$.  So, in case (i), 
we have $q_1-a=c$ for some $q_1<m$, thus $c'=m-a>c$.  In case (ii), 
we have $q_2-b=c$ for some $q_2>m$, thus $b'=m-c<b$. In case (iii), we have
$q_1-a=q_2-b$ for some $(q_1,q_2)\in\cR$ with $q_1-q_2>m$, thus
$a-b>m$ and so $a'=b+m<a$.
\end{proof}

It would be interesting to know if, for any matrix $A\in\GL_3(\bR)$, there exists an
integral $2$-parameter $3$-system $\uP\colon\bR^2\to\bR^3$ such that
$\uL_A(0,q_1,q_2)-\uP(q_1,q_2)$ is a bounded function of $(q_1,q_2)\in\bR^2$,
and conversely if, for any integral $2$-parameter $3$-system 
$\uP\colon\bR^2\to\bR^3$, there exists $A\in\GL_3(\bR)$ with the same 
property.

%
%

\section{Proof of Theorem \ref{main:P:thm1}}
\label{sec:proof1}

Fix an integer $k\ge 4$.  Since $q_2(w_{k-1})=q_2(w_{k-1}w_{k-3})=2F_{k-1}$, 
the conditions \eqref{main:P:thm1:eq1} on a point $\uq=(q_1,q_2)\in\bR^2$ 
amount to
\begin{equation}
\label{proof1:eq1}
  \uq\in \trap(\epsilon,w_{k-2})\cup \trap(w_{k-2},w_{k-1})
         \cup \trap(w_{k-1},w_{k-1}w_{k-3}).
\end{equation}
Fix such a point $\uq$.  We need to show that
\[
 \uP(\up+\uq)=\ur+\uP(\uq) 
 \ \ \text{where}\ \ 
 \up=(4F_{k-2},2F_{k-2})
 \ \text{and} \ 
 \ur=(2F_{k-2},2F_{k-2},2F_{k-2}).
\]

To this end, choose a trapeze $\cT$ in the right hand side of 
\eqref{proof1:eq1} for which $\uq\in\cT$.  By Proposition~\ref{P:prop4},
there exists $\cR\in S$ such that $\uq\in\cR\subseteq\cT$.  So,
$\pi_1(\cR)=[q_1(u),q_1(w)]$ for some $u<w$ in $[\epsilon, w_{k-1}w_{k-3}]$.
Moreover, since 
$\cR \subseteq \cT \subseteq\layer(k-2)$, Proposition~\ref{P:prop3} 
implies that $\cR\in S_\ell$ for some $\ell\in\{1,\dots,k-3\}$.  
Let $u'$ and $w'$ be the prefixes of $w_\infty$ with
\begin{equation}
\label{proof1:eq3}
 |u'|=|u|+2F_{k-2} \et |w'|=|w|+2F_{k-2}.
\end{equation}
We claim that there exists $\cR'\in S_\ell$ such that 
\begin{equation}
\label{proof1:eq4}
 \pi_1(\cR')=[q_1(u'),q_1(w')] \et \up+\cR\subseteq \cR'.
\end{equation}
If we admit this, then $\up+\uq\in\cR'$ and, using \eqref{proof1:eq3} 
and Proposition~\ref{P:prop5}, we find
\begin{align*}
 \uP(\up+\uq)
  &=\Phi(q_1+4F_{k-2}-|u'|,q_2+2F_{k-2}+|u'|-|w'|,|w'|) \\
  &=\Phi(q_1-|u|+2F_{k-2},q_2+|u|-|w|+2F_{k-2},|w|+2F_{k-2})  = \ur+\uP(\uq),
\end{align*}
as needed.  

To prove the claim, we first note that, by Lemma~\ref{Vbar:lemma6}, we 
have $|\alpha(u)|\le |\alpha(u')|$ and $|\alpha(w)|\le |\alpha(w')|$, thus
$\up+\uq(u)\in \cA(u')$ and $\up+\uq(w) \in \cC(w')$, so
\begin{equation}
\label{proof1:eq5}
 \up+\cA(u) \subseteq \cA(u') \et \up+\cC(w) \subseteq \cC(w').
\end{equation}
If $\ell=1$, the words $u<w$ are consecutive prefixes of $w_\infty$,
and we have $\cR=\trap(u,w)$.  Then, by \eqref{proof1:eq3},  the words
$u'<w'$ are also consecutive prefixes of $w_\infty$ and we may form 
$\cR'=\trap(u',w')\in S_1$.  It satisfies \eqref{proof1:eq4} because by
\eqref{proof1:eq5}, we have
\[
 \up+\cR = (\up+\cA(u))\cap(\up+\cC(w))
   \subseteq \cA(u')\cap\cC(w') = \cR'.
\]
Otherwise, we 
have $\ell\ge 2$ and, by Lemma~\ref{P:lemma:unicity}, there exists 
$v\in\cV_\ell\setminus\cV_{\ell+1}$ such that $u<v<w$ are consecutive 
elements of $\cV_\ell$, and $\cR=\cell(u,v,w)$.   We note that
$v\notin\cF$ because the only element of $\cF$ in 
$\cV_\ell\setminus\cV_{\ell+1}$ is $w_\ell$ and $v>u\ge w_\ell$.  Since 
we also have $v<w\le w_{k-1}w_{k-3}$, we deduce from 
Lemma~\ref{Vbar:lemma6} that the prefix $v'$ of $w_\infty$ with
$|v'|=|v|+2F_{k-2}$ satisfies $|\alpha(v)| = |\alpha(v')|$, thus 
$v'\in\cV_\ell\setminus\cV_{\ell+1}$ and $\up+\uq(v)=\uq(v')$, so
\begin{equation}
\label{proof1:eq6}
 \up+\cB(v) = \cB(v').
\end{equation}
By Proposition~\ref{Vbar:prop1}, the words $u'<v'<w'$ are consecutive 
elements of 
$\cV_\ell$, so we may form $\cR'=\cell(u',v',w')\in S_\ell$.  On the basis of
\eqref{proof1:eq5} and \eqref{proof1:eq6}, we conclude as above
that \eqref{proof1:eq4} holds.

%
%

\section{Proof of Theorem~\ref{main:P:thm2}}
\label{sec:proof2}

We begin by establishing three lemmas, the first two of which concern the 
morphism of monoids $\theta\colon E^*\to E^*$ defined in 
section~\ref{main:ssec:Fib}.

\begin{lemma}
\label{proof2:lemma1}
For each $v\in\left[\epsilon,w_\infty\right[$, we have 
$\big| |\theta(v)|-\gamma|v|\big| \le 2$.
\end{lemma}

\begin{proof} 
More precisely, we will show, by induction on $|v|$, that
\begin{equation}
\label{proof2:lemma1:eq1}
\big| |\theta(v)|-\gamma|v|\big| + \max\{1,|v|\}^{-1} \le 2
\end{equation}
for each $v\in\left[\epsilon,w_\infty\right[$.  If $|v|=0$, then 
$\epsilon=v=\theta(v)$, so $|\theta(v)|=0$ and 
\eqref{proof2:lemma1:eq1} holds.  If $|v|=F_k$ for some positive 
integer $k$, then $v=w_k$ and $\theta(v)=w_{k+1}$, so $|\theta(v)|=F_{k+1}$.
Since $F_{k+1}/F_k$ is a convergent of $\gamma$ in reduced form, we have
\[
 |F_{k+1}-\gamma F_k| \le 1/F_k.
\]
Thus, the left hand side of \eqref{proof2:lemma1:eq1} is at most 
$2/F_k\le 2$, and we are done.  

We may therefore assume that $F_k<|v|<F_{k+1}$ for some  
$k\ge 3$, thus $v=w_ku$ with $u\in\left]\epsilon,w_{k-1}\right[$.  
Then, we have $\theta(v)=w_{k+1}\theta(u)$, and so
\[
\big| |\theta(v)|-\gamma|v|\big| 
  \le |F_{k+1}-\gamma F_k| + \big| |\theta(u)|-\gamma |u|\big|.
\] 
By induction, we may assume that $\big| |\theta(u)|-\gamma |u|\big|
\le 2-1/|u|$.  So, we obtain 
\[
\big| |\theta(v)|-\gamma|v|\big| + 1/|v| 
\le 1/F_k+2-1/|u|+1/|v| = 1/F_k + 2 - F_k/(|u|\,|v|)
\]
because $|v|-|u|=F_k$.  Since $|u|\,|v| < F_{k-1}F_{k+1} = F_k^2\pm 1$, 
we also have $|u|\,|v| \le F_k^2$ and the desired estimate 
\eqref{proof2:lemma1:eq1} follows.
\end{proof}

\begin{lemma}
\label{proof2:lemma2}
Let $v\in\cV_4$ and let $v'=\theta(v)$.  Then, 
$\norm{\uq(v') - \gamma \uq(v)} \le 4$.
\end{lemma}

\begin{proof} 
Since $v\in\cV_4$, Corollary~\ref{V:prop2:cor} gives $\alpha(v')=\theta(\alpha(v))$.
Thus, by Lemma~\ref{proof2:lemma1},
\[
 \norm{\uq(v') - \gamma \uq(v)}
 \le 2\max\left\{ \big||v'|-\gamma|v|\big|, 
          \big| |\alpha(v')|-\gamma|\alpha(v)|\big| \right\} \le 4.
\qedhere
\]
\end{proof}

\begin{lemma}
\label{proof2:lemma3}
Let $\uq=(q_1,q_2)\in\cA(\epsilon)$ with $q_2\le q_1/2$.  Then,   
$\norm{\uP(\uq)-(q_2,q_1/2,q_1/2)} \le 1$.
\end{lemma}

\begin{proof}
Choose consecutive words $u<v$ in $[\epsilon,w_\infty[$ with
$q_1(u)\le q_1 \le q_1(v)$.  Then, $\uq$ belongs to $\trap(u,v)$ because 
\[
 \min\{q_2(u)+q_1-q_1(u), q_2(v)\}
  \ge \min\{q_1-|u|, |v|\} \ge q_1/2\ge q_2.
\]
Thus, by Theorem~\ref{P:thm}, we have $\uP(\uq)=\Phi(\ur)$ 
where $\ur=(q_1-|u|, q_2-1,|v|)$.   As $\Phi$ is $1$-Lipschitz, this yields
\[
 \norm{\uP(\uq)-(q_2,q_1/2,q_1/2)} 
  = \norm{\Phi(\ur)-\Phi(q_1/2,q_2,q_1/2)} 
  \le \norm{\ur-(q_1/2,q_2,q_1/2)} \le 1.
\qedhere
\]
\end{proof}

\subsection*{Proof of Theorem~\ref{main:P:thm2}}
Let $\uq=(q_1,q_2)\in\cA(\epsilon)$.  We need to show that
\[
 \norm{\uP(\gamma\uq) - \gamma\uP(\uq)} \le 40.  
\]

If $q_2\le q_1/2$, Lemma~\ref{proof2:lemma3} applies to $\uq$ 
and to $\gamma\uq$.  Thus, setting $\ur=(q_2,q_1/2,q_1/2)$, we find
\[
 \norm{\uP(\gamma\uq)-\gamma\uP(\uq)}
  \le \norm{\uP(\gamma\uq)-\gamma\ur}
     +\gamma\norm{\ur-\uP(\uq)}\\ 
  \le 1+\gamma.
\]

Suppose now that $q_2>q_1/2$ and $\uq\in\layer(4)$.   The above 
special case gives
\begin{equation}
\label{proof2:proof:eq1}
 \norm{\uP(\gamma\uq') - \gamma\uP(\uq')} \le 1+\gamma
 \quad\text{where}\quad \uq'=(q_1,q_1/2).  
\end{equation}
Since $\uq\in\layer(4)$, we have $\uq\in\trap(u,v)$ for consecutive 
$u<v$ in $\{\epsilon\}\cup\cV_4$, and then
\[
 2|u|\le q_1\le 2|v| \et q_2\le \min\{q_1-|u|+|\alpha(u)|, |v|+|\alpha(v)|\}.
\]
If $u\neq \epsilon$, Corollary~\ref{Vbar:prop2:cor1} gives $|v|-|u|\le F_3=3$,
and Corollary~\ref{Vbar:prop2:cor2} gives $\{u,v\}\not\subseteq \cV_6$, 
thus $\min\{|\alpha(u)|,|\alpha(v)|\} \le F_5=8$, and so
\[
 q_2-q_1/2 \le \max\{q_1/2-|u|, |v|-q_1/2\} +8 \le |v|-|u|+8 \le 11.
\]
If $u=\epsilon$, we have $v=w_4$, thus $q_2-q_1/2\le q_1/2\le |v| = 5$.
Hence, in all cases, we have $0\le q_2-q_1/2 \le  11$.  As $\uP$ is $1$-Lipschitz
by Proposition~\ref{properties:prop1} and Lemma~\ref{properties:lemma2}, 
we deduce that
\[
 \norm{\uP(\uq) - \uP(\uq')}\le 11
 \et
 \norm{\uP(\gamma\uq) - \uP(\gamma\uq')} \le 11\gamma.
\]
Using \eqref{proof2:proof:eq1}, we conclude that  
$\norm{\uP(\gamma\uq) - \gamma\uP(\uq)} \le 1+23\gamma\le 40$, as claimed.

Finally, suppose that $\uq\notin\layer(4)$.  By Proposition~\ref{P:prop2},
there exist an integer $\ell\ge 4$ and a polygon $\cR\in S_\ell$ such that
$\uq\in\cR$.  Let $u<v<w$ be consecutive elements of $\cV_\ell$ with 
$v\notin\cV_{\ell+1}$ such that $\cR=\cell(u,v,w)$, and set
\[
 u'=\theta(u), \quad v'=\theta(v) \et w'=\theta(w).
\]
By Proposition~\ref{V:prop2}, the words $u'<v'<w'$ are consecutive 
elements of $\cV_{\ell+1}$ with $v'\notin\cV_{\ell+2}$.  Thus, we may 
form $\cR'=\cell(u',v',w')\in S_{\ell+1}$.  Using Lemma~\ref{proof2:lemma2}, 
we find
\begin{align*}
 \gamma\cA(u) &= \gamma\uq(u)-\uq(u')+\cA(u')
     \subseteq [-4,4]^2+\cA(u'),\\
 \gamma\cB(v) &= \gamma\uq(v)-\uq(v')+\cB(v') 
     \subseteq [-4,4]^2+\cB(v'),\\
 \gamma\cC(w) &= \gamma\uq(w)-\uq(w')+\cC(w') 
     \subseteq [-4,4]^2+\cC(w').
\end{align*}
Taking term by term intersections, we deduce that
\[
  \gamma\cR \subseteq [-8,8]^2+\cR'
\]
because the sides of $\cA(u')$, $\cB(v')$ and $\cC(w')$ are horizontal,
vertical, or have slope $1$.  Thus, there exists $\uq'=(q'_1,q'_2)\in\cR'$ 
such that 
\begin{equation}
\label{proof2:proof:eq3}
  \norm{\gamma\uq-\uq'}\le 8. 
\end{equation}
As $\uP$ is $1$-Lipschitz, this implies that
\begin{equation}
\label{proof2:proof:eq4}
 \norm{\uP(\gamma\uq)-\uP(\uq')} \le 8.
\end{equation}
Since $\uq\in\cR$ and $\uq'\in\cR'$, Theorem~\ref{P:thm} gives
$\uP(\uq)=\Phi(\ur)$ and $\uP(\uq')=\Phi(\ur')$ where
\[
\ur=(q_1-|u|,q_2+|u|-|w|,|w|)
 \et
 \ur'=(q'_1-|u'|,q'_2+|u'|-|w'|,|w'|).
\]
As $\Phi$ is $1$-Lipschitz and commute with scalar multiplication, we 
deduce that
\[
 \norm{ \uP(\uq') -\gamma \uP(\uq) }  
   \le \norm{\ur'-\gamma\ur}\le 12,
\]
where the second estimate uses \eqref{proof2:proof:eq3} 
and Lemma~\ref{proof2:lemma1}.  Combining this with 
\eqref{proof2:proof:eq4}, we conclude that  
$\norm{\uP(\gamma\uq) - \gamma\uP(\uq)} \le 20$.


%
%

\section{A specific example}
\label{sec:ex}

Using the terminology of section~\ref{main:ssec:extremal}, we show 
below that the real number $\xi$ given by \eqref{main:extremal:eq:example} 
belongs to $\cE_3^+$,  by computing explicitly an associated Fibonacci 
sequence in $\SL_2(\bZ)$.  Although this is not needed for the rest of the 
paper, it provides a concrete example of a number $\xi$ to which our results
apply, independently of \cite{R2008}.  We start by recalling some general 
facts from continued fraction theory.

Let $\sigma\colon(\bN\setminus\{0\})^*\to\GL_2(\bZ)$ be the morphism of
monoids such that
\[
 \sigma(a)=\matrice{a}{1}{1}{0} \quad\text{for each $a\in\bN\setminus\{0\}$.}
\]
By \cite[Corollary~4.2]{R2008}, its image is $\cS\cup\{I\}$ where $I$ is the 
identity of $\GL_2(\bZ)$ and
\[
 \cS=\Big\{ \matrice{a}{b}{c}{d}\in\GL_2(\bZ) \,;\, 
     a\ge \max\{b,c\}\ \text{and}\ \min\{b,c\}\ge d\ge 0\Big\}.
\]
Moreover, its restriction $\sigma\colon(\bN\setminus\{0\})^*\to\cS\cup\{I\}$ 
is an isomorphism of monoids.  It follows that, for each word 
$v \in (\bN \setminus \{0\})^*$, the matrix $\sigma(v)$ is symmetric if 
and only if $v$ is a palindrome.

Let $(a_i)_{i\ge 1}$ be a sequence of positive integers and let 
$\xi\in\left]0,1\right[$ be the real number with continued fraction expansion
\[
 \xi=[0,a_1,a_2,a_3\dots]=1/(a_1+1/(a_2+\cdots)).
\]
For each integer $k\ge 1$, we have
\[
 \sigma(a_1\cdots a_k)
   =\matrice{a_1}{1}{1}{0}\cdots\matrice{a_k}{1}{1}{0}
   =\matrice{q_k}{q_{k-1}}{p_k}{p_{k-1}}
\]
where $p_k/q_k=[0,a_1,\dots,a_k]$ is the $k$-th convergent of $\xi$ 
in reduced form (and $p_0/q_0=0$), thus
\begin{equation}
\label{ex:eq1}
 \norm{(\xi,-1)\sigma(a_1\cdots a_k)}
   = |q_{k-1}\xi-p_{k-1}|
   \asymp q_k^{-1} 
   = \norm{\sigma(a_1\cdots a_k)}^{-1}
\end{equation}
with absolute implied constants (see  \cite[Chapter~I]{Sc1980}).  

\begin{proposition}
\label{ex:prop}
Let $\xi=[0,\uun,f_{\udeux,\uun}]$ where $\uun=(1,1)$ and $\udeux=(2,2)$.  Then 
$\xi$ belongs to $\cE_3^+$ with associated Fibonacci sequence $(W_i)_{i\ge 0}$
in $\SL_2(\bZ)$ starting with
\begin{equation}
\label{ex:prop:eq1}
 W_0=\sigma(\uun)=\matrice{2}{1}{1}{1}
 \et
 W_1=\sigma(\uun)\sigma(\udeux)\sigma(\uun)^{-1}
        =\matrice{7}{-2}{4}{-1}.
\end{equation}
Moreover, let $E=\{a,b\}$ be an alphabet on two letters, let $w_\infty=f_{a,b}$,
and let $\varphi\colon E^*\to\SL_2(\bZ)$ be the morphism of monoids such 
that $\varphi(a)=W_1$ and $\varphi(b)=W_0$.  Then, for each 
$v \in \left[ \epsilon, w_\infty \right[$, we have
\begin{equation}
\label{ex:prop:eq2}
 \norm{(\xi,-1)\varphi(v)} \asymp \norm{\varphi(v)}^{-1}.
 \end{equation}
\end{proposition}

In \eqref{ex:prop:eq2} and in the proof below, all
implied constants are explicitly computable numbers.

\begin{proof}
We first note that, for each $v\in E^*$, we have
\begin{equation}
\label{ex:prop:eq3}
 \varphi(v)=\sigma(\uun)\sigma(\tau(v))\sigma(\uun)^{-1}
\end{equation}
where $\tau\colon E^*\to(\bN\setminus\{0\})^*$ is the morphism of
monoids such that $\tau(a)=\udeux$ and $\tau(b)=\uun$.  Indeed,
both sides of this equality define morphisms from $E^*$ to 
$\SL_2(\bZ)$ which, in view of \eqref{ex:prop:eq1} agree for $v=a$ 
and $v=b$.   We also note that $W_i=\varphi(w_i)$ for each $i\ge 0$,
where $(w_i)_{i\ge 0}$ is the Fibonacci sequence in $E^*$ with
$w_0=b$ and $w_1=a$. Thus, \eqref{ex:prop:eq3} yields
\begin{equation}
\label{ex:prop:eq4}
 W_i=\sigma(\uun)\sigma(\tau(w_i))\sigma(\uun)^{-1}  
 \quad (i\ge 0).
\end{equation}

For $v \in \left[ \epsilon, w_\infty \right[$, the word
$(\uun, \tau(v))$ is a prefix of $(\uun,f_{\udeux,\uun})$ and so,
using \eqref{ex:prop:eq3}, the general estimate \eqref{ex:eq1} yields
\[
 \norm{(\xi,-1)\varphi(v)} 
  \asymp \norm{(\xi,-1)\sigma(\uun, \tau(v))} 
  \asymp \norm{\sigma(\uun, \tau(v))}^{-1}
  \asymp \norm{\varphi(v)}^{-1},
\]
which proves \eqref{ex:prop:eq2}.  For the choice of $v=w_i$, this gives
$\norm{(\xi,-1)W_i} \asymp \norm{W_i}^{-1}$ for each $i\ge 0$.  Thus, 
the sequence $(W_i)_{i\ge 0}$ fulfils condition (E3) from 
section~\ref{main:ssec:extremal}.

By \eqref{ex:prop:eq4}, we have $\|W_i\| \asymp \|\sigma(\tau(w_i))\|$ for 
each $i\ge 0$.  Since $\|A\|\,\|B\| \le \|AB\| \le 2\|A\|\,\|B\|$ for any 
$A,B\in\cS$ and since $\sigma(\tau(w_i))\in\cS$ for each $i\ge 1$, we 
deduce that 
\begin{align*}
 \norm{W_{i+2}} \asymp \norm{\sigma(\tau(w_{i+2}))} 
   &= \norm{\sigma(\tau(w_{i+1}))\sigma(\tau(w_i))} \\
   &\asymp \norm{\sigma(\tau(w_{i+1}))}\norm{\sigma(\tau(w_i))}
   \asymp \norm{W_{i+1}}\norm{W_i}
\end{align*}
for each $i\ge 0$.  Thus, $(W_i)_{i\ge 0}$ is unbounded and fulfils condition (E2).

Finally, we note that 
\begin{equation}
\label{ex:prop:eq5}
 \sigma(\udeux)=\matrice{5}{2}{2}{1}=\sigma(\uun)M={}^tM\sigma(\uun)
 \quad \text{where} \quad 
 M=\matrice{3}{1}{-1}{0}.
\end{equation}
For each $i\ge 0$, let $\ux_i=W_iM_i^{-1}$ where $M_i$ is defined as in 
condition (E1) for the above choice of $M$.  For $i\ge 2$, we claim that
\begin{equation}
\label{ex:prop:eq6}
 \ux_i=\sigma(\uun, \tau(w_i^{**}), \uun). 
\end{equation}
Indeed, if $i$ is even, $w_i$ ends in $ab$.  Then,
using \eqref{ex:prop:eq4} and \eqref{ex:prop:eq5}, we find that
\[
 \ux_i=\sigma(\uun)\sigma(\tau(w_i^{**}))\sigma(\udeux)M^{-1}
     =\sigma(\uun)\sigma(\tau(w_i^{**}))\sigma(\uun).
\]
Otherwise, $i$ is odd, $w_i$ ends in $ba$, and we find similarly that
\[
 \ux_i
  =\sigma(\uun)\sigma(\tau(w_i^{**}))\sigma(\uun)
         \sigma(\udeux)\sigma(\uun)^{-1}({}^tM)^{-1}
  =\sigma(\uun)\sigma(\tau(w_i^{**}))\sigma(\uun).
\]
Moreover, $w_i^{**}$ is a palindrome in $E^*$ by Lemma~\ref{V:lemma2}.  
Since $\tau(a)=\udeux$ and $\tau(b)=\uun$
are palindromes in $(\bN\setminus\{0\})^*$, we deduce that 
$(\uun, \tau(w_i^{**}), \uun)$ is a palindrome in
$(\bN\setminus\{0\})^*$.  So, by \eqref{ex:prop:eq6}, the matrix 
$\ux_i$ is symmetric for $i\ge 2$.  A direct computation shows that 
$\ux_0$ and $\ux_1$ 
are also symmetric.  Thus, (E1) is fulfilled as well and so $\xi\in\cE_3^+$.
\end{proof}

A Markoff triple is a solution in positive integers $\um=(m,m_1,m_2)$ of the 
Markoff equation $m^2+m_1^2+m_2^2=3mm_1m_2$, up to permutation.
Theorem~3.6 of \cite{R2011} provides an explicit bijection 
$\um \mapsto \xi_\um$ from the set of all Markoff triples $\um$ with 
$\um\neq (1,1,1)$ to the set $\cE_3^+ \cup \left] 1/2, 1 \right[$.  It can be 
shown that the number $\xi$ of Proposition~\ref{ex:prop} is $\xi_\um$
for $\um=(2,1,1)$.  Thus, by \cite[Corollary~5.10]{R2011}, its Lagrange constant
is $1/3$, the largest possible value for an irrational non-quadratic real number.

%
%

\section{Preliminary estimates}
\label{sec:prelim}

It follows from \cite[Theorem~2.1]{R2008} that any Fibonacci 
sequence $(W_i)_{i\ge 1}$ in $\GL_2(\bZ)$ associated to an extremal number
$\xi$ of $\GL_2(\bZ)$-type satisfies $\norm{W_{i+1}}\asymp \norm{W_i}^\gamma$.  
The goal of this section is to prove the following result which provides a 
sharper estimate as well as additional properties of this sequence and 
of the corresponding sequence of symmetric matrices $(\ux_i)_{i\ge 1}$.

\begin{proposition}
\label{prelim:prop}
Let $\xi$ be an extremal real number of $\GL_2(\bZ)$-type.  Fix an
unbounded Fibonacci sequence $(W_i)_{i\ge 1}$ in $\GL_2(\bZ)$, a matrix 
$M\in\GL_2(\bQ)$ and a sequence of symmetric matrices $(\ux_i)_{i\ge 1}$ 
in $\GL_2(\bQ)$ satisfying conditions {\rm (E1)--(E3)} from 
section~\ref{main:ssec:extremal}.  Set
\begin{equation}
\label{prelim:prop:eq1a}
   t_i=\trace(W_i) \et  d_i=\det(W_i)
\end{equation}
for each integer $i\ge 1$.  Set also  
\begin{equation}
\label{prelim:prop:eq1b}
  J=\matrice{0}{1}{-1}{0}, \quad
 \Xi=\matrice{1}{\xi}{\xi}{\xi^2} \et
 \theta_0=\trace(\Xi M).
\end{equation}
Then we have $\theta_0\neq 0$ and, for each $i\ge 1$, 
\begin{itemize}
  \item[(i)] $\ux_{i+2}=W_{i+1}\ux_i=W_i\ux_{i+1}$,
  \medskip
  \item[(ii)] $\ux_{i+3} = t_{i+1}\ux_{i+2}-d_{i+1}\ux_{i}$,
  \medskip
  \item[(iii)] $t_{i+3}=t_{i+1}t_{i+2}-d_{i+1}t_{i}$,
  \medskip
  \item[(iv)] $\det(\ux_i,\ux_{i+1},\ux_{i+2})=\pm\det(M)^{-2}\trace(MJ)\neq 0$. 
\end{itemize}
Moreover, there exists $\rho>0$ and $i_0\ge 1$ such that, 
for each $i\ge i_0$, we have $t_i\neq 0$,
\begin{equation}
 \label{prelim:prop:eq2}
 \log|t_i|=\rho F_i+\cO(\gamma^{-i}) \et 
 W_i = \theta_0^{-1} t_i \Xi M_i + \cO(|t_i|^{-1}).
\end{equation}
\end{proposition}

The main novelty is the first estimate in \eqref{prelim:prop:eq2}.  In
formula (iv), we identify each 
\[
 \ux_k=\matrice{x_{k,0}}{x_{k,1}}{x_{k,1}}{x_{k,2}} \in\GL_2(\bQ)
\]
with the triple $\ux_k=(x_{k,0},x_{k,1},x_{k,2})\in\bQ^3$.  Then,
$\det(\ux_i,\ux_{i+1},\ux_{i+2})$ represents the determinant of the 
$3\times 3$ matrix whose rows are the triples $\ux_i$, $\ux_{i+1}$ and 
$\ux_{i+2}$.

\begin{proof}[Proof of Proposition~\ref{prelim:prop}]
For each integer $i\ge 1$, we find
\[
 \ux_{i+2}=W_{i+2}M_{i}^{-1}=W_{i+1}\ux_i=\ux_{i+1}M_{i+1}\ux_i.
\]
Taking the transpose, we deduce that 
$\ux_{i+2}=\ux_iM_i\ux_{i+1}=W_i\ux_{i+1}$.
This proves (i) and yields 
\[
 \ux_{i+3}=W_{i+1}\ux_{i+2}=W_{i+1}^2\ux_i \quad (i\ge 1).
\]
Moreover, the Cayley-Hamilton theorem gives 
$W_{i+1}^2=t_{i+1}W_{i+1}-d_{i+1}I$, where $I$ is the identity 
of $\GL_2(\bZ)$. Substituting this into the formula for
$\ux_{i+3}$, we obtain (ii).

We also note that, for each $k\ge 1$, the matrices $\ux_k M$,
and $\ux_k \trM={}^t(M\ux_k)$ have the same trace 
$t_k$.  Thus multiplying both sides of (ii) on the right by $M$ and 
taking traces yields (iii).

By Formula (2.1) in \cite{R2004}, we have
\[
 \det(\ux_i, \ux_{i+1}, \ux_{i+2}) = \trace(J\ux_i J\ux_{i+2} J\ux_{i+1}).
\]
Using $\ux_{i+2}=\ux_i M_i\ux_{i+1}$ and $\ux_k J \ux_k=\det(\ux_k)J
=\pm\det(M)^{-1}J$ for each $k\ge 1$, we deduce that
\[
 \det(\ux_i, \ux_{i+1}, \ux_{i+2}) 
 = \trace(J\ux_i J\ux_i M_i\ux_{i+1} J\ux_{i+1})
 =\pm\det(M)^{-2}\trace(M_i J),
\]
which implies (iv) upon noting that $\trace(\trM J) = -\trace(M J) \neq 0$ 
since $\trM\neq M$.

Let $c\ge 1$ be a constant for which (E2) and (E3) hold.  We have
\begin{equation}
 \label{prelim:prop:eq4}
 c^{-1}\norm{W_{i+2}} \ge (c^{-1}\norm{W_{i+1}})(c^{-1}\norm{W_i}) 
 \quad (i\ge 1).
\end{equation}
Since $(W_i)_{i\ge 1}$ is unbounded, there exists an index $k\ge 2$ such that
$\norm{W_k}\ge ec^2$.  As we have $\norm{W_{k-1}}\ge 1$, applying  
\eqref{prelim:prop:eq4} with $i=k-1$ and $i=k$ yields  
$c^{-1}\norm{W_{k+2}}\ge c^{-1}\norm{W_{k+1}}\ge e$.
Then, by induction on $i$, we obtain 
\begin{equation}
 \label{prelim:prop:eq5}
 c^{-1}\norm{W_{k+i+1}} \ge c^{-1}\norm{W_{k+i}}\ge \exp(F_{i-1})  \quad (i\ge 1).
\end{equation}

Since $\norm{(\xi,-1)W_i}\le c \norm{W_i}^{-1}$, we have 
$\norm{(\xi,-1)\ux_i}\ll \norm{W_i}^{-1}$ and so,
\begin{equation}
 \label{prelim:prop:eq6}
 \ux_i= x_{i,0}\Xi + \cO(\norm{W_i}^{-1}).
\end{equation}
As the entries of $\ux_i$ are rational numbers with a common denominator
$d\ge 1$ independent of $i$, this implies that $x_{i,0}\neq 0$ for each large 
enough $i$ and that 
\begin{equation}
 \label{prelim:prop:eq7}
 t_i = \trace(\ux_i M) = x_{i,0}\theta_0 + \cO(\norm{W_i}^{-1}).
\end{equation}
If $\theta_0=0$, this implies that $\trace(\ux_i M)=0$ 
for each large enough $i$.  However, it follows from (iv) that, for each $i\ge 1$,
the matrices $\ux_i$, $\ux_{i+1}$ and $\ux_{i+2}$ span the vector 
space of $2\times 2$ symmetric matrices $\ux$ 
with coefficients in $\bQ$.   Thus, $\trace(\ux M)=0$ for all those 
$\ux$ and so $\trM=-M$, against the hypothesis.  
We conclude that $\theta_0\neq 0$. 

Since $\theta_0\neq 0$, we deduce from  \eqref{prelim:prop:eq5}, 
\eqref{prelim:prop:eq6} and \eqref{prelim:prop:eq7} that there exists an 
integer $i_0>k$ such that, for each $i\ge i_0$, both $t_i$ and $x_{i,0}$ 
are non-zero with 
\begin{equation}
 \label{prelim:prop:eq8}
 \norm{W_i} \asymp \norm{\ux_i} \asymp |x_{i,0}| \asymp |t_i|,
\end{equation}
and the second estimate in \eqref{prelim:prop:eq2} follows.

To prove the first estimate in \eqref{prelim:prop:eq2}, we note that, 
for $i\ge i_0$, we have
\[
 |t_{i+2}| \gg |t_{i+1}| \gg |t_i| \gg \exp(F_{i-1-k}) \gg \exp(i)
\]
by \eqref{prelim:prop:eq5} and \eqref{prelim:prop:eq8}. Then, the 
recurrence relation (iii) yields
\[
 \Big| \frac{t_{i+3}}{t_{i+1}t_{i+2}}-1\Big| 
  = \Big| \frac{t_i}{t_{i+1}t_{i+2}}\Big|
  \ll \frac{1}{|t_{i+2}|} \ll \exp(-i),
\]
from which we deduce that
\begin{equation}
\label{prelim:prop:eq9}
 \left| \log \Big|\frac{t_{i+3}}{t_{i+1}t_{i+2}}\Big| \right| 
  \ll \exp(-i) \le \gamma^{-2i}.
\end{equation}
For each $i\ge i_0$, we set
\[
 \rho_i=\gamma^{-i}\log|t_i| \et \delta_{i}=\rho_{i+2}-\rho_{i+1}.
\]
With this  notation,  we find that
\[
  \gamma^{-i-3} \log \Big|\frac{t_{i+3}}{t_{i+1}t_{i+2}}\Big|
   =\rho_{i+3} - \gamma^{-1}\rho_{i+2} - \gamma^{-2}\rho_{i+1}
   =\delta_{i+1} + \gamma^{-2}\delta_{i}.
\]
Thus, for $i\ge i_0$, \eqref{prelim:prop:eq9} translates into
$\big| \delta_{i+1} + \gamma^{-2}\delta_{i} \big|  \le c_1 \gamma^{-3i-3}$
for a constant $c_1>0$ which is independent of $i$, and so
\[
 \gamma^{2(i+1)} |\delta_{i+1}| 
  \le \gamma^{2i}|\delta_{i}| + c_1 \gamma^{-i-1} \quad (i\ge i_0).
\]
Thus, there is a constant $c_2>0$ such that $\gamma^{2i} |\delta_i| \le c_2$
for each $i\ge i_0$.  Hence, $(\rho_i)_{i\ge i_0}$ is a Cauchy sequence in $\bR$ 
with
\[
 |\rho_{i}-\rho_{i+1}| = |\delta_{i-1}| \le c_2\gamma^{-2i+2} 
 \quad (i\ge i_0+1).
\]
So, it converges to a real number $\rho_\infty$ with
$|\rho_i-\rho_\infty| \ll \gamma^{-2i}$ for $i\ge i_0+1$, and then
\[
 \log |t_i| = \rho_\infty\gamma^i + \cO(\gamma^{-i}) \quad (i\ge i_0+1).
\]
By Binet's formula \eqref{main:Binet}, we also have 
$\gamma^i=\sqrt{5}\gamma^{-1}F_i+\cO(\gamma^{-i})$ for $i\ge 0$.  
Substituting this into the previous estimate yields the first part of
\eqref{prelim:prop:eq2} with $\rho=\sqrt{5}\gamma^{-1}\rho_\infty$.
\end{proof}

%
%

\section{The first two coordinates of the approximation points}
\label{sec:norm}

In this section, we provide estimates for $|x_0(v)|$ and $|x_0(v)\xi-x_1(v)|$ 
for the points $\ux(v)$ with $v\in\prefixesp$ attached to a number $\xi\in\cE_m$ 
with $m\ge 1$, as defined in section~\ref{main:ssec:points}.  We first establish 
a general result which applies to any extremal number of $\GL_2(\bZ)$-type.

\begin{proposition}
\label{norm:prop}
Let the notation be as in Proposition~\ref{prelim:prop}, let $E=\{a,b\}$ be an 
alphabet of two letters, and let $\varphi\colon E^*\to\GL_2(\bZ)$ be the 
morphism of monoids such that $\varphi(a)=W_1$ and 
$\varphi(b)=W_0:=W_1^{-1}W_2$.  For each $v\in\prefixesp$, we have 
\begin{itemize}
 \item[(i)] $\norm{\varphi(v)} \asymp \exp(\rho|v|)$,
 \medskip
 \item[(ii)] $\norm{(\xi,-1)\varphi(v)} \asymp \norm{\varphi(v)}^{-1}$.
\end{itemize}
\end{proposition}

\begin{proof}
Part (ii) follows from \cite[Theorem~2.3]{R2008} since $|\det\varphi(v)|=1$
for each $v$.  It can also be deduced, with additional work, from the proof
of \cite[Proposition~4.3]{R2008}.  An independent argument is provided by 
Proposition~\ref{ex:prop} for the number $\xi$ given by 
\eqref{main:extremal:eq:example}.

To prove part (i), we proceed as in the proof of \cite[Lemma~5.2]{R2008}.
Using \eqref{prelim:prop:eq2} in Proposition~\ref{prelim:prop}, we first note 
that, for each $i\ge 1$, we have
\begin{equation}
\label{norm:prop:eq1}
 \exp(-\rho|w_i|)\varphi(w_i) = \exp(-\rho F_i)W_i = A_i +R_i
\end{equation}
where $\norm{R_i} \le c_1\gamma^{-i}$ for a constant $c_1\ge 1$, and where
$A_i=\pm\theta_0^{-1}\Xi M_i$ belongs to 
\[
 \cA =\big\{ \pm I, \pm\theta_0^{-1}\Xi M, \pm\theta_0^{-1}\Xi \trM\big\}.
\]
Since $\Xi M\Xi = \Xi\trM\Xi=\theta_0\Xi$, the set $\cA$ is stable under 
multiplication.  Choose $c_2\ge 1$ such that $c_2^{-1}\le \norm{A}\le c_2$
for each $A\in\cA$, and choose an integer $\ell$ such that
$\gamma^{\ell-1}\ge 16c_1c_2^3$.

Any $v\in\left]w_\ell,w_\infty\right[$ can be written as a product
\[
 v = w_{i_1} \cdots w_{i_s} u
\]
for a decreasing sequence of integers $i_1 > \cdots > i_s$ with $i_s>\ell$, 
and some $u\in\left[\epsilon,w_\ell\right]$. Set $w=w_{i_1}\cdots w_{i_s}$.  
Then, using \eqref{norm:prop:eq1}, we obtain
\[
 \exp(-\rho|w|)\varphi(w)=(A_{i_1}+R_{i_1})\cdots(A_{i_s}+R_{i_s}) =  A+R
\]
where $A=A_{i_1}\cdots A_{i_s}\in\cA$ and where $R$ is a sum, indexed 
by the non-empty subsequences $(j_1,\dots,j_t)$ of $(i_1,\dots,i_s)$, of products 
of the form $B_1R_{j_1}\cdots B_tR_{j_t}B_{t+1}$ with $B_1,\dots,B_{t+1}\in\cA$.
As the norm of such a product is at most $c_3^t\norm{R_{j_1}}\cdots\norm{R_{j_t}}$
with $c_3=4c_2^2$, we find 
\[
 \norm{R} \le \prod_{k=1}^s (1+c_3\norm{R_{i_k}})-1
   \le \exp\Big( c_3\sum_{i=\ell+1}^{\infty} \norm{R_i}\Big) -1
   \le \exp\big(c_1c_3\gamma^{-\ell+1}\big) -1.
\]
Since $c_1c_3\gamma^{-\ell+1}\le (4c_2)^{-1} \le 1/2$, this gives $\norm{R}\le (2c_2)^{-1} \le \norm{A}/2$, thus
\[
 \norm{\exp(-\rho|v|)\varphi(v)}
 \asymp \norm{\exp(-\rho|w|)\varphi(w)} = \norm{A+R} \asymp \norm{A} \asymp 1,
\]
and so $\norm{\varphi(v)} \asymp \exp(\rho|v|)$.  This last estimate also holds if
$v\in \left]\epsilon, w_\ell\right]$.
\end{proof}

\begin{corollary}
\label{norm:cor}
Let $\xi\in\cE_m$ for some integer $m\ge 1$.  There exists $\rho>0$ such that, 
for each $v\in\left]\epsilon,w_\infty\right[$, the point $\ux(v)$ defined in section~\ref{main:ssec:points} satisfies
\[
  \max\{1,|x_0(v)|\} \asymp \exp(\rho |v|) 
  \et
  |x_0(v)\xi-x_1(v)| \asymp \exp(-\rho|v|).
\]
\end{corollary}

\begin{proof}
Let $v\in\left]\epsilon,w_\infty\right[$.  Since $\ux(v)$ has the same first 
column as $\varphi(v)M(v)^{-1} \in \GL_2(\bZ)$, we have
$(x_0(v),x_1(v))\neq (0,0)$ and the proposition yields
\begin{align*}
 &\max\{1,|x_0(v)|\}\le \norm{\varphi(v)M(v)^{-1}} 
                \asymp\exp(\rho|v|), \\
 &|x_0(v)\xi-x_1(v)|\le \norm{(\xi,-1)\varphi(v)M(v)^{-1}}
                \asymp\exp(-\rho|v|).
\end{align*}
As $\xi$ is badly approximable, we also have 
$\max\{1,|x_0(v)|\} |x_0(v)\xi-x_1(v)|\gg 1$,
and the conclusion follows.
\end{proof}

We take this opportunity to fill a small gap in the proof of 
\cite[Proposition~4.3]{R2008}, itself a crucial step 
towards \cite[Theorem~2.2]{R2008}.  The argument there involves an 
unspecified real number $c_i$ whose absolute value is tacitly 
assumed to be bounded away from $0$ when $i$ is large enough.
To show that this is indeed the case, one notes that, in the notation 
of the proof, we have $c_i=c^2y_{i,2}$ where $c\neq 0$ is independent
of $i$ and defined by the condition $(1/\xi,1){}^tU_k^{-1}=c(r,1)$.  
Since $y_{i,2}$ is a non-zero integer for each large enough $i$,
we conclude that $|c_i|\ge c^2$ for all those $i$.

%
%

\section{Recurrence relations}
\label{sec:rec}

Let the notation and hypotheses be as in section~\ref{main:ssec:points}.  
In particular, $\xi$ is a fixed number in $\cE_m$ for some integer $m\ge 1$,
and both Proposition~\ref{prelim:prop} and Corollary~\ref{norm:cor} apply.
In this section, we complement Corollary~\ref{norm:cor} by estimating 
$\norm{(\xi,-1)\ux(v)}$ from above for each $v\in\cV_\ell$ with $\ell\ge 4$ 
large enough.  We also prove Theorem~ \ref{main:points:thm}\,(ii). 

We say that two symmetric matrices $A$ and $B$ in $\Mat_{2\times 2}(\bZ)$
are \emph{equivalent}, and we write $A\equiv B$, if they have the same first
column.  With this notation, we recall that, for each non-empty word
$v\in E^*$, the matrix $\ux(v)$ is the unique 
symmetric matrix in $\Mat_{2\times 2}(\bZ)$ which satisfies 
\begin{equation}
\label{rec:eq1}
 \ux(v)\equiv \varphi(v)M(v)^{-1}
 \et
 |x_1(v)\xi-x_2(v)| <1/2.
\end{equation}
To estimate $\norm{(\xi,-1)\ux(v)}$, we will first construct recursively, for 
 each $v\in\cV_4$, a symmetric matrix $\tux(v)$ in $\Mat_{2\times 2}(\bZ)$ 
 with $\tux(v)\equiv\ux(v)$.  Then, we will  estimate $\norm{(\xi,-1)\tux(v)}$ 
 from above and show that, for each  $v\in\cV_\ell$ with $\ell\ge 4$ 
 large enough, this norm is less than $1/2$, thus $\tux(v)=\ux(v)$.   
 As a consequence, for each triple of consecutive elements 
 $u<v<w$ of $\cV_\ell$ with $\ell$ large enough, we will derive an 
 explicit relation of linear dependence between $\ux(u)$, $\ux(v)$ and 
 $\ux(w)$ when $v\in\cV_{\ell+1}$, while we will show that the determinant 
 of these three points is $\pm 2$ when $v\notin\cV_{\ell+1}$, thereby 
 proving Theorem~ \ref{main:points:thm}\,(ii).
 
 We start with three simple lemmas.

\begin{lemma}
\label{rec:lemma:ab}
We have $\ux(vw_i)=\ux(v\tw_i)$ for each $v \in E^*$ and each $i\ge 2$.
\end{lemma}

\begin{proof}
Let $u\in E^*$, and let $U=\varphi(u)$.  By definition, we have
\[
 \ux(uab)\equiv UW_1W_0M^{-1}=UW_2M_2^{-1}=U\ux_2.
\]
Proposition~\ref{prelim:prop}\,(i) also gives $W_1\ux_2=W_2\ux_1$,
thus $\ux_2=W_0\ux_1$ and so we find that
\[
 \ux(uba)\equiv  UW_0W_1\trM^{-1}=UW_0\ux_1=U\ux_2.
\]
This shows that $\ux(uab)\equiv\ux(uba)$ and so $\ux(uab)=\ux(uba)$.
Applying this to $u=vw_i^{**}$ for an integer $i\ge 2$ and a word
$v\in E^*$, this yields $\ux(vw_i)=\ux(v\tw_i)$.
\end{proof}

\begin{lemma}
\label{rec:lemma:trace}
For each integer $i\ge 2$, the matrices $\varphi(\tw_i)$ and $\varphi(w_i)$ have 
the same characteristic polynomial.
\end{lemma}

\begin{proof}
For $i\ge 3$, Lemma~\ref{V:lemma2} gives $w_i=w_{i-2}\tw_{i-1}$ which
implies that $\tw_i=w_{i-2}w_{i-1}$.  Since we also have $w_i=w_{i-1}w_{i-2}$,
we deduce that 
\[
 \varphi(\tw_i) = W_{i-2}W_{i-1}  = W_{i-2}\varphi(w_i)W_{i-2}^{-1}.
\]
The last formulas remain true for $i=2$.  Thus, $\varphi(\tw_i)$ and 
$\varphi(w_i)$ are conjugate matrices, and so their characteristic 
polynomials are the same.
\end{proof}

The third lemma below relies on the specific form of the matrix $M$, given by 
\eqref{main:extremal:eq:M}.

\begin{lemma}
\label{rec:lemma:uv}
For any non-empty words $u,v\in E^*$ we have 
\[
\varphi(u)M(v)^{-1}
  \equiv 
  \begin{cases}
   \ux(u) &\text{if $M(u)=M(v)$,}\\
   -\ux(u) &\text{otherwise.}
  \end{cases}
\]
\end{lemma}

\begin{proof}
This follows from the definition if $M(u)=M(v)$.  Otherwise, we have
$M(v)={}^tM(u)$.  Since 
\[
 M^{-1} = \matrice{0}{-1}{1}{m} \equiv \matrice{0}{-1}{1}{-m}=-{}^tM^{-1},
\]
we deduce that $M(v)^{-1} \equiv -M(u)^{-1}$, thus $\varphi(u)M(v)^{-1}\equiv
-\varphi(u)M(u)^{-1} \equiv -\ux(u)$.
\end{proof}

The next result is the key to our analysis.

\begin{proposition}
\label{rec:prop1}
Let $\ell\ge 4$ be an integer and let $u<v<w$ be consecutive words 
in $\cV_\ell$ with $v\in\cV_{\ell+1}$.  Then, we have $|w|-|v|=|v|-|u|=F_i$
for some $i\in\{\ell-2,\ell-1\}$, and 
\begin{equation}
\label{rec:prop1:eq}
 \ux(w) \equiv
  \begin{cases}
  t_i\ux(v)-d_i\ux(u) &\text{if $M(u)=M(v)$,}\\
  t_i\ux(v)+d_i\ux(u) &\text{otherwise,}
  \end{cases}
\end{equation}
where $t_i=\trace(W_i)$ and $d_i=\det(W_i)$.
\end{proposition}

\begin{proof}
The fact that $|w|-|v|=|v|-|u|=F_i$ for some $i\in\{\ell-2,\ell-1\}$ follows 
from Corollary~\ref{Vbar:prop2:cor1} and Proposition~\ref{Vbar:prop3}.
According to Proposition~\ref{V:prop1}, this implies that 
$v=us$ and $w=vs'$ for some $s, s'\in\{w_i,\tw_i\}$.  
By Lemma~\ref{rec:lemma:ab}, we have $\ux(vs)=\ux(vs')$, thus
\[
 \ux(w)=\ux(vs)=\ux(us^2).
\]
As $M(us^2)=M(us)=M(v)$, we deduce that
\[
 \ux(w)\equiv \varphi(us^2)M(v)^{-1} = \varphi(u) \varphi(s)^2 M(v)^{-1}.
\]
Since $s\in\{w_i,\tw_i\}$ and $i\ge 2$, Lemma~\ref{rec:lemma:trace} shows that
$\varphi(s)$ has the same characteristic polynomial as $\varphi(w_i)=W_i$,
and so the Cayley-Hamilton theorem gives
\[
 \varphi(s)^2 =t_i\varphi(s)-d_i I
\]
where $I$ denotes the $2\times 2$ identity matrix.  Altogether, this yields
\[
 \ux(w)\equiv \varphi(u)\big( t_i \varphi(s) - d_i I \big)M(v)^{-1}
      = t_i\varphi(v) M(v)^{-1} -d_i\varphi(u)M(v)^{-1},
\]
and \eqref{rec:prop1:eq} follows using Lemma~\ref{rec:lemma:uv}.
\end{proof}

We will prove below that the congruence \eqref{rec:prop1:eq} is 
in fact an equality when $\ell$ is large enough.  To show this, we 
first construct ``algebraic'' points $\tux(w)$ for each $w\in\cV_4$.

\begin{corollary}
\label{rec:prop1:cor}
The following recurrence process constructs,  for each $w\in\cV_4$, 
a symmetric matrix $\tux(w)$ in $\Mat_{2\times 2}(\bZ)$ with 
$\tux(w)\equiv\ux(w)$.

{\rm (i)} If $w\in\cF$, then $w=w_i$ for some $i\ge 4$, and we set
$ \tux(w)=\ux_i$.

{\rm (ii)}  If $w\notin\cF$, then $\alpha(w)=w_\ell$ for some integer $\ell\ge 4$,
and $w$ is at least the third element of $\cV_\ell$.  Thus, we can find 
$u,v\in\cV_\ell$ such that $u<v<w$ are consecutive elements of $\cV_\ell$.  
Since $w\notin\cV_{\ell+1}$, Corollary~\ref{Vbar:prop2:cor2} implies 
that $v\in\cV_{\ell+1}$ and so, by Proposition~\ref{rec:prop1}, we have
$|w|-|v|=|v|-|u|=F_i$ for some $i\in\{\ell-2,\ell-1\}$.  Then, we define
\begin{equation}
\label{rec:prop1:cor:eq}
 \tux(w)
  =\begin{cases}
  t_i\tux(v)-d_i\tux(u) &\text{if $M(u)=M(v)$,}\\
  t_i\tux(v)+d_i\tux(u) &\text{otherwise.}
  \end{cases}
\end{equation}
\end{corollary}

\begin{proof}
In case (i), this is because $\ux_i=W_iM_i^{-1} = \varphi(w_i)M(w_i)^{-1} 
\in \GL_2(\bZ)$ is symmetric and so $\ux_i\equiv \ux(w_i)$, for each 
$i\ge 1$.  In case (ii), we may assume, by induction on the length, that
$\tux(u)\equiv\ux(u)$ and $\tux(v)\equiv\ux(v)$ are symmetric $2\times2$ 
integral matrices.  As $\tux(w)$ is an integral linear combination of these, it is 
also a symmetric $2\times2$ integral matrix,  and we have $\tux(w)\equiv\ux(w)$
by Proposition~\ref{rec:prop1}.
\end{proof}

Thus, by construction, \eqref{rec:prop1:cor:eq} holds for each triple 
of consecutive elements $u<v<w$ of $\cV_\ell$ with $\ell\ge 4$ 
and $w\notin\cV_{\ell+1}$.   This contrasts with the congruence
\eqref{rec:prop1:eq} which holds for the larger set of triples 
of consecutive elements $u<v<w$ of $\cV_\ell$ with $\ell\ge 4$ 
and $v\in\cV_{\ell+1}$.   

By definition, we have $\tux(w_\ell)=\ux_\ell$ for  each $\ell\ge 4$.
Below, we compute $\tux(w)$ for the next simplest families 
of prefixes $w$ of $w_\infty$.
 
\begin{lemma}
\label{rec:lemma:yz}
For each integer $\ell\ge 4$, we have
\begin{align*}
 &\tux(w_{\ell+1}w_{\ell-1}) = \uy_\ell := t_{\ell-1}\ux_{\ell+1}+d_{\ell-1}\ux_\ell,\\ 
 &\tux(w_{\ell+2}w_{\ell-2}) = \uz_\ell := t_{\ell-2}\ux_{\ell+2}+d_{\ell-2}\uy_\ell.
\end{align*}
\end{lemma}

\begin{proof}
Fix a choice of $\ell\ge 4$.  By Lemma~\ref{V:lemma1}, the words 
$w_\ell<w_{\ell+1}<w_{\ell+1}w_{\ell-1}$ are the first three elements 
of $\cV_\ell$.  Since $w_{\ell+1}w_{\ell-1} \notin \cV_{\ell+1}$, they satisfy the hypotheses of Corollary~\ref{rec:prop1:cor}~(ii)
with $i=\ell-1$.  As $M(w_\ell)=M_\ell\neq M_{\ell+1}=M(w_{\ell+1})$, 
we deduce that
\[
\tux(w_{\ell+1}w_{\ell-1})
  =t_{\ell-1}\tux(w_{\ell+1})+d_{\ell-1}\tux(w_{\ell})
  =t_{\ell-1}\ux_{\ell+1}+d_{\ell-1}\ux_\ell
\]
which is denoted $\uy_\ell$.   Lemma~\ref{V:lemma1} also shows that 
$w_{\ell+1}w_{\ell-1}<w_{\ell+2}<w_{\ell+2}w_{\ell-2}$ are consecutive elements 
of $\cV_\ell$.  Since $w_{\ell+2}w_{\ell-2}\notin\cV_{\ell+1}$, they satisfy 
the hypothesis of Corollary~\ref{rec:prop1:cor}~(ii) with $i=\ell-2$.  
As $M(w_{\ell+1}w_{\ell-1})=M_{\ell-1}\neq M_{\ell+2}=M(w_{\ell+2})$, 
we deduce that
\[
\tux(w_{\ell+2}w_{\ell-2})
  =t_{\ell-2}\tux(w_{\ell+2})+d_{\ell-2}\tux(w_{\ell+1}w_{\ell-1})
  =t_{\ell-2}\ux_{\ell+2}+d_{\ell-2}\uy_\ell
\]
which is denoted $\uz_\ell$.
\end{proof}

We can now proceed to our main estimate.  

\begin{proposition}
\label{rec:prop2}
Let $\rho$ be as in Proposition~\ref{prelim:prop}.  Then, there is 
a constant $c>0$ such that, for any $v\in\cV_4$, we have
 \[
  \delta(v):=\norm{(\xi,-1)\tux(v)}\exp(\rho|\alpha(v)|) \le c.
\]
\end{proposition}

\begin{proof}
For any pair of integers $k$, $\ell$ with $4\le \ell\le k$, we set 
\[
 d_k(\ell)=\max\{ \delta(v)\,;\, v\in\cV_\ell\ \text{and}\ v\le w_k\}.
\]
We need to show that the non-decreasing sequence $\big(d_k(4)\big)_{k\ge 4}$ 
is bounded from above.   We proceed in two steps.

\textbf{Step 1.} By estimates \eqref{prelim:prop:eq2} in Proposition~\ref{prelim:prop}, there exist $c_1, c_2>0$ such that
\begin{align*}
 &|t_\ell| \le (1+c_1\gamma^{-\ell})\exp(\rho F_\ell), \\
 &\delta(w_\ell) = \norm{(\xi,-1)\ux_\ell} \exp(\rho F_\ell) \le c_2,
\end{align*}
for each integer $\ell\ge 1$, because 
$\tux(w_\ell) = W_\ell M_\ell^{-1} =\ux_\ell$ and $(\xi,-1)\Xi=0$.  
Then, for any $\ell\ge 4$, Lemma~\ref{rec:lemma:yz} yields
\begin{align*}
  \delta(w_{\ell+1}w_{\ell-1}) 
    &= \norm{(\xi,-1)\uy_\ell} \exp(\rho F_\ell) \\
    &\le (1+c_1\gamma^{-\ell+1})\delta(w_{\ell+1})+\delta(w_\ell)
       \le (2+c_1)c_2,\\
  \delta(w_{\ell+2}w_{\ell-2}) 
    &= \norm{(\xi,-1)\uz_\ell} \exp(\rho F_\ell) \\
    &\le (1+c_1\gamma^{-\ell+2})\delta(w_{\ell+2})+\delta(w_{\ell+1}w_{\ell-1}) 
       \le (3+2c_1)c_2.
\end{align*}  
By Lemma~\ref{V:lemma1}, this implies that, for any $\ell\ge 4$, we have    
\begin{equation}
\label{rec:prop2:eq1} 
  \delta(v) \le c_3 \quad 
  \text{for each}\ v\in\cV_\ell\cap\left[\epsilon,w_{\ell+3}\right]
\end{equation}
where $c_3=(3+2c_1)c_2$, and so $d_{\ell+3}(\ell)\le c_3$.

\textbf{Step 2.} We claim that there is a constant $c_4>0$ such that 
\begin{equation}
\label{rec:prop2:eq2} 
  \delta(v) \le (1+c_4\gamma^{-\ell})d_k(\ell+1)
\end{equation}  
for each $v \in \cV_\ell \cap \left]w_{\ell+3} , w_k\right]$ with $4 \le \ell\le k-4$.

If we take this for granted, then, in view of \eqref{rec:prop2:eq1}, we obtain
\[
 \max\{c_3, d_k(\ell)\} \le (1+c_4\gamma^{-\ell})\max\{c_3, d_k(\ell+1)\}
\]
for each $k\ge 8$ and each $\ell=4,\dots,k-4$.  As $d_k(k-3)\le c_3$, this gives,
as needed
\[
d_k(4) \le c:=c_3\prod_{i=4}^\infty (1+c_4\gamma^{-\ell}) < \infty.
\]

To prove the claim, we may assume that $v\notin\cV_{\ell+1}$ because 
otherwise $\delta(v)\le d_k(\ell+1)$ and \eqref{rec:prop2:eq2} is automatic.
Then, we have $v\in\cV_\ell\setminus\cV_{\ell+1}$ and, since $v>w_{\ell+3}$, 
there is a maximal sequence of consecutive elements of $\cV_\ell$
\[
 v_1<u_1<\cdots<v_s<u_s
\]
of even cardinality $2s$ with $s\ge 2$, such that
\[
 v=v_s \et \{v_2,\dots,v_s\}\subseteq \cV_\ell\setminus\cV_{\ell+1}.
\]
As $w_{\ell+2}w_\ell < w_{\ell+3}$ are consecutive elements of $\cV_\ell$
contained in $\cV_{\ell+1}$, we must have $v_1\ge w_{\ell+2}w_\ell$ and 
so $v_1\in\cV_{\ell+1}$ by maximality of the sequence.  Moreover, for each 
$i=2,\dots,s$, the words $u_{i-1}<v_i<u_i$ are consecutive in $\cV_\ell$ with
$v_i\notin\cV_{\ell+1}$.  By Proposition~\ref{Vbar:prop4}, this implies that 
$u_{i-1}<u_i$ are consecutive elements of $\cV_{\ell+1}$ with 
$|u_i|-|u_{i-1}|=F_\ell$.  Thus, $|u_1|<|u_2|<\cdots<|u_s|$ is an arithmetic 
progression made of consecutive elements of $\cVbar_{\ell+1}$.  
By Corollary~\ref{Vbar:prop2:cor1}, this implies that $s\le 5$.  If $s\ge 3$,
then, by Proposition~\ref{Vbar:prop3}, we further have $u_i\in\cV_{\ell+2}$
for $i=2,\dots,s-1$.

For each $i=1,\dots,s-1$, the words $v_i<u_i<v_{i+1}$ are 
consecutive elements of $\cV_\ell$ with $v_{i+1}\notin\cV_{\ell+1}$.
So, Corollary~\ref{rec:prop1:cor}~(ii) yields
\[
 \tux(v_{i+1}) = t_k\tux(u_i) \pm \tux(v_i)
\]
for some $k=k_i\in\{\ell-2,\ell-1\}$.   For $i=1$, we have
$\{u_1,v_1\}\subseteq\cV_{\ell+1}$.  Thus, 
\begin{equation}
\label{rec:prop2:eq3}
 \begin{aligned}
 \delta(v_2) 
   &\le \left(|t_{\ell-1}|\delta(u_1)+\delta(v_1)\right)
            \exp(\rho F_\ell - \rho F_{\ell+1}) \\
  &\le (1+c_1\gamma^{-\ell+1})\delta(u_1) + \exp(-\rho F_{\ell-1})\delta(v_1)\\
  &\le \left(1+c_1\gamma^{-\ell+1} + \exp(-\rho F_{\ell-1})\right)d_k(\ell+1).
 \end{aligned}
\end{equation}
When $2\le i\le s-1$, we have 
$\{v_i,v_{i+1}\}\subseteq\cV_\ell\setminus\cV_{\ell+1}$
and $u_i\in\cV_{\ell+2}$, thus
\[
 \delta(v_{i+1}) 
   \le |t_k| \delta(u_i) \exp(\rho F_\ell - \rho F_{\ell+2}) + \delta(v_i).
\]
Since $|t_k|\le(1+c_1)\exp(\rho F_{\ell-1})$ and $\delta(u_i)\le d_k(\ell+1)$,
this yields
\begin{equation}
\label{rec:prop2:eq4}
  \delta(v_{i+1}) 
   \le \delta(v_i) + (1+c_1)\exp(-\rho F_\ell) d_k(\ell+1) \quad (2\le i \le s-1).
\end{equation}
As $v=v_s$ with $s\le 5$, we conclude from \eqref{rec:prop2:eq3} and 
\eqref{rec:prop2:eq4}, that \eqref{rec:prop2:eq2} holds for a constant
$c_4$ depending only on $c_1$ and $\rho$.
\end{proof}

\begin{corollary}
\label{rec:prop2:cor}
There is an integer $\ell_1\ge 4$ with the following properties.
\begin{itemize}
 \item[(i)] For any integer $\ell\ge \ell_1$ and any $v\in\cV_\ell$, 
 we have 
 \begin{equation}
  \label{rec:prop2:cor:eq1}
  \ux(v)=\tux(v), \quad 
  \norm{(\xi,-1)\ux(v)} \le c\exp(-\rho |\alpha(v)|) \et
  x_0(v)\neq 0.
 \end{equation}
 \item[(ii)] For any integer $\ell\ge \ell_1$ and any triple of 
consecutive words $u<v<w$ in $\cV_\ell$ with $v\in\cV_{\ell+1}$, 
we have $|w|-|v|=|v|-|u|=F_i$ for some $i\in\{\ell-2,\ell-1\}$, and 
\begin{equation}
\label{rec:prop2:cor:eq2}
 \ux(w)
  =\begin{cases}
  t_i\ux(v)-d_i\ux(u) &\text{if $M(u)=M(v)$,}\\
  t_i\ux(v)+d_i\ux(u) &\text{otherwise.}
  \end{cases}
\end{equation}
\end{itemize}
\end{corollary}

Thus, the congruence \eqref{rec:prop1:eq} is 
an equality when $\ell\ge \ell_1$, as claimed earlier.

\begin{proof}[Proof of Corollary~\ref{rec:prop2:cor}]
By estimates \eqref{prelim:prop:eq2} of Proposition~\ref{prelim:prop}, 
there is an integer $\ell_1\ge 4$ such that $|t_i| \le 2\exp(\rho F_i)$ for 
each $i\ge \ell_1$.   We choose this $\ell_1$ so that  
$3c\exp(-\rho F_{\ell_1})<1/2$, where $c\ge 1$ is the constant of  
the preceding proposition.  

(i) For $v\in\cV_\ell$ with $\ell\ge \ell_1$, we have $|\alpha(v)|\ge F_\ell
\ge F_{\ell_1}$ and the preceding proposition yields 
\[
 \norm{(\xi,-1)\tux(v)} \le  c\exp(-\rho|\alpha(v)|) < 1/2. 
\] 
As $\tux(v) \in \Mat_{2\times 2}(\bZ)$ is a symmetric integral matrix, 
\eqref{rec:prop2:cor:eq1} follows.

(ii) Let $\ell\ge \ell_1$ and let $u<v<w$ be consecutive elements of $\cV_\ell$ 
with $v\in\cV_{\ell+1}$.  By Proposition~\ref{rec:prop1}, we have 
$|w|-|v|=|v|-|u|=F_i$ for some $i\in\{\ell-2,\ell-1\}$ and $\ux(w)\equiv \uy$ 
where $\uy$ denotes the right hand side of \eqref{rec:prop2:cor:eq2}.  By part (i)
proved above, \eqref{rec:prop2:cor:eq1} applies to both $u$ and $v$, thus
\begin{align*}
 \norm{(\xi,-1)\uy} 
  &\le c|t_i|\exp(-\rho|\alpha(v)|) + c\exp(-\rho|\alpha(u)|)\\
  &\le  2c\exp(\rho F_{\ell-1}-\rho F_{\ell+1}) + c\exp(-\rho F_\ell)
     = 3c\exp(-\rho F_\ell) < 1/2.
\end{align*} 
Since $\uy$ is symmetric with integer coefficients, we conclude that
$\ux(w)=\uy$. 
\end{proof}

We conclude with the following complement.

\begin{proposition}
\label{rec:prop3}
Let $\ell_1$ be as in Corollary~\ref{rec:prop2:cor}, and let 
$u<v<w$ be consecutive words in $\cV_\ell$ with 
$v\notin\cV_{\ell+1}$, for an integer $\ell\ge \ell_1$.  Then, we 
have 
\[
 \det(\ux(u),\ux(v),\ux(w))=\pm 2.
\]
\end{proposition}

\begin{proof}
We proceed by induction on $|v|$.  Since $v\notin\cV_{\ell+1}$,  
Lemma~\ref{V:lemma1} implies that $v\ge w_{\ell+1}w_{\ell-1}$. 
  
To start, suppose that $v=w_{\ell+1}w_{\ell-1}$.  Then we have 
$u=w_{\ell+1}$ and $w=w_{\ell+2}$.  As $\ell\ge \ell_1$, 
Corollary~\ref{rec:prop2:cor}\,(i) applies to $u$, $v$ and $w$.
Thus, we find 
\[
 \det(\ux(u),\ux(v),\ux(w))
  =\det(\ux_{\ell+1},\uy_\ell,\ux_{\ell+2})
  =\pm \det(\ux_\ell,\ux_{\ell+1},\ux_{\ell+2})
  =\pm 2,
\]
using Lemma~\ref{rec:lemma:yz} and then Proposition~\ref{prelim:prop}(iv).

Assume from now on that $v>w_{\ell+1}w_{\ell-1}$.  
By Lemma~\ref{V:lemma1}, we have $v\ge w_{\ell+2}w_{\ell-2}$ 
and we can extend $u<v<w$ to a sequence of consecutive words 
\[
 u'<v'<u<v<w
\]
in $\cV_\ell$.  Since $v\notin\cV_{\ell+1}$, Proposition~\ref{Vbar:prop4} 
shows that $u<w$ are consecutive elements of $\cV_{\ell+1}$ with 
$|w|-|u|=F_\ell$.   Moreover, since $v'<u<v$ are consecutive elements 
of $\cV_\ell$ with $u\in\cV_{\ell+1}$, Corollary~\ref{rec:prop2:cor}\,(ii) 
gives $\ux(v)=t_i\ux(u)\pm\ux(v')$ for some $i\in\{\ell-2,\ell-1\}$, and 
therefore
\begin{equation}
 \label{rec:prop3:eq1}
 \det(\ux(u),\ux(v),\ux(w)) = \pm \det(\ux(v'),\ux(u),\ux(w)).
\end{equation}

Suppose first that $v'\notin\cV_{\ell+1}$.   Then, arguing as above, we find 
that $u'<u$ are consecutive elements of $\cV_{\ell+1}$ with $|u|-|u'|=F_\ell$.
Thus, $u'<u<w$ are consecutive elements of $\cV_{\ell+1}$ with
$|w|-|u|=|u|-|u'|=F_\ell$.  By Proposition~\ref{Vbar:prop3}, this implies 
that $u\in\cV_{\ell+2}$, and so Corollary~\ref{rec:prop2:cor} gives 
$\ux(w)=t_\ell\ux(u)\pm\ux(u')$.  Substituting this into the right hand side 
of \eqref{rec:prop3:eq1}, we deduce that 
\[
 \det(\ux(u),\ux(v),\ux(w)) = \pm \det(\ux(u'), \ux(v'),\ux(u)).
\]
As $u'<v'<u$ are consecutive elements of $\cV_\ell$ with 
$v'\notin\cV_{\ell+1}$ and $|v'|<|v|$, we may assume by induction 
that the determinant in the right hand side of this equality is $\pm 2$, 
and we are done.  

Finally, suppose that $v'\in\cV_{\ell+1}$.  Then $v'<u<w$ are consecutive 
elements of $\cV_{\ell+1}$.  Since $v'<u$ are also consecutive in $\cV_\ell$, 
Corollary~\ref{Vbar:prop2:cor1} gives $|u|-|v'|=F_{\ell-1}$, thus 
$|u|-|v'|<F_\ell=|w|-|u|$, and so $u\notin\cV_{\ell+2}$ by 
Proposition~\ref{Vbar:prop3}.  As $|u|<|v|$, we may assume by induction 
that the determinant in the right hand side of \eqref{rec:prop3:eq1} 
is $\pm 2$, and we are done once again.  
\end{proof}

Combining this proposition with Corollary~\ref{rec:prop2:cor}\,(ii)
yields the following qualitative statement, and thus proves
Theorem~\ref{main:points:thm}\,(ii).

\begin{corollary}
\label{rec:prop3:cor}
For consecutive words $u<v<w$ in $\cV_\ell$ with $\ell\ge \ell_1$,
the points $\ux(u)$, $\ux(v)$ and $\ux(w)$ are linearly independent
if and only if $v\notin\cV_{\ell+1}$.
\end{corollary}

%
%

\section{Proof of Theorem \ref{main:points:thm}}
\label{sec:final}

Let the notation be as in section~\ref{main:ssec:points}, let $\rho$ 
be as in Proposition~\ref{prelim:prop} for the given $\xi\in\cE_m$,
and let $\ell_1$ be as in Corollary~\ref{rec:prop2:cor}.  
In view of Corollary~\ref{rec:prop3:cor}, 
it remains to prove that there exists an integer $\ell$ with 
$\ell\ge\ell_1$ such that, for each $v\in\cV_{\ell}$ and each
$\uq\in\cA(\epsilon)$, we have 
$L_{\ux(v)}(\uq) = \rho P_v(\rho^{-1}\uq) + \cO_\xi(1)$, where
$\cO_\xi(1)$ stands for a function of $v$ and $\uq$ whose absolute 
value is bounded above by a constant that depends only on $\xi$.

For each $v\in]\epsilon,w_\infty[$, we set
\[
 \Delta(v)
     = (\Delta_0(v),\Delta_1(v),\Delta_2(v))
     = (x_0(v), x_0(v)\xi-x_1(v), x_0(v)\xi^2-x_2(v)),
\]
so that, for any $\uq=(q_1,q_2)\in\bR^2$, we have
\[
 L_{\ux(v)}(\uq) 
   = \max\{ \log|\Delta_0(v)|, q_1+\log|\Delta_1(v)|, q_2+\log|\Delta_2(v)|\}.
\]
By Corollaries~\ref{norm:cor} and \ref{rec:prop2:cor}\,(i), for each 
$v\in\cV_\ell$ with $\ell\ge\ell_1$,  we have
\begin{equation}
\label{final:proof:eq1}
\begin{aligned}
 \log|\Delta_0(v)| &=\rho|v|+\cO_\xi(1),\\ 
 \log|\Delta_1(v)| &=-\rho|v|+\cO_\xi(1),\\
 \log|\Delta_2(v)| 
    &\le \log\norm{(\xi,-1)\ux(v)}+\cO_\xi(1) 
       \le -\rho|\alpha(v)|+\cO_\xi(1).
\end{aligned}
\end{equation}
We claim that we may further choose $\ell$ so that  
\begin{equation}
\label{final:proof:eq3}
 \log|\Delta_2(v)| \ge -\rho|\alpha(v)|+\cO_\xi(1) 
 \quad\text{for all $v\in\cV_\ell\setminus\cF$.}
\end{equation}

If we take this for granted, then, for $v\in\cV_\ell\setminus\cF$
and $\uq=(q_1,q_2)\in\bR^2$, we obtain
\[
  L_{\ux(v)}(\uq) 
   = \max\{ \rho|v|, q_1-\rho|v|, q_2-\rho|\alpha(v)|\} +\cO_\xi(1)
   = \rho P_v(\rho^{-1}\uq) +\cO_\xi(1).
\]
This still holds for $v\in\cF$ and $\uq\in\cA(\epsilon)$, as we then
have $\alpha(v)=v$ and $q_2\le q_1$, thus, 
\[
 q_2+\log|\Delta_2(v)|  
  \le q_1-\rho|v|+\cO_\xi(1)
  = q_1+\log|\Delta_1(v)|+\cO_\xi(1).
\] 

To prove \eqref{final:proof:eq3}, choose $v \in \cV_{\ell_1}\setminus\cF$
and let $\ell\ge \ell_1$ such that $\alpha(v)=w_\ell$.  Then,  $v\neq w_\ell$, 
and so $v$ is the middle term of a triple of consecutive elements 
$u<v<w$ of $\cV_\ell$.  As $v\notin\cV_{\ell+1}$, 
Proposition~\ref{rec:prop3} gives 
\[
 2 = |\det(\ux(u),\ux(v),\ux(w))| = |\det(\Delta(u),\Delta(v),\Delta(w))|.
\]
The determinant on the right is a sum of six products
$\pm\Delta_i(u)\Delta_j(v)\Delta_k(w)$ where $(i,j,k)$ runs through the 
permutations of $(0,1,2)$.  Since the estimates \eqref{final:proof:eq1}
also apply to $u$ and $w$ in place of $v$, 
we find that, for $(i,j,k)\neq (1,2,0)$, these products 
tend to $0$ as $\ell$ go to infinity, uniformly in $v$.  For example, 
we have
\[
 \log |\Delta_2(u)\Delta_1(v)\Delta_0(w)|
  \le \rho(-|\alpha(u)|-|v|+|w|)+\cO_\xi(1)
  \le -\rho F_\ell+\cO_\xi(1)
\]
because $|v|\ge |u|+F_{\ell-2}$ by Corollary~\ref{Vbar:prop2:cor1},
$|\alpha(u)|\ge F_{\ell+1}$ by Corollary~\ref{Vbar:prop2:cor2}, and finally
$|w|-|u|=|\alpha(v)|=F_\ell$ by Proposition~\ref{Vbar:prop3}.  Thus, 
if $\ell$ is large enough, we obtain 
\[
\begin{aligned}
 0\le \log |\Delta_1(u)\Delta_2(v)\Delta_0(w)| 
      &= \rho(|w|-|u|)+\log|\Delta_2(v)| +\cO_\xi(1)\\
      &= \rho|\alpha(v)|+\log|\Delta_2(v)| +\cO_\xi(1)
\end{aligned}
\]
which gives \eqref{final:proof:eq3}.

%
%

\section{Weighted exponents of approximation}
\label{sec:exp}

For each $\uxi=(1,\xi_1,\xi_2)\in\bR^3$ and each $\sigma\in \left[0,\infty\right[$, 
we define $\lambda_\sigma(\uxi)$ (resp.\ $\lambdahat_\sigma(\uxi)$) 
as the supremum of all real numbers $\lambda > 0$ such that the inequalities
\begin{equation}
\label{exp:eq1}
 |x_0|\le Q, \quad 
 |x_0\xi_1-x_1|\le Q^{-\lambda+1} \et  
 |x_0\xi_2-x_2|\le Q^{-\sigma\lambda+1}
\end{equation}
admit a non-zero solution $\ux=(x_0,x_1,x_2)\in\bZ^3$ for arbitrarily large values of 
$Q\ge 1$ (resp.\ for all sufficiently large values of $Q\ge 1$).  These are essentially the usual 
weighted exponents of approximation to $\uxi$, as in \cite{G2022} for example, except 
for the additive constant $1$ in the exponents of $Q$.   This allows us 
to define equivalently $\lambda_\sigma(\uxi)$ as the supremum of all 
$\lambda > 0$ such that 
\[
  |\xi_1-x_1/x_0|\le x_0^{-\lambda} \et  
 |\xi_2-x_2/x_0|\le x_0^{-\sigma\lambda}
\]
for infinitely many $\ux=(x_0,x_1,x_2)\in\bZ^3$ with $x_0>0$.  This modification
also makes the exponents easier to handle via the following result.

\begin{lemma}
\label{exp:lemma}
For  $\uxi$ and $\sigma$ as above we have
\[
 \lambda_\sigma(\uxi)^{-1}=\liminf_{q\to\infty} \frac{L_{\uxi,1}(q,\sigma q)}{q}
 \et
 \lambdahat_\sigma(\uxi)^{-1}=\limsup_{q\to\infty} \frac{L_{\uxi,1}(q,\sigma q)}{q}.
\] 
\end{lemma} 

\begin{proof}
For $Q=e^q$ with $q\ge 0$, the condition that \eqref{exp:eq1} admits a 
non-zero solution in $\bZ^3$ is equivalent to asking that 
$L_{\uxi,1}(\lambda q,\sigma\lambda q)\le q$.  Thus, $\lambda_\sigma(\uxi)$ 
(resp.\ $\lambdahat_\sigma(\uxi)$) is also the supremum of all $\lambda > 0$ 
such that $L_{\uxi,1}(q,\sigma q)/q < 1/\lambda$ for arbitrarily large values of 
$q>0$ (resp.\ for all sufficiently large $q>0$), and the formulas follow. 
\end{proof}

We conclude this paper with the following computation.

\begin{theorem}
\label{exp:thm}
Let $\uxi=(1,\xi,\xi^2)$ where $\xi\in\cE_m$ for some positive integer $m$.
Then, we have $\lambda_\sigma(\uxi)=2$ for each $\sigma\in[0,1]$, and
\begin{equation}
\label{exp:thm:eq}
 \lambdahat_\sigma(\uxi) =
  \begin{cases}
    \gamma 
        &\text{if\/ $1-\gamma^{-4}\le \sigma\le 1$,} \\
    (1+\gamma^{-2})/\sigma 
        &\text{if\/ $5/(2\gamma^2+1)\le \sigma\le 1-\gamma^{-4}$,}\\
    (2\gamma^2+1)/(\gamma^2+1) 
       &\text{if\/ $1-\gamma^{-3}\le \sigma\le 5/(2\gamma^2+1)$.}
 \end{cases}
\end{equation}
\end{theorem}

\begin{proof}
Let $\sigma\in [0,1]$.  Since $\xi$ is badly approximable, \eqref{exp:eq1} has no 
non-zero solution $\ux\in\bZ^3$ for $\xi_1=\xi$ and $\lambda>2$ when $Q$ is
large enough.  On the hand, since $\xi\in\cE_m$, conditions (E1)--(E3) of 
section~\ref{main:ssec:extremal} apply and yield $\norm{(\xi,-1)\ux_i}\asymp
\norm{\ux_i}^{-1}$ for each $i\ge 1$.  Thus,  for the current point $\uxi$
and for any given $\lambda$ with $0< \lambda<2$, the point $\ux=\ux_i$ 
satisfies \eqref{exp:eq1} with $Q=\norm{\ux_i}$ for each large enough $i$.  
This shows that $\lambda_\sigma(\uxi)=2$.

Choosing $\rho$ as in Theorem~\ref{main:extremal:thm}, we find by 
Lemma~\ref{exp:lemma}
\[
 \lambdahat_\sigma(\uxi)^{-1}
  =\limsup_{q\to\infty} \frac{\rho P_1(\rho^{-1}q,\sigma\rho^{-1}q)}{q}
  =\limsup_{q\to\infty} \frac{P_1(q,\sigma q)}{q},
\]
thus
\[
 \lambdahat_\sigma(\uxi)^{-1}
  =\limsup_{k\to\infty} \phitop_k(\sigma)
  \quad\text{where}\quad
  \phitop_k(\sigma)=\max\{P_1(q,\sigma q)/q\,;\, 2F_k\le q\le 2F_{k+1}\}.
\]
Let $k\ge 5$ be an arbitrarily large integer.  To estimate $\phitop_k(\sigma)$, we set
\[
 \cR=\cell(w_k, w_kw_{k-4}, w_kw_{k-2})\in S_{k-2}
 \et
 \cR'=\cell(w_k, w_kw_{k-2}, w_{k+1})\in S_{k-1}.
\]
We denote by $\cR_1$, $\cR_2$ and $\cR_3$ the subsets of $\cR$ made of 
the points $\uq=(q_1,q_2)$ where $P_1(\uq)$ is given respectively by
$q_1-F_k$, $q_2-F_{k-2}$, and $F_k+F_{k-2}$.  As explained right after 
Proposition~\ref{properties:prop1}, these are admissible polygons with a 
common vertex $\ur$, and they form a partition of $\cR$, as illustrated 
in Figure~\ref{exp:fig}. 
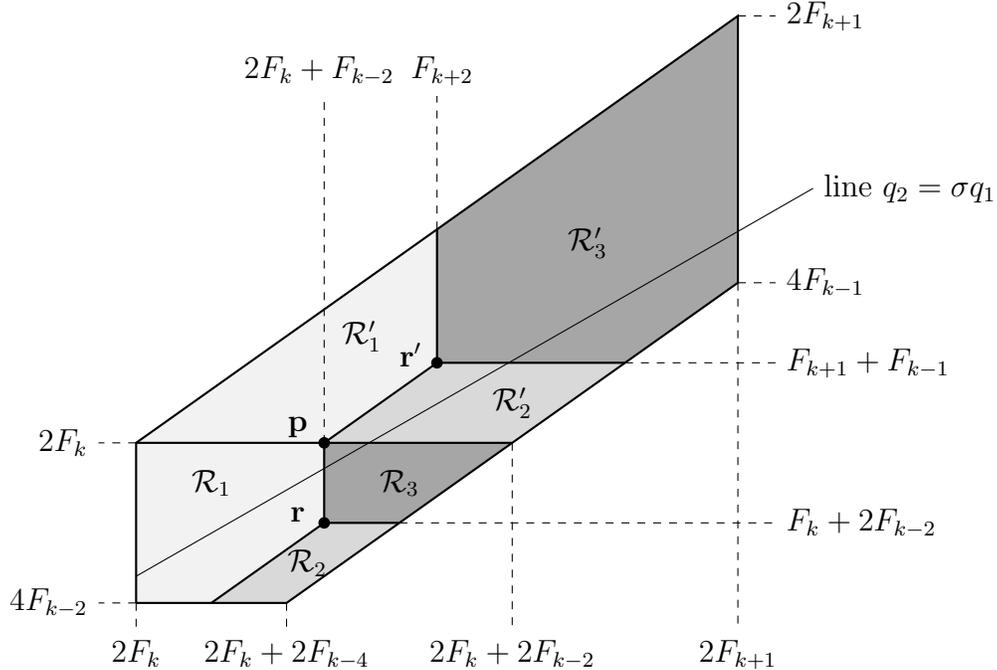
\begin{figure}[ht]
\begin{tikzpicture}[xscale=0.5, yscale=0.355]
\fill[opacity=0.1,gray] (26,20)--(28,20)--(31,23)--(31,26)--(34,29)--(34,34)--(26,26)--(26,20);
\fill[opacity=0.7,gray] (34,29)--(34,34)--(42,42)--(42,32)--(39,29)--(34,29);
\fill[opacity=0.7,gray] (31,23)--(31,26)--(36,26)--(33,23)--(31,23);
\fill[opacity=0.3,gray] (31,23)--(33,23)--(30,20)--(28,20)--(31,23);
\fill[opacity=0.3,gray] (34,29)--(39,29)--(36,26)--(31,26)--(34,29);
\draw[-, thick] (26,20)--(30,20)--(42,32)--(42,42)--(26,26)--(26,20);
\draw[-, thick] (26,26)--(36,26);
\draw[-, thick] (28,20)--(31,23)--(31,26)--(34,29)--(34,34);
\draw[-, thick] (31,23)--(33,23);
\draw[-, thick] (34,29)--(39,29);
\node[draw,circle,inner sep=1.4pt, fill] at (31,23){};
\node at (30.3,23.3){$\ur$};
\node[draw,circle,inner sep=1.4pt, fill] at (34,29){};
\node at (33.3,29.3){$\ur'$};
\node[draw,circle,inner sep=1.4pt, fill] at (31,26){};
\node at (30.3,26.6){$\up$};
\draw[dashed] (26,20)--(26,19) node[below]{$2F_k$};
\draw[dashed] (30,20)--(30,19) node[below]{$2F_k+2F_{k-4}$};
\draw[dashed] (36,26)--(36,19) node[below]{$2F_k+2F_{k-2}$};
\draw[dashed] (42,32)--(42,19) node[below]{$2F_{k+1}$};
\draw[dashed] (31,26)--(31,39) node[above]{$2F_k+F_{k-2}\ $};
\draw[dashed] (26,20)--(25,20) node[left]{$4F_{k-2}$};
\draw[dashed] (26,26)--(25,26) node[left]{$2F_k$};
\draw[dashed] (34,34)--(34,39) node[above]{$\ F_{k+2}$};
\draw[dashed] (42,42)--(43,42) node[right]{$2F_{k+1}$};
\draw[dashed] (42,32)--(43,32) node[right]{$4F_{k-1}$};
\draw[dashed] (39,29)--(43,29) node[right]{$F_{k+1}+F_{k-1}$};
\draw[dashed] (33,23)--(43,23) node[right]{$F_{k}+2F_{k-2}$};
\draw[-] (26,21)--(44,44*21/26) node[right]{line $q_2=\sigma q_1$};
\node at (28,24.5) {$\cR_1$};
\node at (33,24.5) {$\cR_3$};
\node at (30.55,21.5) {$\cR_2$};
\node at (32,30) {$\cR'_1$};
\node at (36,27.5) {$\cR'_2$};
\node at (38,33.5) {$\cR'_3$};
\end{tikzpicture}
\caption{$P_1$ on $\cell(w_k,w_kw_{k-4},w_kw_{k-2}) \cup \cell(w_k,w_kw_{k-2},w_{k+1})$}
\label{exp:fig}
\end{figure}
Similarly, we denote respectively by $\cR'_1$, $\cR'_2$ and $\cR'_3$ the subsets 
of $\cR'$ where $P_1(\uq)$ is given respectively by $q_1-F_k$, $q_2-F_{k-1}$
and $F_{k+1}$, and we denote by $\ur'$ the common vertex of these 
polygons.  We also define $\up$ as the mid-point of the top side of
$\cR$.  Upon setting $\mu(\uq)=q_2/q_1$ for each 
$\uq=(q_1,q_2)\in\cA(\epsilon)$, we find
\[
 \mu(\ur')=\frac{F_{k+1}+F_{k-1}}{F_{k+2}} >
 \mu(\up)=\frac{2F_k}{2F_k+F_{k-2}} > 
 \mu:=\max\Big\{\frac{2F_{k-2}}{F_k},\frac{2F_{k-1}}{F_{k+1}}\Big\} >
 \mu(\ur)
\]
for each large enough $k$.  When $\sigma\ge \mu$, the line $q_2=\sigma q_1$
meets the left vertical side of $\cR$ and the right vertical side of $\cR'$.
In between, it remains in $\cR\cup\cR'$.
As this line crosses the regions $\cR_1$, $\cR'_1$ or $\cR'_2$, the ratio 
$P_1(q,\sigma q)/q$ increases.  As it crosses $\cR_3$ or $\cR'_3$,  the same 
ratio decreases.  So, this ratio is maximal at the point where the line meets
the left vertical side of $\cR_3$ or the left vertical side of $\cR'_3$ or the 
horizontal bottom side of $\cR'_3$.  As $P_1$ is constant equal to $F_k+F_{k-2}$
on $\cR_3$ and constant equal to $F_{k+1}$ on $\cR'_3$, we deduce that
\[
 \phitop_k(\sigma) 
  =\begin{cases}
   \disp\frac{F_{k+1}}{F_{k+2}} 
        &\text{if\/ $\mu(\ur')\le \sigma\le 1$,}\\[10pt]
   \disp\frac{F_{k+1}\sigma}{F_{k+1}+F_{k-1}} 
        &\text{if\/ $\mu(\up)\le \sigma\le \mu(\ur')$,}\\[10pt]
   \disp\max\Big\{ \frac{F_k+F_{k-2}}{2F_k+F_{k-2}}, 
                              \frac{F_{k+1}\sigma}{F_{k+1}+F_{k-1}} \Big\}
        &\text{if\/ $\mu\le \sigma \le \mu(\up)$.}
   \end{cases}
\]
Letting $k$ go to infinity, we deduce that 
\[
 \lambdahat_\sigma(\uxi)^{-1} 
  =\begin{cases}
   \disp \frac{1}{\gamma} 
        &\disp \text{if \ $\frac{\gamma^2+1}{\gamma^3}\le \sigma\le 1$,}\\[10 pt]
   \disp \frac{\gamma^2\sigma}{\gamma^2+1} 
        &\disp \text{if \ $\frac{2\gamma^2}{2\gamma^2+1}\le \sigma
                          \le \frac{\gamma^2+1}{\gamma^3}$,}\\[10pt]
   \disp \max\Big\{ \frac{\gamma^2+1}{2\gamma^2+1}, 
                     \frac{\gamma^2\sigma}{\gamma^2+1} \Big\}
        &\disp \text{if \ $\frac{2}{\gamma^2} \le \sigma 
                           \le \frac{2\gamma^2}{2\gamma^2+1}$,}
   \end{cases}
\]
which rewrites as \eqref{exp:thm:eq}.
\end{proof}

With additional work, the same method enables one to calculate 
$\lambdahat_\sigma(\uxi)$ for any given $\sigma\in\left]1/2,1\right]$.  
For $\sigma\in\left]0,1/2\right]$, Lemma~\ref{tool:lemma3} yields trivially
$\lambdahat_\sigma(\uxi)=\sigma^{-1}$ (take $\ux=(0,0,1)$ in \eqref{exp:eq1}).

\begin{acknowlegments}
The content of this paper was presented at the conference ``Diophantine approximation and related fields'' at the University of York in June 2025.  The author thanks Victor Beresnevich and all organizers for their invitation.  He also thanks Anthony Po\"els and Nicolas de Saxc\'e for helpful comments.
\end{acknowlegments}

\end{document}